\documentclass[10pt,a4paper]{article}
\usepackage{amsthm}
\usepackage{amsmath}
\usepackage{amsfonts}
\usepackage{amssymb}
\usepackage{graphicx}
\usepackage[utf8]{inputenc}
\usepackage{verbatim}
\usepackage{comment}
\usepackage{subcaption}
\usepackage{float}
\usepackage{mathtools}
\usepackage[scaled=0.90]{helvet}
\usepackage{hyperref}
\usepackage{geometry}
\usepackage{pgf,tikz,pgfplots}
\pgfplotsset{compat=1.15}
\usetikzlibrary{calc, arrows, shapes, positioning, patterns, angles, quotes, intersections, through, backgrounds}\usetikzlibrary{decorations.markings}

\makeatletter
\let\save@mathaccent\mathaccent
\newcommand*\if@single[3]{%
  \setbox0\hbox{${\mathaccent"0362{#1}}^H$}%
  \setbox2\hbox{${\mathaccent"0362{\kern0pt#1}}^H$}%
  \ifdim\ht0=\ht2 #3\else #2\fi
  }
\newcommand*\rel@kern[1]{\kern#1\dimexpr\macc@kerna}
\newcommand*\widebar[1]{\@ifnextchar^{{\wide@bar{#1}{0}}}{\wide@bar{#1}{1}}}
\newcommand*\wide@bar[2]{\if@single{#1}{\wide@bar@{#1}{#2}{1}}{\wide@bar@{#1}{#2}{2}}}
\newcommand*\wide@bar@[3]{%
  \begingroup
  \def\mathaccent##1##2{%
    \let\mathaccent\save@mathaccent
    \if#32 \let\macc@nucleus\first@char \fi
    \setbox\z@\hbox{$\macc@style{\macc@nucleus}_{}$}%
    \setbox\tw@\hbox{$\macc@style{\macc@nucleus}{}_{}$}%
    \dimen@\wd\tw@
    \advance\dimen@-\wd\z@
    \divide\dimen@ 3
    \@tempdima\wd\tw@
    \advance\@tempdima-\scriptspace
    \divide\@tempdima 10
    \advance\dimen@-\@tempdima
    \ifdim\dimen@>\z@ \dimen@0pt\fi
    \rel@kern{0.6}\kern-\dimen@
    \if#31
      \overline{\rel@kern{-0.6}\kern\dimen@\macc@nucleus\rel@kern{0.4}\kern\dimen@}%
      \advance\dimen@0.4\dimexpr\macc@kerna
      \let\final@kern#2%
      \ifdim\dimen@<\z@ \let\final@kern1\fi
      \if\final@kern1 \kern-\dimen@\fi
    \else
      \overline{\rel@kern{-0.6}\kern\dimen@#1}%
    \fi
  }%
  \macc@depth\@ne
  \let\math@bgroup\@empty \let\math@egroup\macc@set@skewchar
  \mathsurround\z@ \frozen@everymath{\mathgroup\macc@group\relax}%
  \macc@set@skewchar\relax
  \let\mathaccentV\macc@nested@a
  \if#31
    \macc@nested@a\relax111{#1}%
  \else
    \def\gobble@till@marker##1\endmarker{}%
    \futurelet\first@char\gobble@till@marker#1\endmarker
    \ifcat\noexpand\first@char A\else
      \def\first@char{}%
    \fi
    \macc@nested@a\relax111{\first@char}%
  \fi
  \endgroup
}
\makeatother

\usepackage{capt-of}

\newcommand*{\N}{\mathbb{N}}

\newcommand*{\R}{\mathbb{R}}

\newcommand*{\tX}{\widetilde{X}}
\newcommand*{\td}{\widetilde{d}}
\newcommand*{\rar}{\rightarrow}

\newcommand*{\lpls}{(X, d, \ll, \leq, \tau)}

\newcommand*{\tll}{\mathrel{\widetilde{\ll}}}
\newcommand*{\tleq}{\mathrel{\widetilde{\leq}}}
\newcommand*{\tx}{\tilde{x}}
\newcommand*{\ty}{\tilde{y}}
\newcommand*{\ttau}{\widetilde{\tau}}
\newcommand*{\btau}{\bar{\tau}}

\newcommand*{\mad}{\measuredangle}

\newcommand*{\bT}{\widebar{T}}

\newcommand*{\bx}{\bar{x}}
\newcommand*{\by}{\bar{y}}
\newcommand*{\bz}{\bar{z}}

\newcommand*{\bp}{\bar{p}}
\newcommand*{\bq}{\bar{q}}
\newcommand*{\br}{\bar{r}}
\newcommand*{\ba}{\bar{a}}
\newcommand*{\bb}{\bar{b}}

\DeclareMathOperator{\arcosh}{arcosh}

\DeclareMathOperator{\sgn}{sgn}
\DeclareMathOperator{\dS}{dS}

\swapnumbers

\newtheorem{thm}{Theorem}[subsection]
\newtheorem{prop}[thm]{Proposition}
\newtheorem{cor}[thm]{Corollary}
\newtheorem{lem}[thm]{Lemma}
\newtheorem{defin}[thm]{Definition}

\theoremstyle{definition}
\newtheorem{ex}[thm]{Example}

\newtheorem{rem}[thm]{Remark}
\numberwithin{equation}{subsection}

\theoremstyle{definition} 

\newcommand{\thistheoremnam}{}
\newtheorem*{genericthm*}{\thistheoremnam}
\newenvironment{chapt*}[1]
  {\renewcommand{\thistheoremnam}{#1}%
   \begin{genericthm*}}
  {\end{genericthm*}}

\title{Gluing constructions for Lorentzian length spaces}
\author{Tobias Beran\footnote{\href{mailto:tobias.beran@univie.ac.at}{tobias.beran@univie.ac.at}, Faculty of Mathematics, University of Vienna, Austria} \ and Felix Rott\footnote{\href{mailto:felix.rott@univie.ac.at}{felix.rott@univie.ac.at}, Faculty of Mathematics, University of Vienna, Austria}}
\date{}

\begin{document}

\maketitle
\begin{abstract}
We introduce an analogue to the amalgamation of metric spaces into the setting of Lorentzian pre-length spaces. This provides a very general process of constructing new spaces out of old ones. The main application in this work is an analogue of the gluing theorem of Reshetnyak for CAT($k$) spaces, which roughly states that gluing is compatible with upper curvature bounds. 
We formulate the theorem in terms of (strongly causal) spacetimes viewed as Lorentzian length spaces. \\


\end{abstract}

\tableofcontents

\section{Introduction}
The theory of Lorentzian length spaces, introduced in \cite{KS18}, is a new approach to developing a synthetic description of Lorentzian geometry without relying on any differential geometric machinery. It is very much inspired by the relationship between metric geometry and Riemannian geometry, where in particular the theory of length spaces has led to fundamental contributions and essentially has given rise to a purely metric and synthetic point of view of Riemannian manifolds. \\

Lorentzian length spaces appear to be a very promising approach in this direction and are on the way to becoming an independent field of research, increasingly attracting many established researchers from Lorentzian geometry and general relativity. 
There have been a variety of interesting results concerning the advancement of Lorentzian length spaces of which we want to mention a few.
\begin{itemize}
\item \cite{GKS19} introduces a notion of (in)extendibility for Lorentzian length spaces.
\item \cite{AGKS19}, via generalized cones, introduces an analogue to warped products into the setting of Lorentzian length spaces.
\item \cite{CM20} introduces optimal transport methods in Lorentzian length spaces, defines timelike Ricci curvature bounds via suitable entropy conditions and gives applications to general relativity (synthetic singularity theorems).
\item \cite{ACS20} further develops the causal ladder for Lorentzian length spaces.
\item \cite{KS21} examines the null distance in Lorentzian length spaces (which was first introduced in \cite{SV16} for manifolds) and in turn studies Gromov-Hausdorff convergence, establishing first compatibility results with respect to curvature bounds. 
\item \cite{BGH21} studies (the existence of) time functions on Lorentzian length spaces.
\item \cite{MS21} defines an analogue to Hausdorff measure on Lorentzian length spaces.
\end{itemize}
\subsection{Motivation and summary}
Currently, the majority of research in Lorentzian length spaces is concentrated around direct applications to general relativity and only few works result from a purely metric motivation. 
Indeed, it still seems that many fundamental concepts and constructions from metric geometry have not yet been fully incorporated or are outright missing from the Lorentzian theory. \\

The main goal of this work is to adapt some of these missing concepts from metric geometry to the Lorentzian setting and in this way contribute to making it an equally applicable and impactful synthetic analogue to the metric theory of length spaces. Leading experts in the field of metric geometry suggest that an idea similar to the amalgamation of metric spaces is essential in this process. 
Indeed, the amalgamation of metric spaces is \emph{the} fundamental construction for producing new spaces from old ones and thus showcases a significant advantage of metric spaces compared to (Riemannian) manifolds, where gluing is in general only possible along isometric/diffeomorphic boundaries, if at all. Instead, one can usually only consider Cartesian products or submanifolds, both of which offer much less flexibility. \\

One of the key results concerning gluing in the metric world is the gluing theorem of Reshetnyak: it states that the amalgamation of metric spaces which satisfy an upper curvature bound, so-called CAT($k$) spaces, also satisfies the same upper curvature bound. 
Metric gluing has also found applications in the theory of semi-dispersing billiards, cf.\ \cite{AKP19,BFK98a,BFK98b}, as well as in geometric group theory, cf.\ \cite{BH99}. \\

A gluing process for Lorentzian pre-length spaces turns out to be a more delicate matter than the corresponding process for metric spaces since, roughly speaking, there is much more compatibility one has to respect. 
In other words, a metric space only consists of a set with a distance function while a Lorentzian pre-length space is both a causal space and a metric space and moreover both have to behave well with respect to the time separation function.
Our first task is to translate the metric amalgamation into the Lorentzian setting. In the metric case, it consists of two steps: first forming the disjoint union and then considering the quotient semi-metric with respect to the identifying equivalence relation. The disjoint union can be easily adapted but a ``quotient time separation'' needs to be treated a bit more carefully. \\

We continue with the preparations necessary for a Lorentzian analogue of the Reshetnyak gluing theorem, which is the central part of this work. Most important for this goal is to establish a gluing lemma for triangles in the sense of \cite[Lemma II.4.10]{BH99}. This turns out to be a quite technically demanding task. 
Even worse, without a solid concept of spacelike distance in Lorentzian pre-length spaces, there is no chance to achieve a reasonably general version of the gluing theorem. It does, however, work out when considering manifolds as Lorentzian pre-length spaces, where spacelike distances are well known: this is the content of the last chapter. This is also where the main result of this paper is formulated, namely:
\begin{chapt*}{5.2.1 Theorem}[Reshetnyak's gluing theorem, Lorentzian version]
Let $(X_1,g_1)$ and $(X_2,g_2)$ be two smooth and strongly causal spacetimes with $\dim(X_1)=:n \geq m := \dim(X_2)$. Let $A_1$ and $A_2$ be two closed non-timelike locally isolating subsets of $X_1$ and $X_2$, respectively. 
Let $f:A_1 \rar A_2$ be a $\tau$-preserving and $\leq$-preserving locally bi-Lipschitz homeomorphism which locally preserves the signed distance. Suppose $A_1$ and $A_2$ are convex in the sense of Remark \ref{assumptions on A}(iii).
Suppose $X_1$ and $X_2$ have (sectional) curvature bounded above by $K \in \R$ in the sense of \cite{AB08}. Then the Lorentzian amalgamation $X:= X_1 \sqcup_A X_2$
is a Lorentzian pre-length space with timelike curvature bounded above by $K$.
\end{chapt*}
\subsection{Outlook}
We conclude the introduction by briefly discussing possible applications of gluing constructions in the (synthetic) Lorentzian setting.
\begin{itemize}
\item The ``causal inheritance'': Many steps of the causal ladder for spacetimes have been translated into the synthetic setting, cf.\ \cite{ACS20}. Given two Lorentzian pre-length spaces that are both situated somewhere on the causal ladder, can the same be said about their amalgamation? If not, are there additional properties that would guarantee the preservation of this property?
\item The compatibility of the amalgamation and Gromov-Hausdorff convergence of Lorentzian length spaces with respect to the null distance:  convergence of Lorentzian length spaces has been studied in \cite{KS21}. Given two sequences of Lorentzian length spaces that converge each to a Lorentzian length space, can the same be said about the sequence of the respective amalgamations?
\item An analogue to the collision theorem: concerning the theory of semi-dispersing billiards, the collision theorem is a particularly nice application of gluing in the metric world, cf.\ \cite[Theorem 2.6.1]{AKP19}. In the Lorentzian case, this could be useful for investigating particle collisions in general relativity.
\item Globalization: the metric version of the gluing lemma is used to globalize upper curvature bounds, see \cite[Proposition II.4.9 \& Lemma II.4.10]{BH99}. A Lorentzian version of such a result would certainly be very interesting and might be possible with similar methods.
\item General relativity: it is expected that gluing constructions can also be directly applied in various topics from general relativity. Examples include: extending a spacetime (or Lorentzian length space) by gluing, cosmic censorship and gluing at singularities, or matching of spacetimes and impulsive gravitational waves.
\end{itemize}

\section{Preliminaries}
By a spacetime $(M,g)$ we mean a smooth manifold $M$ with a Lorentzian metric $g$ and a time orientation. Requiring the spacetime to be $C^k$ means the metric $g$ is $C^k$.
We 
denote by $\eta$ the ordinary Minkowski metric on $\R^n$.
We write $I(x,z):=I^+(x) \cap I^-(z)=\{ y \mid x \ll y \ll z \}$ for timelike diamonds and $J(x,z)$ for causal diamonds.
By a hinge $(\alpha,\beta)$ we mean a configuration of two (timelike) geodesics $\alpha$ and $\beta$ and the included (hyperbolic) angle, usually denoted by $\omega$.
For basic information regarding Lorentzian pre-length spaces see \cite{KS18,Ber20}. For basic information regarding the amalgamation and its compatibility with curvature conditions in the metric case see \cite{BBI,BH99,Rot20}. 
We will anyways present a very short recap of the most fundamental concepts concerning Lorentzian pre-length spaces and we will briefly describe the amalgamation in the metric picture.
\subsection{A brief introduction to Lorentzian pre-length spaces}
Simply put, a Lorentzian pre-length space encodes certain fundamental properties of a Lorentzian manifold while completely ignoring others. The focus lies on the causality relations and the time separation function, while the Lorentzian metric and the general manifold structure are discarded entirely. Compare this to metric geometry, where length spaces serve as a very useful generalization of Riemannian manifolds.
\begin{defin}[Lorentzian pre-length space]
A tuple $\lpls$ is called a Lorentzian pre-length space if it satisfies the following:
\begin{itemize}
\item[(i)]$(X,\ll,\leq)$ is a causal space, i.e., $\leq$ is a reflexive and transitive relation on $X$ and $\ll$ is a transitive relation on $X$ contained in $\leq$.
\item[(ii)] $\tau: X \times X \rar [0,\infty]$ is lower semi-continuous with respect to the metric $d$.
\item[(iii)] $\tau$ respects the causal structure in the following way: $\tau$ satisfies the reverse triangle inequality for $\leq$-related points and is compatible with $\ll$ in the sense that $\tau(a,b) >0 \iff a \ll b$.
\end{itemize}
\end{defin}
Note that due to \cite[Example 2.11]{KS18} any smooth spacetime is a Lorentzian pre-length space (where the distance metric is induced by some (complete) Riemannian background metric). Moreover, any continuous causally plain metric on a spacetime yields a Lorentzian pre-length space, cf.\ \cite[Proposition 5.8]{KS18}
\begin{defin}[Causal/timelike curves]
Let $\lpls$ be a Lorentzian pre-length space. 
\begin{itemize}
\item[(i)] A locally Lipschitz curve $\gamma:[a,b] \to X$ is called-future directed causal (respectively timelike), if $\gamma(s) \leq \gamma(t)$ (respectively $\gamma(s) \ll \gamma(t)$) for all $s,t \in [a,b], s < t$.
Past-directed curves are defined analogously.
Unless explicitly stated otherwise, we assume all causal curves to be future-directed.
\item[(ii)] The $\tau$-length of a causal curve $\gamma$ is given as
\begin{equation}
L_{\tau}(\gamma):= \inf \{ \sum_{i=0}^n \tau(\gamma(t_i),\gamma(t_{i+1}) \mid 
a=t_0 < t_1 < \ldots < t_n=b, n \in \N \}.
\end{equation}
If $\gamma(a)=x,\gamma(b)=y$ and $L_{\tau}(\gamma)=\tau(x,y)$ we say that $\gamma$ is $\tau$-realizing and we call (the image of) such a curve a geodesic segment.
\end{itemize}
\end{defin}
The main difference between a Lorentzian length space and a Lorentzian pre-length space is in spirit the same as between a length space and a metric space. That is, the time separation function of a Lorentzian length space is intrinsic in the sense that it is given by the (supremum of the) lengths of connecting causal curves. There are also some additional technical assumptions on a Lorentzian length space resembling the existence of small ``convex'' neighbourhoods. As we will mainly work with Lorentzian pre-length spaces, we only refer to the definition, see \cite[Definition 3.22]{KS18}. \\

The final ingredient we will need is the description of curvature bounds. As in the metric world, triangle comparison replaces the concept of sectional curvature bounds. We denote by $M_K$ the Lorentzian model space of constant sectional curvature $K$, cf.\ \cite[Definition 4.5]{KS18}. That is, $M_K$ is either an appropriately scaled version of de Sitter- or anti de Sitter space or the Minkowski plane. 
We may denote the Lorentzian metric on $M_K$ by $\langle \cdot , \cdot \rangle$.
Unless explicitly stated otherwise, we assume all mentioned triangles to satisfy the appropriate size bounds for $M_K$, cf.\ \cite[Lemma 2.1]{AB08} or \cite[Lemma 4.6]{KS18}.
\begin{defin}[Timelike curvature bounds]
\label{tl curv bounds}
A Lorentzian pre-length space $\lpls$ has timelike curvature bounded below (respectively above) by $K \in \R$ if every point in $X$ has a neighbourhood $U$, called a comparison neighbourhood, which satisfies the following:
\begin{itemize}
\item[(i)]$\tau|_{U \times U}$ is finite and continuous.
\item[(ii)] For all $x,y \in U$ with $x \ll y$ there exists a $\tau$-realizing curve entirely contained in $U$.
\item[(iii)] Let $\Delta(x,y,z)$ be a timelike triangle in $U$, i.e., $x \ll y \ll z$ and we have $\tau$-realizing curves connecting these points pairwise. Let $\Delta(\bx,\by,\bz)$ be its comparison triangle in $M_K$. Then for all $p,q \in \Delta(x,y,z)$ and corresponding comparison points $\bp, \bq$ in the comparison triangle we have 
$\tau(p,q) \leq \btau(\bp,\bq)$ (respectively $\tau(p,q) \geq \btau(\bp,\bq))$.
\end{itemize}
\end{defin}
\subsection{A brief introduction to metric amalgamation and the gluing theorem}
Here, we collect all metric prerequisites needed for a Lorentzian gluing construction, following \cite{BH99,BBI}.
\begin{defin}[Disjoint union metric]
Let $(X_i, d_i)_{i \in I}$ be a family of metric spaces. Let $X:= \sqcup_{i \in I} X_i$ be the disjoint union. Then
\begin{equation}
\label{disjoint union metric}
d(x,y):=
\begin{cases} d_i(x,y) & x,y \in X_i \\
\infty & \, \text{else}.
\end{cases}
\end{equation}
defines a metric on $X$, called the disjoint union metric.
\end{defin}
\begin{defin}[Quotient semi-metric]
Let $(X,d)$ be a metric space and let $\sim$ be an equivalence relation on $X$. The quotient semi-metric with respect to $\sim$ is defined as
\begin{equation}
\td([x],[y]):=\inf \{ \sum_{i=1}^n d(x_i,y_i) \mid x \sim x_1, x_{i+1} \sim y_i, y_n \sim y, n \in N \}.
\end{equation}
\end{defin}
\begin{defin}[Amalgamation]
Let $(X_i,d_i)_{i \in I}$ be a family of metric spaces. Let $(A_i)_{i \in I}$ be a family of closed subspaces 
each of which is isometric to some metric space $A$ via the isometry
$f_i : A \to A_i$. Equip the disjoint union $X:= \sqcup_{i \in I} X_i$ with the metric $d$ from above. On $X$ let $\sim$ be the equivalence relation generated by the condition $f_i(a) \sim f_j(a)$ for all $i,j \in I, a \in A$.
Then the quotient of $X$ equipped with the quotient semi-metric with respect to $\sim$ is called the amalgamation of the family $(X_i)_{i \in I}$ with respect to $A$ and is denoted by $\sqcup_A X_i$, i.e., $\sqcup_A X_i = (X/\sim, \td)$.
\end{defin}
Finally, we mention the gluing theorem of Reshetnyak. In the following formulation, it is in fact possible to omit the assumption of each $X_i$ being proper, but the proof then gets significantly more difficult. For a proof, see e.g.\ \cite[Theorem II.11.1 \& Theorem II.11.3]{BH99}.
\begin{thm}[Reshetnyak]
Let $(X_i,d_i)_{i \in I}$ be a family of proper CAT($k)$ spaces. Let $(A_i)_{i \in I}$ be a family of closed convex and complete subspaces each of which is isometric to some metric space $A$. Then the amalgamation $\sqcup_A X_i$ is a CAT($k$) space.
\qed
\end{thm}
\section{Lorentzian structure on a quotient}
In this section, we introduce the amalgamation construction for Lorentzian pre-length spaces. To this end we first have to discuss more elementary aspects of gluing.
\subsection{Basic gluing preparations}
We begin introducing a Lorentzian structure on the quotient of a Lorentzian pre-length space by adapting the definition of the quotient semi-metric with additional causality assumptions.
\begin{rem}[On notation and conventions I]
As any distance metric and time separation function takes values only in $[0,\infty]$, we set $\sup \emptyset = 0$ and $\inf \emptyset = \infty$ for the sake of convenience.
We will sometimes apply shortcuts commonly used in the theory of metric spaces and just write $X$ for a Lorentzian pre-length space $\lpls$. 
By a subspace $A \subseteq X$ we mean a subset viewed as a Lorentzian pre-length space equipped with the restriction of the original metric, causality relations and time separation. Moreover, we will write $[x,y]$ for a geodesic segment between $x$ and $y$. Either the context or a more detailed description will prevent any ambiguity.
We will usually write $\tX:=X/\sim$ for the (topological) quotient of $X$ with respect to an equivalence relation $\sim$. We will denote the natural projection $x \mapsto [x]$ by $\pi: X \to \tX$. 
\end{rem} 
\begin{defin}[Quotient time separation]
Let $\lpls$ be a Lorentzian pre-length space and let $\sim$ be an equivalence relation on $X$. The quotient time-separation function is defined as $\ttau: \tX \times \tX \rar [0,\infty]$, 
\begin{equation}
\ttau([x],[y]):=\sup \{ \sum_{i=1}^n \tau(x_i,y_i) \mid x \sim x_1 \leq y_1 \sim x_2 \leq y_2 \sim \ldots \sim x_n \leq y_n \sim y, n \in \N \}.
\end{equation}
We call a sequence $(x_1,y_1,\ldots, x_n,y_n)$ as above an $n$-chain from $[x]$ to $[y]$. 
\end{defin}
\begin{rem}[Restricting the set of chains]
\label{better chains}
Note that we can always assume without loss of generality that $x_1=x$ and $y_n=y$. Indeed, suppose $(x_1,y_1,\ldots, x_n,y_n)$ is an $n$-chain from $[x]$ to $[y]$, then $(x,x,x_1,y_1,\ldots, x_n,y_n,y,y)$ is an $(n+2)$-chain with at least the same length. \\

Furthermore, we can always assume that $y_i \neq x_{i+1}$. Otherwise by the reverse triangle inequality for $\tau$ we could replace 
\begin{equation}
\tau(x_i,y_i)+\tau(x_{i+1},y_{i+1})=\tau(x_i,y_i)+\tau(y_i,y_{i+1})\leq \tau(x_i,y_{i+1})
\end{equation}
to obtain a longer chain. 
We say that the relation between $y_i$ and $x_{i+1}$ is nontrivial if $y_i \neq x_{i+1}$ and $y_i \sim x_{i+1}$ .
\end{rem}
We define both causality relations on $\tX$ via the quotient time separation.
\begin{defin}[Quotient causality]
Let $\lpls$ be a Lorentzian pre-length space and let $\sim$ be an equivalence relation on $X$. 
On $\tX$, we define $[x] \tll [y] :\iff \ttau([x],[y]) >0$ and 
$[x] \tleq [y] : \iff \{ \sum_{i=1}^n \tau(x_i,y_i) \mid x \sim x_1 \leq y_1 \sim x_2 \leq y_2 \sim \ldots \sim x_n \leq y_n \sim y, n \in \N \} \neq \emptyset$. 
Especially the causal relation might be better described in words: we have $[x] \tleq [y]$ if and only if there exists a chain from $[x]$ to $[y]$. Regarding the timelike relation, we have $[x] \tll [y]$ if and only if there exists a chain of positive length from $[x]$ to $[y]$.
\end{defin}
A skeptical reader will rightfully claim at this point that without additional assumptions (on e.g.\ $\sim$) this construction is badly behaved or does not yield a Lorentzian pre-length space at all. This is not very surprising, and in some sense, parallels the metric world, where the quotient semi-metric might not be positive definite.
There are, however, some properties of a Lorentzian pre-length space that any quotient satisfies.
\begin{prop}[Quotient causal space]
Let $\lpls$ be a Lorentzian pre-length space and let $\sim$ be an equivalence relation on $X$. Then $(\tX, \tll, \tleq)$ is a causal space. 
\end{prop}
\begin{proof}
The inclusion of $\tll$ in $\tleq$ is clear from the definition. Suppose $[x] \tll [y] \tll [z]$. The concatenation of chains with positive length from $[x]$ to $[y]$ and from $[y]$ to $[z]$, respectively, results in a chain with positive length from $[x]$ to $[z]$. Hence $\ttau([x],[z])>0$ and so $[x] \tll [z]$. 
By the same argument we have that if there exists a chain from $[x]$ to $[y]$ and a chain from $[y]$ to $[z]$ then there exists a chain from $[x]$ to $[z]$. Thus, $[x] \tleq [y] \tleq [z]$ implies $[x] \tleq [z]$. 
The reflexivity of $\tleq$ follows from the reflexivity of $\leq$: $(x,x)$ is a valid 1-chain for any $[x] \in \tX$ and so $[x] \tleq [x]$.
\end{proof}
\begin{prop}[Reverse triangle inequality]
Let $\lpls$ be a Lorentzian pre-length space and let $\sim$ be an equivalence relation on $X$. Then $\ttau$ satisfies the reverse triangle inequality for causally related points, i.e., if $[x] \tleq [y] \tleq [z]$, then $\ttau([x],[z]) \geq \ttau([x],[y]) + \ttau([y],[z])$.
\end{prop}
\begin{proof}
This follows immediately from the definition: for any chain from $[x]$ to $[y]$ and any chain from $[y]$ to $[z]$, their concatenation results in a chain from $[x]$ to $[z]$. Since there might be chains from $[x]$ to $[z]$ without going through $[y]$, $\ttau([x],[z])$ can only get larger. More precisely, given $\varepsilon >0$, let $(x_1,y_1,\ldots,x_n,y_n)$ and $(y_1',z_1,\ldots, y_m',z_m)$ be chains from $[x]$ to $[y]$ and from $[y]$ to $[z]$ with lengths at least $\ttau([x],[y]) - \frac{\varepsilon}{2}$ and $\ttau([y],[z]) - \frac{\varepsilon}{2}$, respectively. Then $(x_1,y_1,\ldots,x_n,y_n,y_1',z_1,\ldots, y_m',z_m)$ is a chain from $[x]$ to $[z]$ with length at least $\ttau([x],[y])+\ttau(y,z) - \varepsilon$ and the claim follows.
\end{proof}
In summary, we obtain the following intuitive properties on any Lorentzian quotient. This can be thought of as the analogue to ``gluing can only shrink distances'' in the metric case.
\begin{cor}[Immediate intuitive properties]
\label{easy properties}
Let $\lpls$ be a Lorentzian pre-length space and $\sim$ an equivalence relation on $X$. The quotient Lorentzian structure always satisfies the following properties for all $[x],[y] \in X$. 
\begin{itemize}
\item[(i)] $\ttau([x],[y]) \geq \tau(x,y)$.
\item[(ii)] $x \ll y \Rightarrow [x] \tll [y]$ and $x \leq y \Rightarrow [x] \tleq [y]$.
\end{itemize}
\qed
\end{cor}
Finally, we observe that the only property that might prevent the quotient of a Lorentzian pre-length space from being a Lorentzian pre-length space itself is that $\ttau$ need not be lower semi-continuous.
The following examples show how the lower semi-continuity of $\ttau$ may fail and how this can be prevented.
\begin{ex}[Showcasing gluing in the Minkowski plane]
\label{gluing example1}
Consider the ordinary Minkowski plane $\R^2_1$.
Identify two spacelike related points $x$ and $y$ as in Figure \ref{lsc tau}. 
Then the resulting space is not a Lorentzian pre-length space since $\ttau$ is not lower semi-continuous. To see this, let $p \in \partial J^-(x) \setminus J^-(y)$ and $q \in I^+ (y) \setminus J^+(x)$. Then $\ttau(p,q)>0$ (red line). But if we choose a sequence $(p_n)_{n \in \N}$ such that $p_n \to p$ and $p_n \notin J^-(x) \cup I^-(q)$ for all $n \in \N$, then $\ttau(p_n,q)=0$.
\begin{figure}
\begin{center}
\begin{tikzpicture}

\begin{scriptsize}
\coordinate [circle, fill=black, inner sep=0.7pt, label=90: {$x$}] (A1) at (1,0);
\coordinate [circle, fill=black, inner sep=0.7pt, label=0: {$y$}] (B1) at (2,0);

\coordinate [circle, fill=black, inner sep=1pt, label=50: {$q$}] (q) at (2,0.5);
\coordinate [circle, fill=black, inner sep=1pt, label=300: {$p$}] (p) at (0.5,-0.5);

\coordinate [label=200: {$p_n$}] (pn) at (0.3,-0.3);

\coordinate (P1) at (0,-1);
\coordinate (P2) at (2,-1);

\end{scriptsize}


\draw [dashed] (P1) -- (A1) -- (P2);

\draw [dotted] (pn) -- (p);


\draw [red] (p)--(A1);
\draw [red] (B1)--(q);
\end{tikzpicture}
\end{center}
\caption{$\ttau$ is not lower semi-continuous.}
\label{lsc tau}
\end{figure}
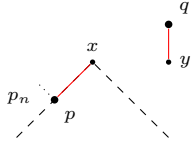
Identifying a closed vertical strip, say $(1,t) \sim (2,t)$ for all $t \in [0,1]$ has exactly the same problem. Doing the identification along an open strip, say $(1,t) \sim (2,t)$ for all $t \in (0,1)$ at first glance seems to eliminate this obstacle. In this case, however, the quotient semi-metric will not be positive definite since for example we would have $\td((1,0),(2,0))=0$. Intuitively, we need the identified sets to be (topologically) closed and at the same time always have timelike related points nearby.
Identifying $(1,t) \sim (2,t)$ for all $t \in \R$ covers both conditions. This actually turns out to be a Lorentzian pre-length space. The lower semi-continuity immediately follows from the more general proof of below.
\end{ex}

\subsection{Amalgamation prerequisites}
In the spirit of metric amalgamation, we first introduce a very easy construction of formally viewing two Lorentzian pre-length spaces as one.
\begin{defin}[Lorentzian disjoint union]
Let $(X_1,d_1,\ll_1,\leq_1,\tau_1)$ and $(X_2,d_2,\ll_2,\leq_2,\tau_2)$ be two Lorentzian pre-length spaces and set $X:= X_1 \sqcup X_2$. Define ${\leq} := {\leq_1} \sqcup {\leq_2}$, i.e., ``${\leq} \subseteq X \times X$'' and $x \leq y :\iff \exists i \in \{1,2\}: x,y \in X_i \wedge x \leq_i y$. Similarly, define ${\ll} := {\ll_1} \sqcup {\ll_2}$. Let $d$ be the disjoint union metric on $X$, cf.\ (\ref{disjoint union metric}).
Define $\tau:X \times X \rar [0,\infty]$ by 
\begin{equation}
\tau(x,y):=
\begin{cases} \tau_i(x,y) & x,y \in X_i \\
0 & \, \text{else}.
\end{cases}
\end{equation}
We call $\lpls$ the Lorentzian disjoint union of $X_1$ and $X_2$.
\end{defin}
\begin{prop}[Disjoint union Lorentzian pre-length space]
Let $(X_1,d_1,\leq_1,\ll_1,\tau_1)$ and $(X_2,d_2,\ll_2,\leq_2,\tau_2)$ be two Lorentzian pre-length spaces. Then the Lorentzian disjoint union $(X,d,\ll, \leq,\tau)$ is a Lorentzian pre-length space.
\end{prop}
\begin{proof}
Clearly, $(X,\ll,\leq)$ is a causal space. The reverse triangle inequality is directly inherited from the respective inequalities in $X_1$ and $X_2$, since causal relation can only occur between points coming from the same spaces. Similarly, the lower semi-continuity of $\tau$ is inherited in this way, since for $x_n \rar x$, say $x \in X_1$, for large enough $n_0$ we have $x_n \in X_1$ for all $n \geq n_0$. Finally, the compatibility of the causal relations with $\tau$ also follows directly from their counterparts in $X_1$ and $X_2$.
\end{proof}
Analogous to the metric case, we define the Lorentzian amalgamation as a quotient of the Lorentzian disjoint union where we identify certain subsets. To ensure that this construction actually results in a Lorentzian pre-length space (i.e., that $\ttau$ is lower semi-continuous) we require the following property of the identified subsets. 
\begin{defin}[Local timelike isolation]
\label{tlisolation}
A subset $A$ of a Lorentzian pre-length space $(X,d, \linebreak \ll, \leq,\tau)$ is said to be non-future locally isolating if for all $a \in A$ with $I^+(a) \neq \emptyset$ and for all neighbourhoods $U_a \subseteq A$ of $a$ there exists $b_+ \in U_a$ such that $a \ll b_+$. Similarly, we define a non-past locally isolating set. We say $A$ is non-timelike locally isolating if it satisfies both properties.
\end{defin}
\begin{rem}[Examples and comments on local timelike isolation]
\label{tl isolation comment}
Clearly, the open image of a timelike curve is non-timelike locally isolating.
Furthermore, for any Lorentzian length space $X$, the set $X$ is non-timelike locally isolating by the sequence lemma, cf.\ \cite[Lemma 2.18]{ACS20}. But note that of course a subset of a Lorentzian length space is not a Lorentzian length space in general (with the restricted structure).
We make the additional assumption of $I^{\pm}(a) \neq \emptyset$ to also allow gluing of spaces with ``future/past boundary''. For example, consider a closed rectangle in the Minkowski plane as a Lorentzian pre-length space and identify two vertical line segments (which are not causally related at all). Then the boundary points of these segments would fail to have the non-isolating property introduced above, but only because there is nothing in the future (respectively in the past) of these points to begin with. In this case also the counterexample of Example \ref{gluing example1} fails: the quotient time separation of a sequence approaching the past of the future boundary point of one segment cannot have a positive value in the limit since on the other segment this is also a future boundary point with empty future.
\end{rem}
In the metric case, amalgamation usually occurs along maps which at least locally ``preserve structure'', i.e., local isometries.
This turns out to be not necessary to define our gluing process. So fore the sake of generality, we will formulate the Lorentzian amalgamation with as few assumptions as possible.
Not surprisingly, this may be badly behaved, which is why we will almost exclusively work with additional assumptions. Next, we introduce the (for our purposes) correct notion of structure preserving maps in Lorentzian pre-length spaces. 
The following definition is closely related to corresponding notions in \cite{ACS20} and \cite{GKS19}.
\begin{defin}[Structure preserving maps]
Let $X_1$ and $X_2$ be two Lorentzian pre-length spaces.
\begin{itemize}
\item[(i)] A map $f: X_1 \rar X_2$ is called $\tau$-preserving if $\tau_1(x,y)=\tau_2(f(x),f(y))$ for all $x,y \in X_1$.
\item[(ii)] A map $f: X_1 \rar X_2$ is called $\ll$-preserving if $x \ll_1 y \iff f(x) \ll_2 f(y)$ for all $x,y \in X_1$. 
It is called $\leq$-preserving if $x \leq_1 y \iff f(x) \leq_2 f(y)$ for all $x,y \in X_1$.
If $f$ is both $\ll$-preserving and $\leq$-preserving it is called causality preserving.
\item[(iii)] A map $f: X_1 \rar X_2$ is called locally $\tau, \ll$ or $\leq$-preserving if for all $x \in X_1$ there exists a neighbourhood $U \subseteq X_1$ of $x$ such that $f|_U$ is $\tau, \ll$ or $\leq$-preserving.
\end{itemize}
\end{defin}
\begin{rem}[Implication and counterexample]
One the one hand, it is clear that any $\tau$-preserving map is $\ll$-preserving. On the other hand, a $\tau$-preserving map need not be $\leq$-preserving in general. Indeed, consider in the Minkowski plane a null segment and a spacelike segment, say $\{(s,s) \mid s \in [0,1] \}$ and $\{(0,s) \mid s \in [0,1] \}$. Then the map $(s,s) \mapsto (0,s)$ is $\tau$-preserving (it vanishes identically in both cases). But $(s,s) \leq (t,t) \iff s \leq t$ while $(0,s)$ and $(0,t)$ are never causally related.
\end{rem}
The following is immediate from the definition.
\begin{cor}[Inverse is also preserving]
If a bijective  map $f:X_1 \rar X_2$ between two Lorentzian pre-length spaces is (locally) $\tau$-preserving (respectively $\ll$ or $\leq$-preserving), then so is its inverse. In particular, the neighbourhoods in the local case are compatible in the sense that $f|_U$ is preserving if and only if $f^{-1}|_{f(U)}$ is.
\qed
\end{cor}
As a final prerequisite, we discuss the underlying quotient semi-metric. In metric amalgamation, gluing happens along isometric subsets. In the Lorentzian setting the focus lies on the time separation and the causality relations where as the distance metric plays only a background role as a topological tool. 
Since we decided to not require any preservation of Lorentzian structure of the identified sets, it is too restrictive to insist on using metric (local) isometries.
That is why we decided to use locally bi-Lipschitz homeomorphisms\footnote{We call a map $f:X \to Y$ between metric spaces locally bi-Lipschitz if every $x \in X$ has a neighbourhood $U$ such that $f|_U:U \to f(U)$ and its inverse are Lipschitz.}
 instead. 
The bi-Lipschitz condition ensures that the quotient semi-metric is positive definite and, when assuming additionally some causality preservation, that a causal curve in one of the identified sets is also a causal curve in the other.
Note that in this case, however, there is in general no chance of obtaining a nice representation of the quotient semi-metric á la \cite[Lemma I.5.24]{BH99}. 
\begin{prop}[Metric amalgamation with locally bi-Lipschitz maps]
\label{metric amalgamation}
Let $X_1$ and $X_2$ be two metric spaces and let $A_1 \subseteq X_2$ and $A_2 \subseteq X_2$ be closed subspaces.
Let $A$ be a metric space and let $f_1 : A \rightarrow A_1$ and $f_2:A \rightarrow A_2$ be locally bi-Lipschitz homeomorphisms.
Let $d$ be the disjoint union metric on $X_1 \sqcup X_2$ and consider the equivalence relation generated by $f_1(a) \sim f_2(a)$ for all $a \in A$. Then the quotient semi-metric $\tilde{d}$ on $X:=(X_1 \sqcup X_2)/\sim$ with respect to $\sim$ is a metric.
\end{prop}
\begin{proof}
$\td([x],[x])=0$ for all $[x] \in X$ and the symmetry of $\td$ are immediate from the definition. 
Concerning the triangle inequality, let $(x_1,y_1,\ldots, x_n,y_n)$ be an $n$-chain from $[x]$ to $[y]$ such that $\sum_{i=1}^n d(x_i,y_i)<\td([x],[y])+\varepsilon$ and let $(y_1',z_1,\ldots,y_m',z_m)$ be an $m$-chain from $[y]$ to $[z]$ such that $\sum_{i=1}^m d(y_i',z_i)<\td([y],[z])+\varepsilon$. Then the concatenation $(x_1,y_1,\ldots,x_n,y_n,y_1',z_1,\ldots,y_m',z_m)$ is a chain from $[x]$ to $[z]$ of length less than $\td([x],[y])+\td([y],[z])+2\varepsilon$. Hence $\td([x],[z])\leq \td([x],[y])+\td([y],[z])+2\varepsilon$ and the claim follows.
It is left to show that $\td$ is positive definite. Let $[x],[y] \in X$. Note that if, say $[x] \in X_1 \setminus A_1$, i.e., $[x]=\{x^1\}$, then any chain that is not $(x^1,y^1)$ has to go through $A_1$ by similar arguments as in Remark \ref{better chains}. In both cases we end up with a positive distance: $d(x^1,y^1)>0$ since $d$ is a metric and $d(x^1,A_1)>0$ since $A$ is closed. So the only case left to consider is when both points lie in $A$. 
To this end, note that since $f_1$ and $f_2$ are locally bi-Lipschitz homeomorphisms, it follows that $f:=f_2 \circ f_1^{-1} : A_1 \to A_2$ and $f^{-1} : A_2 \to A_1$ are locally bi-Lipschitz homeomorphisms as well.
Let $[x] \neq [y], [x], [y] \in A$. Then $x^1,y^1 \in A_1$ and $x^2,y^2 \in A_2$.
Let $r>0$ be such that $y^1 \notin B_r(x^1)\subseteq A_1$ and $y^2 \notin B_r(x^2) \subseteq A_2$. By choosing $r$ smaller if necessary, we can assume that $B_r(x^1)$ and $B_r(x^2)$ are bi-Lipschitz neighbourhoods of $f$ and $f^{-1}$ for $x^1$ and $x^2$ in $A_1$ and $A_2$, respectively. 
We can furthermore suppose that this is with respect to the same Lipschitz constant $L \geq 1$. That is, we have
\begin{equation}
\frac{1}{L}d_2(f_1(a),f_1(b)) \leq 
d_1(a,b) \leq 
Ld_2(f(a),f_1(b))
\end{equation}
for all $a,b \in B_r(x^1)$ and
\begin{equation}
\frac{1}{L}d_1(f^{-1}(a'),f^{-1}(b')) \leq 
d_2(a',b') \leq 
Ld_1(f^{-1}(a'),f^{-1}(b'))
\end{equation}
for all $a',b' \in B_r(x^2)$.
By the above arguments we can assume that any chain is of the form $(x^1,f(a_1),f^{-1}(a_1),f^{-1}(a_2),f(a_2),f(a_2),\ldots,f(a_{n-1}),y^1)$ (it does not matter whether we start in $x^1$ or $x^2$ since we could add $(x^2,x^2)$ at the beginning of the chain without increasing its length, and similar for ending in $y^1$).
Given such a chain, by setting $a_0=x^1$ and $a_n=y^1$, there exists a minimal $j \in \{1, \ldots, n\}$ such that $a_i \notin B_r(x^1)$. 
Similarly, there exists a minimal $k \in \{1, \ldots, n\}$ such that $f(a_{k}) \notin B_r(x^2)$. Without loss of generality assume $j \leq k$. Then 
we compute
\begin{align*}
& d_1(x^1,a_1) + d_2(f(a_1),f(a_2)) + d_1(a_2,a_3) + \ldots + d_1(a_{n-1},y^1) \geq \\ 
& d_1(x^1,a_1) + \frac{1}{L}d_1(a_1,a_2) + d_1(a_2,a_3) + \ldots + d_1(a_{n-1},y^1) \geq \\
& \frac{1}{L}d_1(x^1,a_1) + \frac{1}{L}d_1(a_1,a_2) + \frac{1}{L}d_1(a_2,a_3) + \ldots + d_1(a_{n-1},y^1) \geq \\
& \frac{1}{L}d_1(x^1,a_3) + \ldots + d_1(a_{n-1},y^1) \geq \ldots \geq \\
& \frac{1}{L}d_1(x^1,a_{j-1}) + d_1(a_{j-1},a_j) + \ldots + d_1(a_{n-1},y^1) \geq \\
& \frac{1}{L}d_1(x^1,a_{j-1}) + \frac{1}{L}d_1(a_{j-1},a_j) + \ldots + d_1(a_{n-1},y^1) \geq \\
& \frac{1}{L}d_1(x^1,a_{j}) + \ldots + d_1(a_{n-1},y^1) \geq \\
& \frac{1}{L}r + \ldots + d_1(a_{n-1},y^1) > \frac{1}{L}r > 0.
\end{align*}
Thus, we found a uniform lower bound for any chain from $[x]$ to $[y]$ and so $\td$ is positive definite.
\end{proof}
\subsection{Lorentzian amalgamation}
We now introduce the central object of this paper, which allows us to create a new Lorentzian pre-length space out of old ones by gluing them together, a process similar to the amalgamation of metric spaces. To avoid pathological counterexamples, we have to make minor additional assumptions on the identified subsets.
\begin{defin}[Lorentzian amalgamation]
\label{amalgamation}
Let $(X_1,d_1,\leq_1,\ll_1,\tau_1)$ and $(X_2,d_2,\leq_2,\ll_2,\tau_2)$ be two Lorentzian pre-length spaces. Let $(A,d_A,\ll_A,\leq_A,\tau_A)$ be a Lorentzian pre-length space and let $A_1$ and $A_2$ be closed non-timelike locally isolating subspaces of $X_1$ and $X_2$, respectively. 
Let $f_1 : A \rar A_1$ and $f_2 : A \rar A_2$ be locally bi-Lipschitz homeomorphisms. Suppose that the causality of $A_1$ and $A_2$ are compatible in the following sense: for all $a \in A$ we have $I_1^{\pm}(f_1(a)) \neq \emptyset \iff I_2^{\pm}(f_2(a)) \neq \emptyset$.
Let $(X_1 \sqcup X_2, d, \ll, \leq, \tau)$ be the Lorentzian disjoint union of $X_1$ and $X_2$ and consider the equivalence relation $\sim$ on $X_1 \sqcup X_2$ generated by $f_1(a) \sim f_2(a)$ for all $a \in A$. Then $((X_1 \sqcup X_2)/\sim, \td, \tll, \tleq, \ttau)$ is called the Lorentzian amalgamation (with respect to $A$) of $X_1$ and $X_2$ and is denoted by $X_1 \sqcup_A X_2$.
\end{defin}
\begin{rem}[On notation and conventions II]
As in the metric case, the space $A$ is usually introduced for easier notation, especially when dealing with an amalgamation of more than two spaces. One can think of $A$ as being any of the identified spaces $A_i$. For just two spaces, it might be more convenient to only consider a map $f:A_1 \to A_2$ (and its inverse). In fact, we will formulate the gluing theorem for two spaces only. Since most of the results in the current chapter can easily be generalized to several spaces, we will work with an extra space $A$. \\

We will use slightly sloppy notation and refer to $A$ as a subset of the amalgamation $X:=X_1 \sqcup_A X_2$ and we will view $X_1$ and $X_2$ as subsets of $X$ via the identifications with $\pi(X_1)$ and $\pi(X_2)$, respectively. 
For example, by $[a] \in A$ we mean $[a]=\{f_1(a),f_2(a)\}$. If a point is originally not in one of the identified sets, say $x \in X_1 \setminus A_1$, then it is from its on equivalence class, i.e., $[x]=\{ x \}$. If we do not care which space such a singleton is from, we may write $[x] \in X \setminus A$.
Occasionally, it will be more convenient to omit the identifying maps $f_i$. We then denote the origin of a point by a superscript. That is, if $[a] \in A$ we have $[a]=\{a^1,a^2\}$ and if $[x] \in X_1 \setminus A_1 \subseteq X \setminus A$ then $[x]=\{x^1\}$. \\

Dealing with pasts and futures in the amalgamation can also be a bit tricky notation-wise, even if we require that the identifying maps preserve structure. To this end we denote by a subscript with respect to which causality the set is constructed. For example, we write $J_1^+(x^1)=\{y^1 \in X_1 \mid x^1 \leq_1 y^1\}$ or $I_X([x],[y])=I_X^+([x]) \cap I_X^-([y])=\{[p] \in X \mid [x] \tll [p] \tll [y] \}$.
\end{rem}
Next, we formulate a lemma that is important to show the last remaining property of semi-continuity of $\ttau$.
\begin{lem}[Restricting to timelike chains]
\label{tlchains}
Let $X:=X_1 \sqcup_A X_2$ be the Lorentzian amalgamation of two Lorentzian pre-length spaces $X_1$ and $X_2$. If $\ttau([x],[y])>0$ then there is a timelike chain, i.e., $x_i \ll y_i$ for all $i$, from $[x]$ to $[y]$ whose length is arbitrarily close to $\ttau([x],[y])$.
\end{lem}
\begin{proof}
Let $[x], [y] \in X$ with $\ttau([x],[y])>0$. Given small enough $\varepsilon > 0$ we find a chain $(x_1,y_1,\ldots,x_n,y_n)$ such that $\sum_{i=1}^n \tau(x_i,y_i)>\ttau([x],[y])-\varepsilon > 0$. 
By Remark \ref{better chains}, we can assume that it is of the form $\tau_1(x^1,f_1(a_1)) + \tau_2(f_2(a_1),f_2(a_2)) + \ldots + \tau_2(f_2(a_n),y^2)$. Note that the assumption of $x^1 \in [x]$ does not lose any generality: if this were not the case the chain simply would start out in $X_2$. Similarly for $y^2 \in [y]$.
Since the length of this chain is positive, we have that at least one entry is positive, say $\tau_1(f_1(a_j),f_1(a_{j+1}))>0$. Then $I_1^+(f_1(a_j)) \neq \emptyset$ and hence also $I_2^+(f_2(a_j)) \neq \emptyset$. Suppose $\tau_2(f_2(a_{j-1}),f_2(a_j))=0$. 
Consider the neighbourhood $S_1:=\{a \in A_1 \mid \tau(a,f_1(a_{j+1}))>\tau(f_1(a_j),f_1(a_{j+1}) - \frac{\varepsilon}{n} \}$ of $f_1(a_1)$ in $A_1$, which is open by the lower semi-continuity of $\tau_1$. Since $f_1$ and $f_2$ are homeomorphisms it follows that $S_2:=f_2(f_1^{-1}(S_1))$ is an open neighbourhood of $f_2(a_j)$ in $A_2$.
Since $A_2$ is non-timelike locally isolating and  $I_2^+(f_2(a_j)) \neq \emptyset$, we find $f_2(b_+) \in S_2$ such that $f_2(a_j) \ll f_2(b_+)$. Then by the push-up property of $\tau_1$ we have $f_2(a_{j-1}) \ll f_2(b_+)$. Also, $f_1(b_+) \in S_1$. Then
\begin{equation}
\tau_2(f_2(a_{j-1}),f_2(b_+)) + \tau_1(f_1(b_+),f_1(a_{j+1})) > 
\tau_2(f_2(a_{j-1}),f_2(a_j)) + \tau_1(f_1(a_j),f_1(a_{j+1})) - \frac{\varepsilon}{n}.
\end{equation}
We can do this for all other entries of this chain to end up with a timelike chain whose length is at most $\varepsilon$ less than the chain we started with. 
Figure \ref{tlchains pic} illustrates this process.
\end{proof}
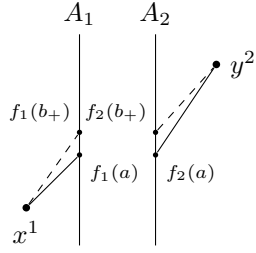
\begin{figure}
\begin{center}
\begin{tikzpicture}
\coordinate (A1) at (0,-1.2);
\coordinate [label=90: {$A_1$}] (A2) at (0,1.6);
\coordinate (B1) at (1,-1.2);
\coordinate [label=90: {$A_2$}] (B2) at (1,1.6);

\coordinate [circle, fill=black, inner sep=1pt, label=below: {$x^1$}] (p) at (-0.7,-0.7);
\coordinate [circle, fill=black, inner sep=1pt, label=right: {$y^2$}] (q) at (1.8,1.2);
\begin{scriptsize}

\coordinate [circle, fill=black, inner sep=0.7pt, label=340: {$f_2(a)$}] (y) at (1,0);
\coordinate [circle, fill=black, inner sep=0.7pt, label=330: {$f_1(a)$}] (x) at (0,0);

\coordinate [circle, fill=black, inner sep=0.7pt, label=110: {$f_2(b_+)$}] (y') at (1,0.3);
\coordinate [circle, fill=black, inner sep=0.7pt, label=110: {$f_1(b_+)$}] (x') at (0,0.3);
\end{scriptsize}

\draw (A1)--(A2);
\draw (B1)--(B2);
\draw (p)--(x);
\draw (y)--(q);
\draw [dashed] (p)--(x');
\draw [dashed] (y')--(q);
\end{tikzpicture}
\end{center}
\caption{Moving $a$ to make the null piece connecting $x^1$ and $f_1(a)$ timelike.}
\label{tlchains pic}
\end{figure}
\begin{ex}[On minimal assumptions and bad behaviour]
In Definition \ref{amalgamation} the compatibility of the causality in $A_1$ and $A_2$ via $I_1^{\pm}(f_1(a)) \neq \emptyset \iff I_2^{\pm}(f_2(a)) \neq \emptyset$ for all $a \in A$, which is weaker than $f_1$ and $f_2$ being locally $\ll$-preserving, is truly a necessary condition. Let $X_1$ and $X_2$ be a closed unit square in the Minkowski plane equipped with the restricted Lorentzian structure. Identify them along a vertical line segment and reverse the time orientation in one of the squares, as is indicated by the arrows in Figure \ref{mink square}. Then $I_1^+(a^1)=\emptyset$ but $I_2^+(a^2) \neq \emptyset$. Choosing endpoints and a sequence as in Example \ref{gluing example1} leads to a similar failure of the lower semi-continuity of $\ttau$.
\begin{figure}
\begin{center}
\begin{tikzpicture}
\draw (0,0) -- (1.5,0) -- (1.5,1.5) -- (0,1.5) -- (0,0);
\draw (0.75,0) -- (0.75,0.75);
\draw [latex reversed-](0.75,0.75) -- (0.75,1.5);

\draw (2.5,0) -- (4,0) -- (4,1.5) -- (2.5,1.5) -- (2.5,0);
\draw [-latex] (3.25,1.5) -- (3.25,0.75);
\draw (3.25,0.8) -- (3.25,0);

\begin{scriptsize}
\coordinate [circle, fill=black, inner sep=0.5pt, label=90: {$a^1$}] (a1) at (0.75,1.5);
\coordinate [circle, fill=black, inner sep=0.5pt, label=90: {$a^2$}] (a2) at (3.25,1.5);
\coordinate [circle, fill=black, inner sep=0.7pt] (q2) at (3.5,0.5);
\coordinate [circle, fill=black, inner sep=0.7pt] (p1) at (0.4,1.15);
\coordinate (pn) at (0.2,1.35);
\end{scriptsize}

\draw [dotted] (pn) -- (p1);
\draw [color=red] (a2) -- (q2);
\draw [color=red] (a1) -- (p1);

\draw [dashed] (a1) -- (1.5,0.75);
\draw [dashed] (a1) -- (0,0.75);

\draw [dashed] (a2) -- (4,0.75);
\draw [dashed] (a2) -- (2.5,0.75);
\end{tikzpicture}
\end{center}
\caption{This amalgamation leads to a $\ttau$ that is not lower semi-continuous. Note that the arrows denote the time orientation and not the gluing identifications.}
\label{mink square}
\end{figure}
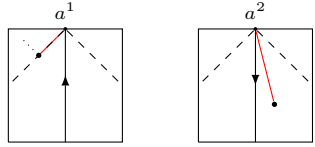
However, this does not mean that every Lorentzian pre-length space resulting from the amalgamation process is well behaved. Indeed, similar to the above example let $X_1$ and $X_2$ be the whole Minkowski plane and identify them along a vertical line with reversed time orientation in one space. Then although $X_1$ and $X_2$ are (much more than) chronological, the amalgamation fails to be chronological and moreover we have $\ttau([x],[y]) = \infty$ for all $[x],[y]$.
\end{ex}
\begin{rem}[Convergence in the amalgamation]
\label{convergence}
Suppose $[x_n] \to [x]$.
Clearly, if $[x]=\{x^1 \}$ then $[x_n]=\{x_n^1\}$ for large enough $n$. If $[x] \in A$, then at least in one of $X_1$ or $X_2$, say $X_1$, there exists a subsequence $x_{n_k}^1 \to x^1$. 
Indeed, if this were not the case then there would exist neighbourhoods of $x^1$ and $x^2$ in $X_1$ and $X_2$, respectively, that do not contain any points of the sequence. This is in contradiction to how the quotient topology is defined.
\end{rem}
\begin{prop}[Amalgamation is Lorentzian pre-length space]
Let $X$ be the Lorentzian amalgamation of two Lorentzian pre-length spaces $X_1$ and $X_2$. Then $X$ is a Lorentzian pre-length space.
\end{prop}
\begin{proof}
We only need to prove that $\ttau$ is lower semi-continuous. Let $[x] , [y] \in X$ and let $[x_n] \rar [x], [y_n] \rar [y]$. Then we need to show $\ttau([x_n],[y_n]) \geq \ttau([x],[y]) - \varepsilon$ for all $\varepsilon > 0$.
Note that if $\ttau([x],[y])=0$ there is nothing to show, so suppose $\ttau([x],[y])>0$. 
Given a small enough $\varepsilon > 0$, by Remark \ref{better chains} and Lemma \ref{tlchains} we find a timelike chain such that 
\begin{equation}
\tau_1(x^1,f_1(a_1)) + \tau_2(f_2(a_1),f_2(a_2)) + \ldots + \tau_2(f_2(a_m),y^2) > \ttau([x],[y]) -\varepsilon > 0.
\end{equation}
Assume $\varepsilon$ is so small that even $\ttau([x],[y]) - 3\varepsilon > 0$.
Then the claim follows from the lower semi-continuity of the original time separation functions: 
by Remark \ref{convergence} and since we assumed $x^1 \in [x], y^2 \in [y]$, we can further assume $x_n^1 \to x^1$ and $y_n^2 \to y^2$ (otherwise $[x],[y] \in A$ and we find subsequences in the other space).
Let
\begin{equation}
\delta:= \frac{1}{2} \min \{\tau_1(x^1,f_1(a_1)), \tau_2(f_2(a_m),y^2), \varepsilon \} > 0.
\end{equation}
Then since $\tau_1$ and $\tau_2$ are lower semi-continuous there exists $n_0 \in \N$ such that $\tau_1(x_n^1,f_1(a_1))>\tau_1(x^1,f_1(a_1)) - \delta$ and $\tau_2(f_2(a_m),y_n^2)>\tau_2(f_2(a_m),y^2) - \delta$ for all $n \geq n_0$. Then
\begin{align*}
\ttau([x_n],[y_n]) & \geq 
\tau_1(x_n^1,f_1(a_1)) + \tau_2(f_2(a_1),f_2(a_2)) + \ldots + 
\tau_2(f_2(a_m),y_n^2) \\
& > \tau_1(x^1,f_1(a_1)) + \tau_2(f_2(a_1),f_2(a_2)) + \ldots + 
\tau_2(f_2(a_m),y^2) - 2 \delta \\
& > \ttau([x],[y]) - \varepsilon - 2\delta > \ttau([x],[y]) - 3\varepsilon >0
\end{align*}
and we are done.
\end{proof}
Finally, we note that if we assume global preservation of the Lorentzian structure, the quotient time separation has the following form, which is familiar from the metric case.
\begin{prop}[Short form of quotient time separation]
\label{reptimesep}
Let $X:= X_1 \sqcup_A X_2$ be the Lorentzian amalgamation of two Lorentzian pre-length spaces $X_1$ and $X_2$ and assume that $f_1$ and $f_2$ are $\tau$-preserving and causality preserving. Then the quotient time separation has the following form:
\begin{equation}
\ttau([x],[y])=
\begin{cases} \tau_i(x^i,y^i) & x^i,y^i \in X_i, \\
\underset{[a] \in J_X([x],[y]) \cap A}{\sup} \{ \tau_i(x^i,f_i(a))+\tau_j(f_j(a),y^j) \} & x^i \in X_i, y^j \in X_j, i,j \in \{1,2\}, i \neq j.
\end{cases}
\end{equation}
\end{prop}
\begin{proof}
Let, say, $x^1 \in [x],y^1 \in [y]$ and let $(x_1,y_1, \ldots, x_n,y_n)$ be an $n$-chain from $[x]$ to $[y]$. By Remark \ref{better chains} we can assume this chain to have nontrivial relations everywhere and $x_1=x^1$ as well as $y_n=y^1$.
Intuitively, the chain starts out at $x^1$ in $X_1$, moves to $A_1$, jumps around between $A_1$ and $A_2$ and ends up at $y^1$ in $X_1$ again. Since $f_1$ and $f_2$ are $\tau$-preserving and causality preserving, we can replace all distances from points in $A_2$ with corresponding equal distances in $A_1$. But then by several applications of the reverse triangle inequality for $\tau_1$ we can replace all these distances in $A_1$ by a single distance to obtain a longer chain. Finally, we can omit the detour through $A_1$ altogether by replacing it with the direct distance from $x^1$ to $y^1$ in $X_1$. \\

The second case is symmetric, so suppose $x^1 \in [x]$ and $y^2 \in [y]$. If we have an $n$-chain $(x_1,y_1, \ldots, x_n,y_n)$ with nontrivial relations, then as above the chain starts out at $x^1$ in $X_1$, moves to $A_1$, jumps around between $A_1$ and $A_2$, but this time ends up at $y^2$ in $X_2$. Since $f_1$ and $f_2$ are $\tau$ and causality preserving we can replace, say, all distances in $A_1$ with equal distances in $A_2$ and then apply the reverse triangle inequality multiple times to end up with a 2-chain with at least the same length as in the claim. More precisely, we estimate the length of the $n$-chain 
\begin{align*}
& \sum_{i=1}^n \tau(x_i,y_i) = \\
& \tau_1(x^1,f_1(a_1))+\tau_2(f_2(a_1),f_2(a_2)) + \ldots + 
\tau_1(f_1(a_{n-2})),f_1(a_{n-1})+\tau_2(f_2(a_{n-1}),y^2) = \\
& \tau_1(x^1,f_1(a_1)) + \tau_2(f_2(a_1),f_2(a_2)) + \ldots + 
\tau_2(f_2(a_{n-2}),f_2(a_{n-1}))+\tau_2(f_2(a_{n-1}),y^2) \leq \\
& \tau_1(x^1,f_1(a_1))+\tau_2(f_2(a_1),y^2).
\end{align*}
\end{proof}
The following lemma is both helpful for the gluing theorem and a nice statement in its own right. It says that certain diamonds in the glued space are built from corresponding diamonds in the original spaces.
\begin{lem}[Amalgamated diamonds]
\label{amalgamated diamonds}
Let $X$ be the amalgamation of two Lorentzian pre-length spaces $X_1$ and $X_2$. Assume that $f_1: A \to A_1$ and $f_2: A \to A_2$ are $\leq$-preserving. Then we have the following decomposition of causal diamonds in $X$:
\begin{itemize}
\item[(i)] If $[x],[y] \in A$, then $J_X([x],[y])=\pi(J_1(x^1,y^1) \sqcup J_2(x^2,y^2))$.
\item[(ii)] Let $i \in \{1,2\}$. If $x^i,y^i \in X_i \setminus A$ are such that $J_i(x^i,y^i) \cap A_i = \emptyset$, then $J_X([x],[y])=\pi(J_i(x^i,y^i))$. 
\end{itemize}
If $f_1$ and $f_2$ are $\ll$-preserving, we get the same statement for timelike diamonds.
\end{lem}
\begin{proof}
One inclusion always holds by Corollary \ref{easy properties}: if $q^i \in J_i(x^i,y^i)$, i.e., $x^i \leq_1 q^i \leq_i y^i$, then $[x] \tleq [q] \tleq [y]$, hence $[q] \in J_X([x],[y])$. We show the other inclusion separately. \\

(i) Let $[q] \in J_X([x],[y])$ and assume without loss of generality $q^1 \in [q]$.
Then $[x] \tleq [q]$, i.e., there exists a chain such that $x^1 \leq_1 a_1^1 \sim a_1^2 \leq_2 a_2^2 \sim a_2^1 \leq_1 a_3^1 \sim \ldots \sim a_n^1 \leq_1 q^1$. Since $f_1$ and $f_2$ are $\leq$-preserving, we can transfer each relation in $X_2$ to a corresponding relation in $X_1$ and ultimately obtain $x^1 \leq_1 q^1$. Analogously, we obtain $q^1 \leq_1 y^1$ and the claim follows. \\

(ii) Let $[q] \in J_X([x],[y])$. Note that $[x]=\{x^i\}$ and $[y]=\{y^i\}$. Since $J_i(x^i,y^i) \cap A_i = \emptyset$ and $f_1$ and $f_2$ are $\leq$-preserving, any chain starting at $x^i$ and ending at $y^i$ must necessarily only contain trivial relations and so has to stay inside $X_i$. Thus, $[q] = \{q^i\}, x^i \leq_i q^i$ and $q^i \leq_i y^i$.
\end{proof}

\section{Gluing preparations}
In this section, we revisit the Lorentzian version of Alexandrov's lemma and show a gluing lemma for timelike triangles, which is essential for the proof of the gluing theorem. At first, we recall some terminology from \cite{AB08}, which is of great use to us. For a more detailed analysis of the contents therein, see \cite{Kir18}. \\

Our approach for proving the gluing theorem in the Lorentzian setting is in spirit very close to the metric version, which introduces a so-called gluing lemma for triangles, cf. \cite[Lemma II.4.10]{BH99}, whose proof relies on Alexandrov's lemma. While \cite[Lemma 2.4]{AB08} certainly is a very powerful formulation of Alexandrov's Lemma valid in any semi-Riemannian manifold, this is not ideal for our situation since it relies too much on the differential structure present in the model spaces. 
\subsection{Signed distance and other techniques}
In this first subsection, we introduce some very general concepts from \cite{AB08} and collect some useful facts concerning angles.
\begin{defin}[Signed length and signed distance]
Let $(M,g)$ be a spacetime and let $p \in M$.
For $v \in T_pM$, we  denote the ``norm'' of $v$ by $|v|:=\sqrt{|g_p(v,v)|}$. 
We then define the signed length of $v$ as $|v|_{\pm}:=\sgn(v)\sqrt{|g_p( v,v)|}=\sgn(v)|v|$, where 
\begin{equation}
\sgn(v):= \begin{cases} 1 & g_p(v,v)  \geq 0, \\
-1 & g_p(v,v) <0. \end{cases}
\end{equation}
If $p,q \in M$ are contained in a normal\footnote{We follow the notation of \cite{AB08} where a normal neighbourhood is a (diffeomorphic) exponential image of an open set of the tangent space. More commonly, such neighbourhoods may also be called geodesically convex neighbourhoods or convex normal neighbourhoods.} neighbourhood $U$, then there exists a unique geodesic $\gamma_{pq}: [0,1] \rar M$ connecting $p$ and $q$ contained in $U$. We define the signed distance of $p$ and $q$ as $|pq|_{\pm}:=|\gamma_{pq}'(0)|_{\pm}$.
Note that if $p \leq q$ and $M$ is strongly causal, then 
\begin{equation}
\label{taurep}
\tau(p,q)=-|pq|_{\pm}=\sqrt{-g_p(\gamma_{pq}'(0), \gamma_{pq}'(0))}.
\end{equation}
\end{defin}
\begin{defin}[Hyperbolic and nonnormalized angle]
Let $M$ be a spacetime and let $p \in M$. Let $q,r \in I^{\pm}(p)$ be such that $\gamma_{pq}'(0)=:v$ and $\gamma_{pr}'(0)=:w$ exist. The hyperbolic angle between $q$ and $r$ at $p$ is defined as
\begin{equation}
\mad_p(q,r):=\arcosh \left( \frac{|g_p( v,w)|}{|v||w|}\right) .
\end{equation}
Let now $q,r$ be any points in a normal neighbourhood of $p$. The nonnormalized angle between $q$ and $r$ at $p$ is defined as $\angle qpr := g_p( v,w )$.
\end{defin}
These two notions of angles are closely related (if the points are timelike related). Keeping the above terminology, one immediately sees that
\begin{equation}
\mad_p(q,r)=\arcosh \left( \frac{ |\angle pqr|}{|v||w|}\right).
\end{equation}
Note that the nonnormalized angle is much more general in the sense that the points need not be timelike related at all.
However, if both angles exist, the nonnormalized angle in some way better captures what ``type'' of angle we are dealing with, i.e., if the tangent vectors lie in the same timecone or not (recall that for timelike vectors $v$ and $w$, $g(v,w) <0$ if both are future or past directed and $g(v,w) >0$ if they have different time orientation).
Observe that the nonnormalized angle is not scale invariant.
\begin{rem}[Implicit inequalities on angles]
\label{angle inequality}
With this in mind, we want to touch on an implication of inequalities of angles. Keeping the above notation, if say $\mad_p(q,r) \leq \mad_{p'}(q',r'), |v|=|v'|,|w|=|w'|$ and both tangent vectors point in the same direction, then $\angle qpr \geq \angle q'p'r'$ (and vice versa). 
If the vectors lie in different timecones, then the inequality of the nonnormalized angles reverses. 
Note that $g_p( v,w)$ and $g_p( v',w' )$ must have the same sign in order to infer some inequality. 
\end{rem}
One big advantage of signed distance and the nonnormalized angle is the very powerful hinge lemma, of which we will make extensive use. For a proof of the following statement see \cite[Lemma 2.2]{AB08}. 
\begin{lem}[Hinge lemma]
\label{hinge}
Let $(|pq|_{\pm},|qr|_{\pm},|pr|_{\pm}) \in \R^3 \setminus \{0\}$ be a triple realizable as the sidelengths of a triangle in $M_K$. If we vary the length of the third side
$|pr|_{\pm}$ while keeping $|pq|_{\pm}$ and $|qr|_{\pm}$ fixed, then:
\begin{itemize}
\item[(i)]The nonnormalized angle $\angle pqr$ is a decreasing function of $|pr|_{\pm}$.
\item[(ii)]The nonnormalized angles $\angle qpr$ and $\angle qrp$ are increasing functions of $|pr|_{\pm}$.
\qed
\end{itemize}
\end{lem}
The following lemma will be essential for one case in the gluing lemma, cf. \cite[Lemma 5.1.1]{Kir18}.
\begin{lem}[Bounding function via derivative]
\label{help}
For $k \in \R$, let $f:[0,L] \rar \R$ be a smooth function such that $f'' + kf \leq 0, f(0)=0$ and $f(L)=0$. If $k>0$, assume additionally that $L<\frac{\pi}{\sqrt{k}}$. Then $f(t)\geq 0$ for all $t \in [0,L]$.
\qed
\end{lem}
We note that, on the one hand, by replacing $f$ with $-f$, we can make a similar statement with reversed inequalities. That is, if $f''+kf \geq 0$, then $f \leq 0$. On the other hand, Lemma \ref{help} holds as well if we only assume $f(0) \geq 0$ and $f(L)\geq 0$. Indeed, suppose indirectly that this is not the case, then there exists $t_0 \in (0,L)$ such that $f(t_0)<0$. 
By the mean value theorem, there must exist $t_1 \in (0,t_0), t_2 \in (t_0,L)$ such that $f(t_1)=f(t_2)=0$. Then simply apply Lemma \ref{help} to $[t_1,t_2] \subseteq [0,L]$ to obtain the desired contradiction.
\begin{rem}[Dealing with hyperbolic angles]
\label{between}
Before proceeding to the following results, we want to mention some useful properties of the hyperbolic angle. Contrary to the nonnormalized angle, the hyperbolic angle is independent of the length of its sides (when considering it as a hinge).
Suppose in $M_k$ (or in fact, any two-dimensional spacetime) we have three geodesics emanating from a point $p$ that all go into the same time direction, say the future. Then we can definitively speak of a ``middle segment'', say $[p,y]$ lies between\footnote{The segment $[p,y]$ is said to lie between $[p,x]$ and $[p,z]$ if the following holds: consider the three tangent vectors corresponding to the three segments. These have the same time orientation. Extend them to rays and denote these by $R_x, R_y$ and $R_z$. Then for all $v \in R_x, w \in R_z$ the connecting segment $[v,w]$ intersects the ray $R_y$.} $[p,x]$ and $[p,z]$. Then the triangle equality for hyperbolic angles holds:
\begin{equation}
\mad_p(x,y)+\mad_p(y,z)=\mad_p(x,z).
\end{equation}
This is immediate when viewing the hyperbolic angle as the area under a hyperbolic segment.
\end{rem}
\begin{rem}[Hinge behaviour]
\label{hingebehaviour}
Furthermore, we want to highlight a fact valid in any spacetime: consider a timelike triangle $\Delta(q,r,p)$ with $a=\tau(q,p), b= \tau(r,p), c=\tau(q,r)$ and $\omega = \mad_p(q,r)$. Moving the point $q$ further into the past along the geodesic extending $[q,p]$ causes the distance from $r$ to $q$ to increase as well. More precisely, if $q' \ll q$ is such that $[q,p] \subseteq [q',p]$, then $c':=\tau(q',r) \geq \tau(q,r)=c$. This is a simple consequence of the reverse triangle inequality. 
We want to reformulate this as follows, so that it can be applied similarly to the hinge lemma: let $\alpha$ and $\beta$ be the (past-oriented) geodesics extending the segments $[q,p]$ and $[r,p]$, respectively. Assume $\alpha(1)=q$. Consider the hinge $(\alpha,\beta)$ and denote the included hyperbolic angle by $\omega$. Then $\tau(\alpha(t),r)$ is an increasing function\footnote{Note that this of course only makes sense if the geodesic $\alpha$ can actually be extended. Since we are anyways only interested in applying this in model spaces, this is not problematic.} in $t$ for all $t \geq 1$. Intuitively, one should think of the behaviour illustrated in Figure \ref{hinge pic}: increasing the longest side in a timelike triangle while keeping one of the short sides fixed (and hence also the included hyperbolic angle) causes the other short side to increase as well.
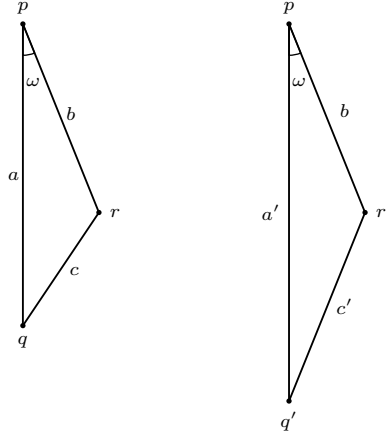
\begin{figure}
\begin{center}
\begin{tikzpicture}[line cap=round,line join=round,>=triangle 45,x=1cm,y=1cm]
\clip(-1.965667650999136,-1.4687684881236897) rectangle (9.582698785866823,4.438312101060029);
\draw [shift={(0,4)},line width=0.5pt] (0,0) -- (-90:0.41518695223973838) arc (-90:-68.19859051364818:0.41518695223973838) -- cycle;
\draw [shift={(3.5,4)},line width=0.5pt] (0,0) -- (-90:0.41518695223973838) arc (-90:-68.19859051364818:0.41518695223973838) -- cycle;
\draw [line width=0.7pt] (0,0)-- (0,4);
\draw [line width=0.7pt] (0,4)-- (1,1.5);
\draw [line width=0.7pt] (0,0)-- (1,1.5);
\draw [line width=0.7pt] (3.5,4)-- (3.5,-1);
\draw [line width=0.7pt] (3.5,-1)-- (4.5,1.5);
\draw [line width=0.7pt] (4.5,1.5)-- (3.5,4);
\begin{scriptsize}
\coordinate [circle, fill=black, inner sep=0.7pt, label=270: {$q$}] (A1) at (0,0);

\coordinate [circle, fill=black, inner sep=0.7pt, label=90: {$p$}] (A1) at (0,4);

\draw[color=black] (-0.1235411665134122,1.9866644061716197) node {$a$};
\draw[color=black] (0.6309215726937071,2.8115477230906163) node {$b$};
\draw[color=black] (0.6737486742857158,0.7244445752210946) node {$c$};

\coordinate [circle, fill=black, inner sep=0.7pt, label=0: {$r$}] (A1) at (1,1.5);

\coordinate [circle, fill=black, inner sep=0.7pt, label=90: {$p$}] (A1) at (3.5,4);

\coordinate [circle, fill=black, inner sep=0.7pt, label=270: {$q'$}] (A1) at (3.5,-1);

\coordinate [circle, fill=black, inner sep=0.7pt, label=0: {$r$}] (A1) at (4.5,1.5);

\draw[color=black] (3.2637958661383847,1.5060802128362043) node {$a'$};
\draw[color=black] (4.231295995793451,0.25386038188567926) node {$c'$};
\draw[color=black] (4.231716573505329,2.854585113538564) node {$b$};
\draw[color=black] (0.1348467917978117,3.230519874737255) node {$\omega$};
\draw[color=black] (3.6348467917978117,3.230519874737255) node {$\omega$};
\end{scriptsize}
\end{tikzpicture}
\end{center}
\caption{Increasing $a$ causes $c$ to increase as well.}
\label{hinge pic}
\end{figure}
\end{rem}
\subsection{A visual Lorentzian version of Alexandrov's lemma}
As in the metric case, the gluing lemma heavily relies on Alexandrov's lemma. \cite{AB08} gives an exceptionally general version of Alexandrov's lemma valid in semi-Riemannian comparison theory. The following lemma may in some sense be regarded as a special case of \cite[Lemma 2.4]{AB08} but we still feel justified to give a full statement here. On the one hand, our approach is much more visual and hence more in the spirit of the original (Euclidean) version of the lemma. On the other hand, we have to avoid assumptions on the behaviour of nonnormalized angles in the comparison situation since the triangles are originally from Lorentzian pre-length spaces, where currently the concept of (hyperbolic) angles is not fully developed. 
\begin{lem}[Alexandrov's lemma, Lorentzian version]
\label{alexlem}
Let $\Delta(x,p,y)$ and $\Delta(p,y,z)$ be two triangles in $M_K$ arranged in a way such that $x$ and $z$ lie on opposite sides of the geodesic line extending $[p,y]$.
Suppose $p,y \in I(x,z)$ and $\tau(x,y)+\tau(y,z) < \tau(x,p)+\tau(p,z)$. 
Let $\Delta(x',y',z')$ be a (timelike) triangle in $M_K$ such that $\tau(x',y')=\tau(x,y), \tau(y',z')=\tau(y,z)$ and $\tau(x',z')=\tau(x,p)+\tau(p,z)$. In particular, we have to assume that $\Delta(x,p,y)$ and $\Delta(p,y,z)$ are small enough such that the size bounds for $\Delta(x',y',z')$ are satisfied as well.
Then $\angle xyz \leq \angle x'y'z'$. 
Moreover, if $[x,z] \cap [p,y] = \emptyset$ then $\angle pzy \geq \angle p'z'y', \angle pxy \geq \angle p'x'y'$ and $|py|_{\pm} \leq |p'y'|_{\pm}$.
And if $[x,z] \cap [p,y] \neq \emptyset$ then $\angle pzy \leq \angle p'z'y', \angle pxy \leq \angle p'x'y'$ and $|py|_{\pm} \geq |p'y'|_{\pm}$ with equalities everywhere if and only if one of the inequalities is an equality, which happens if and only if $[x,z] \cap [p,y] = \{p\}$. 
\end{lem}
\begin{proof}
We only show the case $[x,z] \cap [p,y] = \emptyset$, which is also the relevant one for the gluing lemma. The other case is done analogously (see last paragraph). 
To begin with, we want to explain the assumption 
\begin{equation}
\label{triangle existence}
\tau(x,y)+\tau(y,z) < \tau(x,p)+\tau(p,z).
\end{equation}
On the one hand, this gives the reverse triangle inequality in the ``straightened'' big triangle, guaranteeing its (nondegenerate) existence. On the other hand, this ensures that we can actually apply Alexandrov's lemma to the gluing lemma below. 
The condition $[x,z] \cap [p,y] = \emptyset$ serves as an analogue to having an angle greater than $\pi$ at $p$ in the metric case. So from a Euclidean point of view, the quadrilateral is concave. But if you take a look at Figure \ref{impossibley}, there are two possibilities for this quadrilateral to turn out concave. To put it another way: we know that the quadrilateral is concave, but since we cannot explicitly connect this nonempty intersection with a (Euclidean) large angle at $p$, it could happen that this large angle appears at $y$.

\begin{minipage}{0.3\textwidth}
\begin{tikzpicture}[line cap=round,line join=round,>=triangle 45,x=1cm,y=1cm]
\clip(-1.569188167107672,-0.3031949050428182) rectangle (8.913822726890942,5.042060077152448);
\draw [line width=0.7pt] (0.6555114010631212,2.2634400795546528)-- (1.5,0);
\draw [line width=0.7pt] (1.5,0)-- (1.4788493040152697,1.3001720470540192);
\draw [line width=0.7pt] (1.4788493040152697,1.3001720470540192)-- (0.6555114010631212,2.2634400795546528);
\draw [line width=0.7pt] (0.6555114010631212,2.2634400795546528)-- (2.9621538734494184,4.646840905143243);
\draw [line width=0.7pt] (2.9621538734494184,4.646840905143243)-- (1.4788493040152697,1.3001720470540192);
\begin{scriptsize}
\coordinate [circle, fill=black, inner sep=0.7pt, label=0: {$x$}] () at (1.5,0);
\coordinate [circle, fill=black, inner sep=0.7pt, label=0: {$p$}] () at (1.4788493040152697,1.3001720470540192);
\coordinate [circle, fill=black, inner sep=0.7pt, label=180: {$y$}] () at (0.6555114010631212,2.2634400795546528);
\coordinate [circle, fill=black, inner sep=0.7pt, label=0: {$z$}] () at (2.9621538734494184,4.646840905143243);
\end{scriptsize}
\end{tikzpicture}
\end{minipage}
\begin{minipage}{0.3\textwidth}
\begin{tikzpicture}[line cap=round,line join=round,>=triangle 45,x=1cm,y=1cm]
\clip(-1.569188167107672,-0.3031949050428182) rectangle (8.913822726890942,5.042060077152448);
\draw [line width=0.7pt] (1.6655586485868488,2.1065160018670412)-- (1.5,0);
\draw [line width=0.7pt] (1.5,0)-- (2.1881064673624286,1.383289742037438);
\draw [line width=0.7pt] (2.1881064673624286,1.383289742037438)-- (1.6655586485868488,2.1065160018670412);
\draw [line width=0.7pt] (1.6655586485868488,2.1065160018670412)-- (1.1201575873716878,4.567707241404234);
\draw [line width=0.7pt] (1.1201575873716878,4.567707241404234)-- (2.1881064673624286,1.383289742037438);
\begin{scriptsize}
\coordinate [circle, fill=black, inner sep=0.7pt, label=0: {$x$}] () at (1.5,0);
\coordinate [circle, fill=black, inner sep=0.7pt, label=0: {$p$}] () at (2.1881064673624286,1.383289742037438);
\coordinate [circle, fill=black, inner sep=0.7pt, label=180: {$y$}] () at (1.6655586485868488,2.1065160018670412);
\coordinate [circle, fill=black, inner sep=0.7pt, label=180: {$z$}] () at (1.1201575873716878,4.567707241404234);
\end{scriptsize}
\end{tikzpicture}
\end{minipage}
\begin{minipage}{0.3\textwidth}
\begin{tikzpicture}[line cap=round,line join=round,>=triangle 45,x=1cm,y=1cm]
\clip(-1.569188167107672,-0.3031949050428182) rectangle (8.913822726890942,5.042060077152448);
\draw [line width=0.7pt] (1.6655586485868488,2.1065160018670412)-- (1.5,0);
\draw [line width=0.7pt] (1.5,0)-- (2.1881064673624286,1.383289742037438);
\draw [line width=0.7pt] (2.1881064673624286,1.383289742037438)-- (1.6655586485868488,2.1065160018670412);
\draw [line width=0.7pt] (1.6655586485868488,2.1065160018670412)-- (1.1201575873716878,4.567707241404234);
\draw [line width=0.7pt] (1.1201575873716878,4.567707241404234)-- (2.1881064673624286,1.383289742037438);
\draw [line width=0.7pt,dashed] (1.6655586485868488,2.1065160018670412)-- (1.7187190676909596,2.7829123995461593);
\begin{scriptsize}
\coordinate [circle, fill=black, inner sep=0.7pt, label=0: {$x$}] () at (1.5,0);
\coordinate [circle, fill=black, inner sep=0.7pt, label=0: {$p$}] () at (2.1881064673624286,1.383289742037438);
\coordinate [circle, fill=black, inner sep=0.7pt, label=180: {$y$}] () at (1.6655586485868488,2.1065160018670412);
\coordinate [circle, fill=black, inner sep=0.7pt, label=180: {$z$}] () at (1.1201575873716878,4.567707241404234);
\coordinate [circle, fill=black, inner sep=0.7pt, label=0: {$q$}] () at (1.7187190676909596,2.7829123995461593);
\end{scriptsize}
\end{tikzpicture}
\end{minipage}

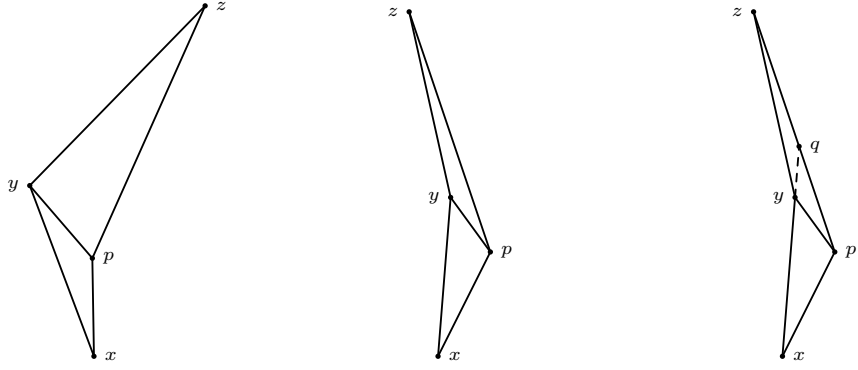
\captionof{figure}{The left and the middle configuration depict the two possible cases of a concave quadrilateral. The figure on the right illustrates how to rule out the middle case.} 
\label{impossibley}
\vspace{7mm}
Suppose we are in this case, then extend the segment $[x,y]$ until it intersects $[p,z]$, say in a point $q$. Then with the reverse triangle inequality and the fact that $y \in [x,q]$ and $q \in [p,z]$, we compute
\begin{align*}
\tau(x,y)+\tau(y,z) & \geq \tau(x,y)+\tau(y,q)+\tau(q,z)=\tau(x,q)+\tau(q,z) \\
& \geq \tau(x,p)+\tau(p,q)+\tau(q,z)=\tau(x,p)+\tau(p,z),
\end{align*}
a contradiction to (\ref{triangle existence}). \\

Turning now to the actual proof, note that by the reverse triangle inequality we have $\tau(x',z')=\tau(x,p)+\tau(p,z) \leq \tau(x,z)$. For the angle at $y$, consider the triangles $\Delta(x,y,z)$ and $\Delta(x',y',z')$. They have two sides of equal length, and since $\tau(x',z') \leq \tau(x,z)$, we have $|xz|_{\pm} \leq |x'z'|_{\pm}$. Thus, $\angle xyz \geq \angle x'y'z'$ by the hinge lemma. \\

For the remaining estimates, we follow the same visual approach as the original version of Alexandrov's lemma.
Let $\tilde{x}$ be the unique point with $\tau(\tx,p)=\tau(x,p)$ such that $\tilde{x}$ lies on the extension of the timelike geodesic segment $[p,z]$, 
see Figure \ref{visual alexlem}.
\begin{figure}
\begin{center}
\begin{tikzpicture}[line cap=round,line join=round,>=triangle 45,x=1cm,y=1cm]
\clip(-2.867209809996504,-0.8591159937864587) rectangle (11.436796309659279,6.719093706391235);
\draw [line width=0.7pt] (0,0)-- (-0.7340132902094361,2.414285714285714);
\draw [line width=0.7pt] (-0.7340132902094361,2.414285714285714)-- (0,1.4);
\draw [line width=0.7pt] (0,1.4)-- (0,0);
\draw [line width=0.7pt] (0,1.4)-- (3.3867606432599557,6.565282921073598);
\draw [line width=0.7pt] (3.3867606432599557,6.565282921073598)-- (-0.7340132902094361,2.414285714285714);
\draw [line width=0.7pt] (3.8850696994038274,3.124568798471209)-- (6,0);
\draw [line width=0.7pt] (3.8850696994038274,3.124568798471209)-- (6.002261946852696,5.300000482679559);
\draw [line width=0.7pt] (6.002261946852696,5.300000482679559)-- (6,0);
\draw [line width=0.7pt,dashed] (3.8850696994038274,3.124568798471209)-- (6.000597495395052,1.4000001275002611);
\draw [samples=50,domain=-0.99:0.99,rotate around={90:(0,1.4)},xshift=0cm,yshift=1.4cm,line width=0.4pt,dash pattern=on 1pt off 1pt] plot ({1.4*(-1-(\x)^2)/(1-(\x)^2)},{1.4*(-2)*(\x)/(1-(\x)^2)});
\draw [line width=0.4pt,dashed] (-0.7340132902094361,2.414285714285714)-- (-1.2157602309138298,-0.4542041255135986);
\draw [shift={(0,1.4)},line width=0.4pt, red]  plot[domain=2.1972355246961035:4.71238898038469,variable=\t]({1*0.3447089581954703*cos(\t r)+0*0.3447089581954703*sin(\t r)},{0*0.3447089581954703*cos(\t r)+1*0.3447089581954703*sin(\t r)});
\draw [shift={(0,1.4)},line width=0.4pt, blue]  plot[domain=2.1972355246961035:4.13203278054244,variable=\t]({1*0.27081028495986365*cos(\t r)+0*0.27081028495986365*sin(\t r)},{0*0.27081028495986365*cos(\t r)+1*0.27081028495986365*sin(\t r)});
\draw [line width=0.4pt,dashed] (0,1.4)-- (-1.2157602309138298,-0.4542041255135986);
\begin{scriptsize}
\coordinate [circle, fill=black, inner sep=0.7pt, label=270: {$x$}] (A1) at (0,0);
\coordinate [circle, fill=black, inner sep=0.7pt, label=0: {$p$}] (A1) at (0,1.4);
\coordinate [circle, fill=black, inner sep=0.7pt, label=180: {$y$}] (A1) at (-0.7340132902094361,2.414285714285714);
\coordinate [circle, fill=black, inner sep=0.7pt, label=270: {$z$}] (A1) at (3.3867606432599557,6.565282921073598);
\coordinate [circle, fill=black, inner sep=0.7pt, label=270: {$x'$}] (A1) at (6,0);
\coordinate [circle, fill=black, inner sep=0.7pt, label=180: {$y'$}] (A1) at (3.8850696994038274,3.124568798471209);
\coordinate [circle, fill=black, inner sep=0.7pt, label=90: {$z'$}] (A1) at (6.002261946852696,5.300000482679559);
\coordinate [circle, fill=black, inner sep=0.7pt, label=0: {$p'$}] (A1) at (6.000597495395052,1.4000001275002611);
\coordinate [circle, fill=black, inner sep=0.7pt, label=180: {$\tilde{x}$}] (A1) at (-1.2157602309138298,-0.4542041255135986);
\end{scriptsize}
\end{tikzpicture}
\end{center}
\caption{A visual approach in the spirit of the original version of Alexandrov's lemma.}
\label{visual alexlem}
\end{figure}
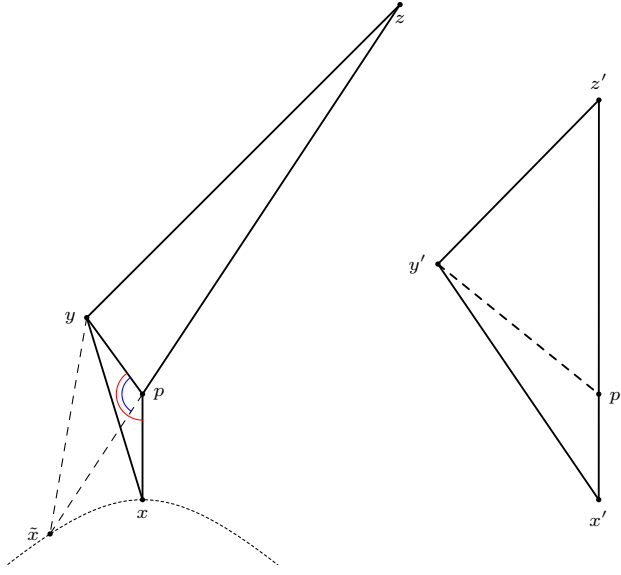
Since $[x,z] \cap [p,y] = \emptyset$, the segment $[\tilde{x},p]$ lies between the segments $[x,p]$ and $[p,y]$ (in the sense of Remark \ref{between}). This already implies $\angle xpy \leq \angle \tilde{x}py$.
To see this, we observe that $\angle xpy \leq \angle \tilde{x}py$ is equivalent to
\begin{equation}
\label{scalar product inequality}
\langle \gamma_{p \tx}'(0) - \gamma_{p x}'(0), \gamma_{p y}'(0) \rangle \geq 0.
\end{equation}
Clearly, $\gamma_{p \tx}'(0) - \gamma_{p x}'(0)$ is a spacelike vector as the difference of two past directed timelike vectors of the same length. Recall that the scalar product with a spacelike vector $v$ is non-negative if and only if the other vector lies in the same half-space as $v$ generated by the normal space of $v$. 
For better visualization, we apply a Lorentz transformation to view $\gamma_{p \tx}'(0) - \gamma_{p x}'(0)$ as a horizontal vector and $p$ as the origin, see Figure \ref{half space}. 
Then the normal space of $\gamma_{p \tx}'(0) - \gamma_{p x}'(0)$ is 
a vertical line through 0. 
From a Euclidean point of view, $\gamma_{p \tx}'(0) - \gamma_{p x}'(0)$ is at an angle of ninety degrees and symmetric with respect to its normal space. 
Since $p \in [\tx,z]$, it follows that $\gamma_{p x}'(0)$ and $\gamma_{p \tx}'(0)$ differ by a Euclidean reflection with respect to the normal space. Together, this implies that $\gamma_{p z}'(0)$ and $\gamma_{p x}'(0)$ lie in the same half-space. \\

Now keep in mind that these directional vectors originally come from the sides of two triangles which together yield a concave situation without self-intersection. In this way, we not only obtain that $\gamma_{py}'(0)$ does not lie between $\gamma_{p x}'(0)$ and $\gamma_{p z}'(0)$, but also that $\gamma_{py}'(0)$ has to lie between the mirrored versions of these vectors, that is between $\gamma_{p \tx}'(0)$ and $-\gamma_{p x}'(0)$. Thus, it is clear that both entries in (\ref{scalar product inequality}) are in the same half-space with respect to the normal space of the spacelike entry, and so the inequality is true.
\begin{figure}
\begin{center}
\begin{tikzpicture}[line cap=round,line join=round,>=triangle 45,x=1cm,y=1cm]
\draw [line width=0.5pt,dotted] (0,0)-- (-1,-2);
\draw [line width=0.7pt] (0,0)-- (1,-2);
\draw [line width=0.7pt] (0,0)-- (1.5,3);
\draw [line width=0.5pt,dotted] (0,0)-- (-1,2);
\draw [line width=0.7pt, color=blue] (1,-2)-- (-1,-2);
\draw [line width=0.7pt] (-0.06346095792292286,-2) -- (0,-1.9238468504924928);
\draw [line width=0.7pt] (-0.06346095792292286,-2) -- (0,-2.0761531495075083);
\draw [line width=0.7pt] (0.06346095792292286,-2) -- (0.12692191584584572,-1.9238468504924928);
\draw [line width=0.7pt] (0.06346095792292286,-2) -- (0.12692191584584572,-2.0761531495075083);
\draw [line width=0.7pt] (-0.1903828737687686,-2) -- (-0.12692191584584572,-1.9238468504924928);
\draw [line width=0.7pt] (-0.1903828737687686,-2) -- (-0.12692191584584572,-2.0761531495075083);
\draw [line width=0.7pt, color=blue] (0,0)-- (-1,0.2);
\draw [line width=0.7pt] (-0.5622285890000532,0.11244571780001074) -- (-0.48506513863998807,0.17467430680006382);
\draw [line width=0.7pt] (-0.5622285890000532,0.11244571780001074) -- (-0.5149348613600131,0.025325693199936217);
\draw [line width=0.7pt] (-0.4377714109999474,0.08755428219999027) -- (-0.3606079606398822,0.14978287120004335);
\draw [line width=0.7pt] (-0.4377714109999474,0.08755428219999027) -- (-0.3904776833599072,0.0004342575999157513);
\draw [line width=0.7pt] (-0.6866857670001587,0.13733715340003216) -- (-0.6095223166400935,0.19956574240008526);
\draw [line width=0.7pt] (-0.6866857670001587,0.13733715340003216) -- (-0.6393920393601185,0.05021712879995765);
\draw [line width=0.7pt,dashed] (-1,0.2)-- (1.5,3);
\draw [line width=0.7pt,dashed] (-1,0.2)-- (1,-2);
\draw [line width=0.7pt,color=red] (0,3.5)-- (0,-2.5);
\begin{scriptsize}
\coordinate [circle, fill=black, inner sep=0.7pt, label=0: {$0$}] (A1) at (0,0);
\coordinate [circle, fill=black, inner sep=0.7pt, label=180: {$\gamma_{p \tx}'(0)$}] (A1) at (-1,-2);
\coordinate [circle, fill=black, inner sep=0.7pt, label=0: {$\gamma_{p x}'(0)$}] (A1) at (1,-2);
\coordinate [circle, fill=black, inner sep=0.7pt, label=0: {$\gamma_{p z}'(0)$}] (A1) at (1.5,3);
\coordinate [circle, fill=black, inner sep=0.7pt, label=180: {$\gamma_{p y}'(0)$}] (A1) at (-1,0.2);
\end{scriptsize}
\end{tikzpicture}
\end{center}
\caption{The two vectors in the scalar product lie in the same half-space.}
\label{half space}
\end{figure}
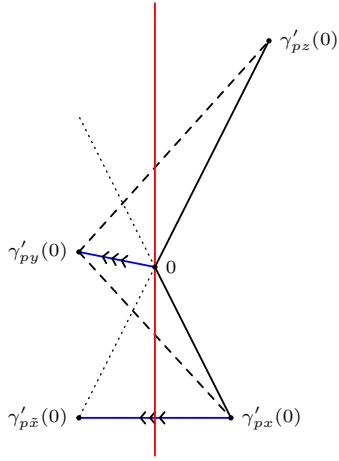
Now consider the triangles $\Delta(x,p,y)$ and $\Delta(\tilde{x},p,y)$. By construction we have $|px|_{\pm}=|p\tilde{x}|_{\pm}$, so they have two sides of equal length since they have the segment $[p,y]$ in common. Moreover, since $\angle \tilde{x}py \geq  \angle xpy$, the hinge lemma gives $|\tilde{x}y|_{\pm} \leq |xy|_{\pm}$.
We continue with the triangles $\Delta(\tilde{x},y,z)$ and $\Delta(x',y',z')$. Again, we have two pairs of equal length since $|yz|_{\pm}=|y'z'|_{\pm}$ and $|x'z'|_{\pm}=|xp|_{\pm}+|pz|_{\pm}=|\tilde{x}p|_{\pm}+|pz|_{\pm}=|\tilde{x}z|_{\pm}$, where the last equality holds since $p$ lies on the segment $[\tilde{x},z]$. For the third side we know $|\tilde{x}y|_{\pm} \leq |xy|_{\pm}=|x'y'|_{\pm}$ by the above considerations. Thus, the hinge lemma implies $\angle \tilde{x}zy \geq \angle x'z'y'$. And since $p$ and $p'$ correspond to each other on the segments $[\tilde{x},z]$ and $[x',z']$, respectively, we also have $\angle pzy \geq \angle p'z'y'$. \\

Analogously, we find a point $\tilde{z}$ that lies on the extension of $[x,p]$ at distance $\tau(p,z)$ from $p$. A similar argument then implies $\angle pxy \geq \angle y'x'z'$. Finally, the triangles $\Delta(x,p,y)$ and $\Delta(x',p',y')$ have two sides with equal length and we know $\angle pxy \geq \angle p'x'y'$. Consequently, we obtain $|py|_{\pm} \leq |p'y'|_{\pm}$ by the hinge lemma. \\

At last, note that if $[x,z] \cap [p,y] = \{p\}$, then $[x,z]$ in fact is composed of the two sides $[x,p]$ and $[p,z]$ in the triangles and hence $\Delta(x,y,z)$ is isometric to $\Delta(x',y',z')$, so equality in all inequalities in the statement follows immediately. Also if $[x,z] \cap [p,y] = \{q\} \neq \{p\}$, then the procedure from above causes $[x,p]$ to be between $[p,y]$ and $[p,\tilde{x}]$. 
Then the same calculation yields $\angle \tx p y \leq \angle xpy$ and consequently all the following inequalities are reversed. Any of the inequalities being an equality forces the others to be equalities as well, which then implies $q=p$. 
\end{proof}
In the statement above we assumed the configuration of triangles to be such that the subdivision happens along the longest side of $\Delta(x',y',z')$. For a proper application we need to guarantee the desired behaviour in the other cases as well. However, this can be done analogously. We give a rough sketch of the proof.
\begin{lem}[Alexandrov's Lemma, other constellations]
\label{alexlem2}
Let $\Delta(x,p,z)$ and $\Delta(p,y,z)$ be two triangles in $M_K$ such that $[x,p], [x,z], [p,y]$ as well as $[y,z]$ are timelike and the triangles 
are arranged on opposite sides of the geodesic line extending the segment $[p,z]$, see Figure \ref{alexlemfig2}. 
Let $\Delta(x',y',z')$ be a timelike triangle such that $\tau(x',y')=\tau(x,p)+\tau(p,y), \tau(y',z')=\tau(y,z)$ and $\tau(x',z')=\tau(x,z)$.
Then $\angle pzy \geq \angle p'z'y'$. 
If $[x,y] \cap [p,z] = \emptyset$, then $\angle pxz \geq \angle p'x'z', \angle pxy \geq \angle p'x'y'$ and $|py|_{\pm} \leq |p'y'|_{\pm}$.
\end{lem}
\begin{figure}
\begin{tikzpicture}[line cap=round,line join=round,>=triangle 45,x=1cm,y=1cm]
\draw (7,0)-- (7,0.7);
\draw (7,0.7)-- (8.704211470062866,3.892857142857137);
\draw (8.704211470062866,3.892857142857137)-- (7,0);
\draw (6.5323077403855505,1.5418646267086404)-- (7,0.7);
\draw (6.5323077403855505,1.5418646267086404)-- (8.704211470062866,3.892857142857137);
\draw (9.6883515793861,1.915649434423587)-- (10.99591142719706,0);
\draw (9.6883515793861,1.915649434423587)-- (10.992258307520794,3.500001906468535);
\draw (10.992258307520794,3.500001906468535)-- (10.99591142719706,0);
\draw [dashed] (10.342131503291583,0.9578247172117903)-- (10.992258307520794,3.500001906468535);
\begin{scriptsize}
\coordinate [circle, fill=black, inner sep=0.7pt, label=0: {$x'$}] (A1) at (10.99591142719706,0);

\coordinate [circle, fill=black, inner sep=0.7pt, label=190: {$p'$}] (A1) at (10.342131503291583,0.9578247172117903);

\coordinate [circle, fill=black, inner sep=0.7pt, label=270: {$x$}] (A1) at (7,0);

\coordinate [circle, fill=black, inner sep=0.7pt, label=230: {$p$}] (A1) at (7,0.7);

\coordinate [circle, fill=black, inner sep=0.7pt, label=0: {$z$}] (A1) at (8.704211470062866,3.892857142857137);

\coordinate [circle, fill=black, inner sep=0.7pt, label=180: {$y$}] (A1) at (6.5323077403855505,1.5418646267086404);

\coordinate [circle, fill=black, inner sep=0.7pt, label=180: {$y'$}] (A1) at (9.6883515793861,1.915649434423587);

\coordinate [circle, fill=black, inner sep=0.7pt, label=90: {$z'$}] (A1) at (10.992258307520794,3.500001906468535) circle (2pt);
\end{scriptsize}
\end{tikzpicture}
\caption{The other configuration of Alexandrov's Lemma.}
\label{alexlemfig2}
\end{figure}
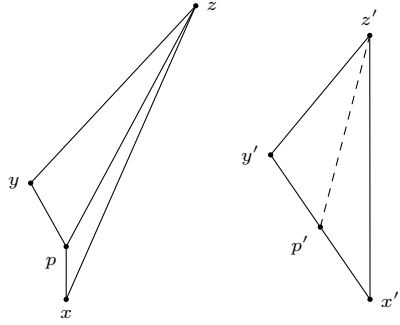
\begin{proof}
Note that since $p \ll y \ll z$, also $[p,z]$ is timelike and so we are in fact dealing with two timelike triangles.
We have $\tau(x',y')=\tau(x,p)+\tau(p,y) \leq \tau(x,y)$, i.e., $|x'y'|_{\pm} \geq |xy|_{\pm}$, hence $\angle pzy \geq \angle p'z'y'$ by the hinge lemma. 
As for the other inequalities, extend the segment $[p,x]$ to obtain a point $\ty$ on this extension such that $\tau(p,y)=\tau(p,\ty)$, cf. Figure \ref{visual alexlem}. Then as in Figure \ref{half space}, we obtain $\angle ypz \leq \angle \ty pz$. Considering the triangles $\Delta(p,\ty,z)$ and $\Delta(p,y,z)$, we infer $|yz|_{\pm} \geq |\ty z|_{\pm}$ from the hinge lemma. Then consider $\Delta(x,\ty,z)$ and $\Delta(x',y',z')$ and obtain $\angle pxz \geq \angle p'x'z'$. We can argue similarly to obtain $\angle pyz \geq \angle p'y'z'$. Finally, $|pz|_{\pm} \leq |p'z'|_{\pm}$ again by the hinge lemma.
\end{proof}
\subsection{The gluing lemma}
We now formulate and prove the gluing lemma for timelike triangles. This really is the main tool in the proof of the gluing theorem, both in the metric case and in the Lorentzian setting. The proof is rather long and quite technical.
\begin{lem}[Gluing lemma for timelike triangles, case I]
\label{gluinglemma}
Let $X$ be a Lorentzian pre-length space and let $U \subseteq X$ be a subset that satisfies $(i)$ and $(ii)$ in the definition for a comparison neighbourhood in $X$, cf. Definition \ref{tl curv bounds}. That is, $\tau|_{U \times U}$ is finite and continuous and for all $x,y \in U$ with $x \ll y$ there is a causal curve contained in $U$ with length $\tau(x,y)$. Let $K \in \R$ and let  $T_3:= \Delta(x,y,z)$ be a timelike triangle in $U$ satisfying size bounds for $M_K$. Let $p \in [x,z]$ such that $p \ll y$ (or $y \ll p$). In other words, $T_1 := \Delta(x,p,y)$ and $T_2 := \Delta(p,y,z)$ are again timelike triangles (if $y \ll p$ then the order of the points changes), see Figure \ref{gluing figure}. Let $\bT_1:=\Delta(\bx,\bp,\by)$ and $\bT_2:=\Delta(\bp,\by,\bz)$ be comparison triangles for $T_1$ and $T_2$ in $M_K$, respectively. Suppose $T_1$ and $T_2$ satisfy timelike curvature bounds from above for $K$, i.e., for all $a,b \in T_i$ and corresponding comparison points $\ba,\bb \in \bT_i, i=1,2$ we have $\tau(a,b) \geq \bar{\tau}(\ba,\bb)$. Then $T_3$ satisfies the same timelike curvature bound from above.
\end{lem}
\begin{figure}
\begin{center}
\begin{tikzpicture}
\draw (0,0)-- (-0.7340132902094361,2.414285714285714);
\draw (-0.7340132902094361,2.414285714285714)-- (0,1.4);
\draw (0,1.4)-- (0,0);
\draw (0,1.4)-- (3.3867606432599557,6.565282921073598);
\draw (3.3867606432599557,6.565282921073598)-- (-0.7340132902094361,2.414285714285714);
\draw (3.8850696994038274,3.124568798471209)-- (6,0);
\draw (3.8850696994038274,3.124568798471209)-- (6.002261946852696,5.300000482679559);
\draw (6.002261946852696,5.300000482679559)-- (6,0);
\draw [dashed] (6.000597495395052,1.4000001275002611) -- (3.8850696994038274,3.124568798471209);

\draw (-4,0) .. controls (-4.5,2) and (-5,3) .. (-5,3.5);
\draw (-5,3.5) .. controls (-4,4) and (-3,5) .. (-2.5,5.5);

\begin{scriptsize}
\draw (-4,0) .. controls (-3,3) and (-2,4) .. (-2.5,5.5) node (A)[circle, fill=black,inner sep=0.7pt,pos=0.15,label=right:$p$]{};;
\draw (A) .. controls (-4.5,2) and (-4,3) .. (-5,3.5);
\coordinate [circle, fill=black, inner sep=0.7pt, label=270: {$x$}] (A1) at (-4,0);
\coordinate [circle, fill=black, inner sep=0.7pt, label=180: {$y$}] (A1) at (-5,3.5);
\coordinate [circle, fill=black, inner sep=0.7pt, label=90: {$z$}] (A1) at (-2.5,5.5);
\coordinate [circle, fill=black, inner sep=0.7pt, label=270: {$\bx$}] (A1) at (0,0);
\coordinate [circle, fill=black, inner sep=0.7pt, label=0: {$\bp$}] (A1) at (0,1.4);
\coordinate [circle, fill=black, inner sep=0.7pt, label=180: {$\by$}] (A1) at (-0.7340132902094361,2.414285714285714);
\coordinate [circle, fill=black, inner sep=0.7pt, label=90: {$\bz$}] (A1) at (3.3867606432599557,6.565282921073598);
\coordinate [circle, fill=black, inner sep=0.7pt, label=270: {$\bx'$}] (A1) at (6,0);
\coordinate [circle, fill=black, inner sep=0.7pt, label=180: {$\by'$}] (A1) at (3.8850696994038274,3.124568798471209);
\coordinate [circle, fill=black, inner sep=0.7pt, label=90: {$\bz'$}] (A1) at (6.002261946852696,5.300000482679559);
\coordinate [circle, fill=black, inner sep=0.7pt, label=0: {$\bp'$}] (A1) at (6.000597495395052,1.4000001275002611);
\end{scriptsize}
\end{tikzpicture}
\end{center}
\caption{A timelike triangle in $X$ subdivided into two timelike triangles, the comparison triangles for the smaller triangles and the comparison triangle for the big triangle.}
\label{gluing figure}
\end{figure}
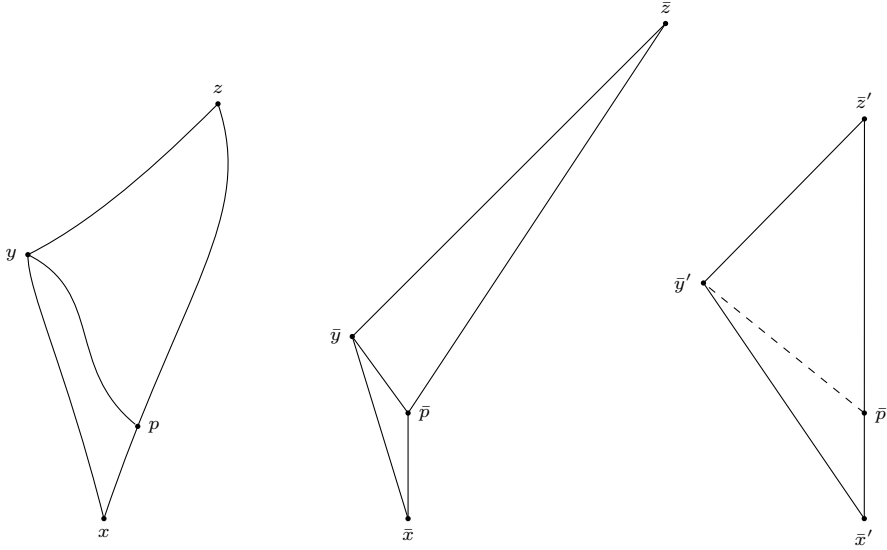
\begin{proof}
Realize the comparison triangles for $T_1$ and $T_2$ in such a way that they share the timelike geodesic segment $[\bp,\by]$ and such that $\bx$ and $\bz$ lie on opposite sides of this segment (as in Alexandrov's lemma). 
Note that because of the size bounds, either $[\bp,\by]$ and $[\bx,\bz]$ intersect in a single point or they do not intersect at all. If $[\bp,\by] \cap [\bx,\bz] =\{p\}$, then as in Lemma \ref{alexlem}, $\bT_1$ and $\bT_2$ together already form a comparison triangle $\bT_3$ for $T_3$ and we are done. The cases of comparing points which are not immediate from this assumption will be covered later in greater generality. Conversely, if $[\bp,\by] \cap [\bx,\bz] =\{q\} \neq \{p\}$, then $\btau(\bx,\bz)=\btau(\bx,\bq)+\btau(\bq,\bz)$ and, since $\bq \in [\bp,\by]$, $\bq$ is a comparison point for some $q \in [p,y]$. Moreover, $x$ and $q$ are on the sides of $T_1$ and $q$ and $z$ are on the sides of $T_2$. Since $T_1$ and $T_2$ satisfy timelike curvature bounds, we compute
\begin{equation}
\btau(\bx,\bz)=\btau(\bx,\bq)+\btau(\bq,\bz) \leq \tau(x,q)+\tau(q,z) \leq \tau(x,z).\footnote{It may a priori not be clear that $x \leq q \leq z$ so one can apply the reverse triangle inequality, but this can be seen as follows: we have $\bq \in [\bx,\bz]$, hence $\bx \ll \bq \ll \bz$ and so $0 < \btau(\bx,\bq) \leq \tau(x,q)$ by the curvature condition and thus $x \ll q$. We get a similar estimate for $q$ and $z$.}
\end{equation}
But this is in contradiction to
\begin{equation}
\tau(x,z)=\tau(x,p)+\tau(p,z)=\btau(\bx,\bp)+\btau(\bp,\bz)<\btau(\bx,\bz), 
\end{equation}
so such an intersection cannot occur with an upper curvature bound.
Thus, $[\bp,\by] \cap [\bx,\bz] =\emptyset$ is the only interesting case we have to consider. Moreover, as $\tau(x,p)+\tau(p,z)=\tau(x,z)\geq \tau(x,y)+\tau(y,z)$ by the reverse triangle inequality, we can realize the situation in $M_K$ as in Figure \ref{gluing figure}. \\

What follows now are several applications of Alexandrov's Lemma \ref{alexlem}. First, we ``bend'' $\bT_1$ and $\bT_2$ in such a way that they form a comparison triangle for $T_3$, cf.\ Figure \ref{visual alexlem}. More precisely, let $\bT_3:=\Delta(\bx',\by',\bz')$, where $\btau(\bx',\by')=\btau(\bx,\by)=\tau(x,y), \btau(\by',\bz') = \btau(\by,\bz) = \tau(y,z)$ and $\btau(\bx',\bz')=\btau(\bx,\bp)+\btau(\bp,\bz)=\tau(x,p)+\tau(p,z)=\tau(x,z)$. 
When talking about the comparison triangles for $T_1$ and $T_2$ simultaneously, it will be convenient to denote their union, which is a quadrilateral, by $\bT_1 \cup \bT_2$. That is, by $\ba \in \bT_1 \cup \bT_2$ we mean a point which belongs either to $\bT_1$ or $\bT_2$ (or both if $\ba \in [\bp,\by]$).
We distinguish several cases, depending on which sides the points lie on. Note that for any point on $\bT_3$, we can find a ``comparison point'' in either $\bT_1$ or $\bT_2$, i.e., a point on the corresponding side of equal time separation to the endpoints (the common edge $[\bp,\by]$ of $\bT_1$ and $\bT_2$ is the only one not (isometrically) transferred to $\bT_3$). 
We choose two points $\ba', \bb' \in \bT_3$ and check all possible configurations.
The general idea is to at first show that time separation in $\bT_3$ is even smaller than in $\bT_1 \cup \bT_2$, i.e., $\btau(\ba,\bb) \geq \btau(\ba',\bb')$, and then relating this inequality to time separations in $X$, i.e., $\tau(a,b) \geq \btau(\ba,\bb)$, thus ensuring the desired curvature bound. 
We may assume without loss of generality that $\ba' \ll \bb'$ always holds, since otherwise there is nothing to show. 
Since all these (sub-)case distinctions may become a bit confusing, we try to give a ``to-do list'' summarizing everything. Note that some descriptions might only make sense when reading the proof of the respective cases. The cases (B) and (C) are entirely analogous, hence we omit the descriptions in (C).
\begin{itemize}
\item[(A)] $\ba' \in [\bx',\by']$ and $\bb' \in [\by',\bz']$ (both points on short sides): this case is easy and also the only one where we do not need any subcases.
\item[(B)] $\ba' \in [\bx',\by']$ and $\bb' \in [\bx',\bz']$ (one point on short and long side each): here, we distinguish whether $\ba'$ and $\bb'$ are in the same triangle or not.
\begin{itemize}
\item[(1)] $\bb' \in [\bx',\bp']$ (same triangle): this case is easy and very similar to (A).
\item[(2)] $\bb' \in [\bp',\bz']$ (different triangle): this needs yet another distinction, namely whether the connecting segment $[\ba,\bb]$ stays inside $\bT_1 \cup \bT_2$ or not.
\begin{itemize}
\item[(i)] $[\ba,\bb]$ stays inside comparison situation: we construct several subtriangles and use the law of cosines.
\item[(ii)] $[\ba,\bb]$ leaves comparison situation: we improve the bound on $\btau(\ba',\bb')$ by taking a detour through $\bp$.
\end{itemize} 
\end{itemize}
\item[(C)] $\ba' \in [\by',\bz']$ and $\bb' \in [\bx',\bz']$ (one point on short and long side each)
\begin{itemize}
\item[(1)] $\bb' \in [\bp',\bz']$ (same triangle)
\item[(2)] $\bb' \in [\bx',\bp']$ (different triangle)
\begin{itemize}
\item[(i)] $[\ba,\bb]$ stays inside comparison situation
\item[(ii)] $[\ba,\bb]$ leaves comparison situation
\end{itemize} 
\end{itemize}
\end{itemize}
Finally, observe that switching the labels or assuming $\bb' \ll \ba'$ does not change the proof at all. Hence this is a complete list that covers all possible configurations. \\

\underline{$\ba' \in [\bx',\by']$ and $\bb' \in [\by',\bz']$ (A)}: 
In this first case $\ba' \ll \bb'$ holds anyways.
We find comparison points $\ba \in [\bx,\by]$ and $\bb \in [\by,\bz]$, i.e., $\btau(\bx,\ba)=\btau(\bx',\ba')$ and $\btau(\by,\bb)=\btau(\by',\bb')$. Consider the triangles $\Delta(\bx,\by,\bz)$ and $\Delta(\bx',\by',\bz')$. In these triangles, two sidelengths are the same, namely $|\bx \by|_{\pm}=|\bx' \by' |_{\pm}$ and $|\by \bz|_{\pm}=|\by' \bz' |_{\pm}$. 
For the third side, we have $\btau(\bx',\bz')=\btau(\bx,\bp)+\btau(\bp,\bz) \leq \btau(\bx,\bz)$, and hence $|\bx \bz|_{\pm} \leq |\bx' \bz' |_{\pm}$. Then by the hinge lemma we infer $\angle \bx \by \bz \geq \angle \bx' \by' \bz'$, exactly as in Lemma \ref{alexlem}. 
Now consider the ``smaller'' triangles $\Delta(\ba,\by,\bb)$ and $\Delta(\ba',\by',\bb')$, i.e., instead of the sides $[\bx,\by]$ and $[\by,\bz]$ we consider the (from the point of view of $\by$) initial segments $[\ba,\by]$ and $[\by,\bb]$ (and the same in the other triangle). Again, two side lengths are pairwise equal. We want to use the hinge lemma in the other direction to obtain estimates on the third side.
This is easily possible since $\angle \ba \by \bb$ and $\angle \ba' \by' \bb'$ are equal multiples of $\angle \bx \by \bz$ and $\angle \bx' \by' \bz'$, respectively.
Thus, we obtain $|\ba \bb|_{\pm} \leq |\ba' \bb'|_{\pm}$ and hence $\btau(\ba,\bb) \geq \btau(\ba',\bb')$. \\

We finished the case were both points lie on short sides. We are left with considering pairs of points where one is on the longest side and the other on a short side. These two possibilities clearly are analogous (since they only differ by time orientation), so we will only consider case (B) explicitly. In each of these cases, however, there are subcases which are respectively similar as well. Namely, we have to distinguish whether the point on the long side is chronologically before or after $\bp'$, i.e., if it is in $[\bx',\bp']$ or in $[\bp',\bz']$. In other words, this distinction tells us if the two points originate from the same triangle or not.
Before returning to the proof, observe that by Lemma \ref{alexlem} we have $\btau(\bp,\by) \geq \btau(\bp',\by')$ as well as $\angle \bp \bx \by \geq \angle \bp' \bx' \by'$ and $\angle \bp \bz \by \geq \angle \bp' \bz' \by'$. \\

\underline{$\ba' \in [\bx',\by']$ and $\bb' \in [\bx',\bp'] \subseteq [\bx',\bz']$ (B.1)}: 
Here, both points are from $\bT_1$.
Consider the smaller triangles $\Delta(\bx,\ba,\bb)$ and $\Delta(\bx',\ba',\bb')$. 
We have as in the case (A) above that $\angle \ba \bx \bp$ and $\angle \ba' \bx' \bp'$ are equal multiples of $\angle \by \bx \bp$ and $\angle \by' \bx' \bp'$, respectively. Two sides are of equal length by construction. Thus, we infer $|\ba \bb|_{\pm} \leq |\ba' \bb'|_{\pm}$ from the hinge lemma and hence $\btau(\ba,\bb) \geq \btau(\ba',\bb')$. 
Note that in this case (and also in (C.1)), since both points originate from the same triangles, we immediately obtain $\tau(a,b) \geq \btau(\ba,\bb)$ as well. \\

Looking at the list from above, we are now in the case (B.2), where we have to make yet another distinction. The extension of $[\bp,\bz]$ meets $[\bx,\by]$ in a unique point, denote it by $\bq$.
We have to distinguish whether $\ba$ lies chronologically before or after $\bq$. This distinction in particular tells us if $[\ba,\bb]$ lies inside of $\bT_1 \cup \bT_2$ or not. We first cover the case where this segment lies inside the triangles (the other case requires even more extra work). \\

\underline{$\ba' \in [\bq',\by'] \subseteq [\bx',\by']$ and $\bb' \in [\bp',\bz']\subseteq [\bx',\bz']$ (B.2.i)}: 
We try to construct a triangle in $\bT_1 \cup \bT_2$ that has both $\ba$ and $\bb$ as vertices and somehow inherits enough properties so that we can deduce the claim, see Figure \ref{case B}.
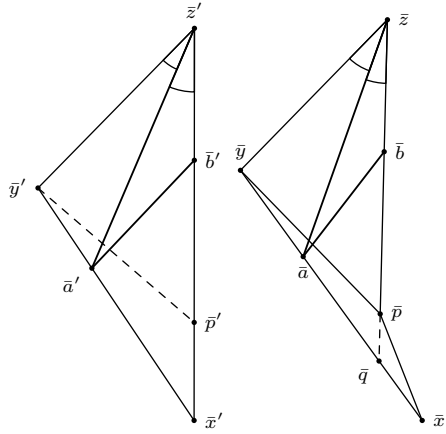
\begin{figure}
\begin{center}
\begin{tikzpicture}[line cap=round,line join=round,>=triangle 45,x=1cm,y=1cm]
\clip(0.4909702706196152,-0.25191180407216) rectangle (9.142544659670762,5.786464510603467);

\draw [shift={(3.5047761204274037,5.200002193396301)},line width=0.5pt] (0,0) -- (-134.1781543193275:0.5903485409608712) arc (-134.1781543193275:-112.9821786233633:0.5903485409608712) -- cycle;

\draw [shift={(3.5047761204274037,5.200002193396301)},line width=0.5pt] (0,0) -- (-112.9821786233633:0.8433550585155303) arc (-112.9821786233633:-90.05262525972239:0.8433550585155303) -- cycle;

\draw [shift={(6.042812190725555,5.31366188908918)},line width=0.5pt] (0,0) -- (-134.0724602471525:0.7168517997382007) arc (-134.0724602471525:-109.43267256627885:0.7168517997382007) -- cycle;

\draw [shift={(6.042812190725555,5.31366188908918)},line width=0.5pt] (0,0) -- (-109.43267256627885:0.8433550585155303) arc (-109.43267256627885:-91.3982148806444:0.8433550585155303) -- cycle;

\draw [line width=0.5pt] (6.5,0)-- (4.109832438913905,3.3170620992179582);
\draw [line width=0.5pt] (4.109832438913905,3.3170620992179582)-- (5.947591456628026,1.4124996278903383);
\draw [line width=0.5pt] (5.947591456628026,1.4124996278903383)-- (6.5,0);
\draw [line width=0.5pt] (5.947591456628026,1.4124996278903383)-- (6.042812190725555,5.31366188908918);
\draw [line width=0.5pt] (6.042812190725555,5.31366188908918)-- (4.109832438913905,3.3170620992179582);
\draw [line width=0.5pt] (1.4473878102759392,3.082728791412537)-- (3.5,0);
\draw [line width=0.5pt] (1.4473878102759392,3.082728791412537)-- (3.5047761204274037,5.200002193396301);

\draw [line width=0.5pt] (3.5047761204274037,5.200002193396301)-- (3.5,0);
\draw [line width=0.5pt,dashed] (1.4473878102759392,3.082728791412537)-- (3.5011940301068507,1.3000005483490746);

\draw [line width=0.7pt] (3.5031692629291817,3.4505357295089287)-- (2.155678780018674,2.0189774524823294);
\draw [line width=0.7pt] (4.93460299623817,2.1724498131203105)-- (6.0000979828528695,3.5636747948996943);
\draw [line width=0.7pt] (4.93460299623817,2.1724498131203105)-- (6.042812190725555,5.31366188908918);
\draw [line width=0.7pt] (2.155678780018674,2.0189774524823294)-- (3.5047761204274037,5.200002193396301);

\draw [line width=0.5pt, dashed] (5.947591456628026,1.4124996278903383)--(5.932343379416936,0.787790903099117);

\begin{scriptsize}
\coordinate [circle, fill=black, inner sep=0.7pt, label=0: {$\bx$}] (A1) at (6.5,0);

\coordinate [circle, fill=black, inner sep=0.7pt, label=0: {$\bp$}] (A1) at (5.947591456628026,1.4124996278903383);

\coordinate [circle, fill=black, inner sep=0.7pt, label=90: {$\by$}] (A1) at (4.109832438913905,3.3170620992179582);

\coordinate [circle, fill=black, inner sep=0.7pt, label=0: {$\bz$}] (A1) at (6.042812190725555,5.31366188908918);

\coordinate [circle, fill=black, inner sep=0.7pt, label=0: {$\bx'$}] (A1) at (3.5,0);

\coordinate [circle, fill=black, inner sep=0.7pt, label=180: {$\by'$}] (A1) at (1.4473878102759392,3.082728791412537);

\coordinate [circle, fill=black, inner sep=0.7pt, label=90: {$\bz'$}] (A1) at (3.5047761204274037,5.200002193396301);

\coordinate [circle, fill=black, inner sep=0.7pt, label=0: {$\bp'$}] (A1) at (3.5011940301068507,1.3000005483490746);

\coordinate [circle, fill=black, inner sep=0.7pt, label=270: {$\ba$}] (A1) at (4.93460299623817,2.1724498131203105);

\coordinate [circle, fill=black, inner sep=0.7pt, label=200: {$\ba'$}] (A1) at (2.155678780018674,2.0189774524823294);

\coordinate [circle, fill=black, inner sep=0.7pt, label=0: {$\bb$}] (A1) at (6.0000979828528695,3.5636747948996943);

\coordinate [circle, fill=black, inner sep=0.7pt, label=0: {$\bb'$}] (A1) at (3.5031692629291817,3.4505357295089287);

\coordinate [circle, fill=black, inner sep=0.7pt, label=200: {$\bq$}] (A1) at (5.932343379416936,0.787790903099117);
\end{scriptsize}
\end{tikzpicture}
\end{center}
\caption{This configuration requires some additional construction steps.}
\label{case B}
\end{figure}
In the end, this will be the timelike triangle $\Delta(\ba,\bb,\bz)$. However, before this we need to estimate the hyperbolic angle $\mad_{\bz}(\bb,\ba)$. 
Recall that we have $\angle \by \bz \bp \geq \angle \by' \bz' \bp'$. As the two legs have the same time orientation and the same length, we infer $\mad_{\bz}(\by,\bp) \leq \mad_{\bz'}(\by',\bp')$, cf.\ Remark \ref{angle inequality}. Now consider the timelike triangles $\Delta(\ba,\by,\bz)$ and $\Delta(\ba',\by',\bz')$. They have two sides of pairwise equal length. Then $|\ba\bz|_{\pm} \leq |\ba'\bz'|_{\pm}$ easily follows from the hinge lemma. 
As both triangles are timelike, we can change our perspective in the sense that we go from signed lengths and the hinge lemma to (positive) time separation values and the Lorentzian law of cosines, cf.\ \cite[Theorem 3.1.3]{Kir18}.
We do this because the adjacent sides of 
$\angle \ba \bz \by$ and $\angle \ba' \bz' \by'$ 
do not have pairwise equal length.
Recall the following consequence of the Lorentzian law of cosines: fixing the two short sides in a timelike triangle and letting the longest side vary, \emph{any} hyperbolic angle is an increasing function in the length (understood as time separation) of the longest side.
In our case, $[\ba, \bz]$ and $[\ba',\bz']$ are the longest sides 
and so we obtain $\mad_{\bz}(\ba,\by) \geq \mad_{\bz'}(\ba',\by')$. 
Also, since 
$\mad_{\bz}(\by,\ba)+\mad_{\bz}(\ba,\bp) = \mad_{\bz}(\by,\bp) 
\leq \mad_{\bz'}(\by',\bp') = \mad_{\bz'}(\by',\ba') + \mad_{\bz'}(\ba',\bp')$, we must have $\mad_{\bz}(\ba,\bb)=\mad_{\bz}(\ba,\bp) \leq \mad_{\bz'}(\ba',\bp')=\mad_{\bz'}(\ba',\bb')$. 
Finally, consider $\Delta(\ba,\bb,\bz)$ and $\Delta(\ba',\bb',\bz')$. We know 
$B:=\btau(\bb,\bz) = \btau(\bb',\bz'), A:=\btau(\ba,\bz) \geq \btau(\ba',\bz')=:A'$ and 
$\omega := \mad_{\bz}(\ba,\bb) \leq \mad_{\bz'}(\ba',\bb') =: \omega'$ and we want to infer $C:=\btau(\ba',\bb') \leq \btau(\ba',\bb')=:C'$. In two steps, we want to transform one triangle into the other. In $\Delta(\ba',\bb',\bz')$, fix $A'$ and $B$ and decrease $\omega'$ to $\omega$. From the law of cosines, it easily follows that $C'$ increases as $\omega'$ decreases. In this way we obtain an intermediate triangle with sidelengths $A',B$ and a hyperbolic angle $\omega$ opposite of a side $\tilde{C} \geq C'$.
Now we want to increase $A'$ to $A$ while keeping $B$ and $\omega$ fixed so that $\tilde{C}$ \emph{changes} to $C$ (we do not yet know whether it increases or decreases).
But this is precisely the situation in Remark \ref{hingebehaviour}, and so we obtain $C \geq \tilde{C} \geq C'$, as desired. \\

So we are only left with the case (B.2.ii).
In almost all the cases we covered so far, relating the distance in $\bT_1 \cup \bT_2$ with the original distances in $X$ is not difficult. Either the points already stem from the same small triangle in $X$, e.g., $a \in [x,y], b \in [x,p]\subseteq [x,z]$ in (B.1), in which case we have $\tau(a,b) \geq \btau(\ba,\bb)$ by assumption. 
Or the points lie in different triangles but the connecting geodesic is entirely contained in the union of $\bT_1$ and $\bT_2$. For instance in (B.2.ii), the very last case we covered, for $\ba' \in [\bx,\bq] \subseteq [\bx',\by']$ and $\bb' \in [\bp',\bz']\subseteq [\bx',\bz']$ there exists an intersection point of $[\bp,\by]$ and $[\ba,\bb]$, call it $\br$. 
Then we compute
\begin{equation}
\btau(\ba',\bb')\leq \btau(\ba,\bb) = \btau(\ba, \br) + \btau( \br,\bb)\leq \tau(a,r)+\tau(r,b) \leq \tau(a,b).
\end{equation}
Illustratively, we divided the segment $[\ba,\bb]$ in two parts which both lie inside a single triangle and then apply the respective curvature conditions of $T_1$ and $T_2$. 
The reader may already have suspected that this unfortunately is not always possible. Take $\ba \in [\bx,\by], \bb \in [\by,\bz]$ in case (A), where $\ba$ and $\bb$ are close to $\bx$ and $\bz$, respectively. Then the connecting segment leaves the triangles. While $\btau(\ba,\bb) \geq \btau(\ba',\bb')$ is easy to see, we have no way of relating $\btau(\ba,\bb)$ with $\tau(a,b)$. Moreover, in the remaining case (B.2.ii) this is by assumption always the case. We rectify this situation by improving the bound on $\btau(\ba',\bb')$: in (B.2.ii) and the cases of (A) where the connecting segment leaves the triangle, we want to show 
\begin{equation}
\label{bound1}
\btau(\ba,\bp)+\btau(\bp,\bb) \geq \btau(\ba',\bb').
\end{equation}
To make the following already very lengthy computations a bit easier and also to do both cases kind of simultaneously, we assume for now $\bb'=\bz'$. After showing (\ref{bound1}) in this case, we let $\bb'$ vary. By moving $\bb'$ onto $[\by',\bz']$ we cover the remaining case of (A) and by moving $\bb'$ onto $[\bp',\bz']$ we cover (B.2.ii).
Remember that we denote by $\bq$ the point where the extension of $[\bp,\bz]$ meets $[\bx,\by]$. \\
 
\underline{$\ba \in [\bx,\bq]\subseteq [\bx,\by]$ and $\bb=\bz$ (B.2.ii \& A)}: 
If we can show (\ref{bound1}), then this implies $\tau(a,z) \geq \tau(a,p)+\tau(p,z) \geq \btau(\ba,\bp)+\btau(\bp,\bz) \geq \btau(\ba',\bz')$.
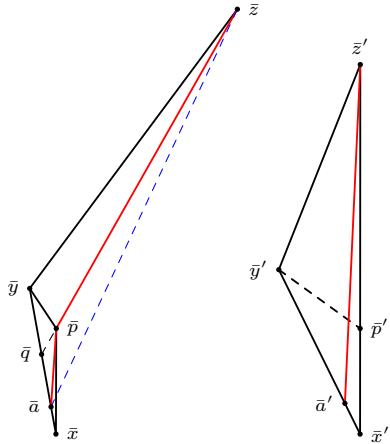
\begin{figure}
\begin{center}
\begin{tikzpicture}[line cap=round,line join=round,>=triangle 45,x=1cm,y=1cm]
\clip(-1.4673467191946998,-0.2537959677252801) rectangle (8.370496166452751,6.008419325580024);
\draw [line width=0.7pt] (-0.3448789520193921,1.9310467346871731)-- (0,0);
\draw [line width=0.7pt] (2.9312356251033593,2.1799672678845905)-- (4,0);
\draw [line width=0.7pt] (0,0)-- (0.004410691598607807,1.4000069479114658);
\draw [line width=0.7pt] (0.004410691598607807,1.4000069479114658)-- (-0.3448789520193921,1.9310467346871731);
\draw [line width=0.7pt] (-0.3448789520193921,1.9310467346871731)-- (2.3860206269701507,5.633453048225068);
\draw [line width=0.7pt,red] (2.3860206269701507,5.633453048225068)-- (0.004410691598607807,1.4000069479114658);
\draw [line width=0.7pt] (2.9312356251033593,2.1799672678845905)-- (4.002955230111009,4.900000891161655);
\draw [line width=0.7pt] (4.002955230111009,4.900000891161655)-- (4,0);
\draw [line width=0.7pt,dashed] (2.9312356251033593,2.1799672678845905)-- (4.000844351460288,1.400000254617616);
\draw [line width=0.7pt,red] (-0.06442575264226036,0.3607327688197249)-- (0.004410691598607807,1.4000069479114658);
\draw [line width=0.7pt,red] (3.8003474875843417,0.40723283095878643)-- (4.002955230111009,4.900000891161655);
\draw [line width=0.4pt,dashed, blue] (-0.06442575264226036,0.3607327688197249)-- (2.3860206269701507,5.633453048225068);
\draw [line width=0.4pt,dashed] (-0.1887233809327427,1.0567002317056584)-- (0.004410691598607807,1.4000069479114658);
\begin{scriptsize}
\coordinate [circle, fill=black, inner sep=0.7pt, label=0: {$\bx$}] (A1) at (0,0);

\coordinate [circle, fill=black, inner sep=0.7pt, label=0: {$\bp$}] (A1) at (0.004410691598607807,1.4000069479114658);

\coordinate [circle, fill=black, inner sep=0.7pt, label=180: {$\by$}] (A1) at (-0.3448789520193921,1.9310467346871731);

\coordinate [circle, fill=black, inner sep=0.7pt, label=0: {$\bz$}] (A1) at (2.3860206269701507,5.633453048225068);

\coordinate [circle, fill=black, inner sep=0.7pt, label=180: {$\bq$}] (A1) at (-0.1887233809327427,1.0567002317056584);

\coordinate [circle, fill=black, inner sep=0.7pt, label=180: {$\ba$}] (A1) at (-0.06442575264226036,0.3607327688197249);

\coordinate [circle, fill=black, inner sep=0.7pt, label=0: {$\bx'$}] (A1) at (4,0);

\coordinate [circle, fill=black, inner sep=0.7pt, label=180: {$\by'$}] (A1) at (2.9312356251033593,2.1799672678845905);

\coordinate [circle, fill=black, inner sep=0.7pt, label=180: {$\ba'$}] (A1) at (3.8003474875843417,0.40723283095878643);

\coordinate [circle, fill=black, inner sep=0.7pt, label=90: {$\bz'$}] (A1) at (4.002955230111009,4.900000891161655);

\coordinate [circle, fill=black, inner sep=0.7pt, label=0: {$\bp'$}] (A1) at (4.000844351460288,1.400000254617616);
\end{scriptsize}
\end{tikzpicture}
\caption{Improving the bound on $\btau(\ba',\bz')$.}
\label{worst case}
\end{center}
\end{figure}
In Figure \ref{worst case}, this effectively means that the detour through $\bp$ in $\bT_1 \cup \bT_2$ is still larger than the direct connection in $\bT_3$.
We proceed in an elementary yet effective way. We define a function that compares these two lengths and show that its sign does not change. 
Unfortunately, we have to show this separately for the different cases of curvature. \\

\underline{Minkowski space ($K=0$)}:
Recall that the points $\ba$ and $\ba'$ can be described as $\ba = \gamma_{\bx \by}(t)=ty+(1-t)x$ and $\ba' = \gamma_{\bx' \by'}(t)=ty'+(1-t)x'$, respectively. Set $\gamma_{\bx \by}(m)=:\bq \in [\bx,\by]$ and similarly $\gamma_{\bx' \by'}(m)=:\bq' \in [\bx',\by']$.
Then define $f:[0,m] \rar \R$ by 
\begin{equation}
f(t):= (\btau(\ba,\bp)+\btau(\bp,\bz))^2-\btau(\ba',\bz')^2.
\end{equation} 
We square these values solely for ease of computation. 
If $t=0$, then $\ba=\bx$ and $\ba'=\bx'$, and since by construction $\btau(\bx,\bp)+\btau(\bp,\bz)=\btau(\bx',\bz')$, we have $f(0)=0$. 
If $t=m$, then $\ba=\bq$ and since $\bq$ lies on the extension of the segment $[\bp,\bz]$, we have $\btau(\bq,\bp)+\btau(\bp,\bz)=\btau(\bq,\bz) \geq \btau(\bq',\bz')$, hence $f(m)\geq 0$. 
Thus, if we can show that $f'' \leq 0$, we obtain $f\geq 0$ and so $(\btau(\ba,\bp)+\btau(\bp,\bz))^2 \geq \btau(\ba',\bz')^2$ which clearly also implies $\btau(\ba,\bp)+\btau(\bp,\bz) \geq \btau(\ba',\bz')$. 
Also note that since $\ba \ll \bq \ll \bp$, all distances are timelike. For the calculation, we first observe that 
\begin{equation}
\gamma_{\ba \bp}'(0) = 
\bp -\bx + t(\bx-\by),
\end{equation}
and we get a similar expression for $\gamma_{\ba' \bz'}'(0)$.
After calculating the squares one observes that $\langle \gamma_{\bp \bz}'(0),\gamma_{\bp \bz}'(0)\rangle$ does not depend on $t$ at all (since it does not depend on the position of $\ba$), so we can ignore this value for the derivative of $f$. 	
The following expressions become quite lengthy, so we abbreviate $B:=\gamma_{\ba \bp}'(0), A:=\frac{d}{dt}\gamma_{\ba \bp}'(0)=\bx-\by, B':=\gamma_{\ba' \bz'}'(0)$ and $A':=\frac{d}{dt}\gamma_{\ba' \bz'}'(0)=\bx'-\by'$. Then $f$ reads 
\begin{equation}
f(t)= -\langle B,B\rangle -2|\bp\bz|_{\pm}\sqrt{-\langle B,B\rangle} -\langle \gamma_{\bp \bz}'(0),\gamma_{\bp \bz}'(0)\rangle +\langle B',B'\rangle,
\end{equation}
and
\begin{equation}
\frac{d}{dt}f(t)=-2\langle A,B \rangle -2|\bp\bz|_{\pm}\frac{-\langle A,B \rangle}{\sqrt{-\langle B,B \rangle}} +2\langle A',B'\rangle.
\end{equation}
We further compute the second derivative of the middle term:
\begin{equation}
\frac{d}{dt}\frac{-\langle A,B \rangle}{\sqrt{-\langle B,B \rangle}} = 
\frac{\langle A,A \rangle \langle B,B \rangle - \langle A,B \rangle ^2}{\sqrt{(-\langle B,B \rangle)^3}} \leq 0,
\end{equation}
where the inequality follows by the reverse Cauchy Schwarz inequality for timelike vectors, cf. \cite[Prop. 5.30 (i)]{on83}. 
For the remaining terms of $(\frac{d}{dt})^2f(t)$, we have
\begin{equation}
\frac{d}{dt}2\langle A,B \rangle=2\langle A,A\rangle=2\langle A',A'\rangle=\frac{d}{dt}2\langle A',B'\rangle,
\end{equation}
where the middle equality holds since by construction we have $\btau(\bx,\by)=\btau(\bx',\by')$. Thus, these two terms cancel out and the second derivative of $f$ consists only of a nonpositive term. Keep in mind that $-2|\bp\bz|_{\pm}$ is positive as $|\bp\bz|_{\pm}$ is negative. \\

\underline{De Sitter space ($K=1$)}: In de Sitter space, the situation is more involved. Here, we will need Lemma \ref{help}. In this case and the following case of anti-de Sitter space, we choose to omit the bars of points since we are anyways only working in $M_K$. Let $\gamma_v$ be the unique geodesic from $x$ to $y$, i.e., $\gamma_v(t):= \cosh(t)x+\sinh(t)v$, where $v$ is a timelike unit vector perpendicular to $x$. That is, $\gamma_v$ is the unit-speed parameterization of $\gamma_{xy}$.
Let $q :=\gamma_v(m)$ be the (unique) point in the intersection of $[\bx,\by]$ and the geodesic extending $\gamma_{\bz \bp}$. 
Applying a suitable Lorentz transformation we may assume that $x=x',y=y'$ and so $v=v'$ and $a=a'$. We let $a$ vary on the geodesic $\gamma_v$ between the parameters $0$ and $m$. The point $a$ is given by $a=\gamma_v(t)=\cosh(t)x+\sinh(t)v=a'$.
Then define $f: [0,m] \rar \R$, 
\begin{equation}
f(t):=\cosh(\tau(a,p)+\tau(p,z))-\cosh(\tau(a,z'),
\end{equation}
where we applied the monotone function $\cosh$ to make computation easier. As in the Minkowski case we know $f(0)=0$ and $f(m)\geq 0$. But we will show $f''(t)-f(t) \leq 0$ instead of $f''(t)\leq 0$, which suffices to infer $f(t) \geq 0$ by Lemma \ref{help}. 
Note that in de Sitter space we have $\tau(x,y)=\arcosh(\langle x,y \rangle)$ for $x \ll y$, cf.\ \cite{CHKR17}.
Then with the help of addition theorems for $\cosh$, the function $f$ simplifies to
\begin{equation}
f(t)=\langle a,p \rangle \langle p,z \rangle + \sinh(\arcosh(\langle a,p \rangle))\sinh(\arcosh(\langle p,z \rangle)) - \langle a,z' \rangle.
\end{equation}
Using the relation $\sinh(\arcosh(x))=\sqrt{x^2-1}$, we further obtain
\begin{equation}
f(t)= 
\langle a,p \rangle \langle p,z \rangle + 
\sqrt{\langle a,p \rangle^2-1}\sqrt{\langle p,z \rangle^2-1} - 
\langle a,z' \rangle.
\end{equation}
Now we take the first derivative. Note that of the three points appearing, only $a$ depends on $t$. For now, we will leave the derivative of $a$ not explicitly calculated. This will be simpler to do later on.
\begin{equation}
\frac{d}{dt}f(t)= 
\langle \frac{d}{dt} a,p \rangle \langle p,z \rangle + 
\sqrt{\langle p,z \rangle^2-1} 
\frac{\langle a,p \rangle \langle \frac{d}{dt} a,p \rangle}
{\sqrt{\langle a,p \rangle^2-1}} - 
\langle \frac{d}{dt} a,z' \rangle
\end{equation}
For the second derivative, note that 
$(\frac{d}{dt})^2a=a$. We calculate and simplify
\begin{equation}
\left( \frac{d}{dt} \right) ^2f(t) = 
\langle a,p \rangle \langle p,z \rangle + 
\sqrt{\langle p,z \rangle^2-1}
\left(
\frac{\langle \frac{d}{dt}a,p \rangle^2+\langle a,p \rangle^2}{\sqrt{\langle a,p \rangle^2-1}} -
\frac{\langle a,p \rangle^2 \langle \frac{d}{dt} a,p \rangle^2}{\sqrt{\langle a,p \rangle^2-1}^3}
\right) - 
\langle a,z' \rangle.
\end{equation}
Now we observe that this very much resembles the original function $f$. Using this, we can write
\begin{scriptsize}
\begin{equation}
\left( \frac{d}{dt} \right) ^2f(t) - f(t)=  -
\sqrt{\langle a,p \rangle^2-1}\sqrt{\langle p,z \rangle^2-1} +
\sqrt{\langle p,z \rangle^2-1}
\left(
\frac{\langle \frac{d}{dt}a,p \rangle^2+\langle a,p \rangle^2}{\sqrt{\langle a,p \rangle^2-1}} -
\frac{\langle a,p \rangle^2 \langle \frac{d}{dt} a,p \rangle^2}{\sqrt{\langle a,p \rangle^2-1}^3}
\right).
\end{equation}
\end{scriptsize}
Then by Lemma \ref{help} it suffices to show that the right hand side is lesser than or equal to 0. 
We further simplify this expression and observe that it is equivalent to
\begin{equation}
\langle a,p \rangle^2-1 -
\langle \frac{d}{dt}a,p \rangle^2 
\leq 0.
\end{equation}
Now we insert $a=\cosh(t)x+\sinh(t)v$ and $\frac{d}{dt}a=\sinh(t)x+\cosh(t)v$ into this inequality and at the same time use the bi-linearity of the scalar product. In this way we get
\begin{equation}
(\cosh(t)\langle x,p \rangle+\sinh(t)\langle v,p \rangle)^2 -1 -
(\sinh(t)\langle x,p \rangle+\cosh(t)\langle v,p \rangle)^2 
\leq 0.
\end{equation}
After computing the big square, we factor appropriately and use $\cosh(t)^2-\sinh(t)^2=1$ to obtain
\begin{equation}
\langle x,p \rangle^2 - 1 
\leq \langle v,p\rangle^2.
\end{equation}
Since $\tau(x,p)=\arcosh(\langle x,p \rangle)$, we have $\langle x,p \rangle=\cosh(\tau(x,p))$ and so $\cosh(\tau(x,p))^2 - 1=\sinh(\tau(x,p))^2$. Thus, after taking the root on both sides we finally end up with
\begin{equation}
\sinh(\tau(x,p)) \leq |\langle v,p\rangle|.
\end{equation}
We now try to give these reformulations some intuitive value. 
First note that since $v$ by definition is perpendicular to $x$, we can write $|\langle v,p\rangle|=|\langle v,p-x\rangle|$. This gives this scalar product a bit more meaning, because $v$ is a (unit) tangent vector at $x$ while $p$ is a point in de Sitter space. Now both $p-x$ and $v$ can be regarded as vectors in the ambient Minkowski space starting at $x$. 
In particular, they are timelike (since $x \ll p$ and $x \ll y$) and so we can use the hyperbolic angle: 
\begin{equation}
| \langle v,p-x \rangle|=-\langle v,p-x\rangle = \| v \| \| p-x \| \cosh(\mad_x(v,p-x))=\| p-x \| \cosh(\mad_x(v,p-x)).
\end{equation}
We now try to minimize this expression while letting $v$ vary among unit tangent vectors. 
In this way, only $\cosh(\mad_x(v,p-x))$ changes, and it is minimal if and only if $\mad_x(v,p-x)$ is minimal. This angle is minimal precisely if $v$ is the unit vector in the direction of the projection of $p-x$ onto the tangent space. 
Since $p$ is in a normal neighbourhood of $x$, we can write $p=\cosh(s)x+\sinh(s)w$ for some appropriate unit vector $w$ and $s=\tau(x,p)$. Then $p-x=(\cosh(s)-1)x+\sinh(s)w$. We can consider this as an orthogonal decomposition with respect to $T_x \dS$: $(\cosh(s)-1)x$ is the orthogonal component and $\sinh(s)w$ is the parallel component. Then the angle is minimal if $v=w$. In particular, we then have 
\begin{align*}
\langle p-x,v \rangle & = \langle (\cosh(s)-1)x+\sinh(s)v,v \rangle= (\cosh(s)-1)\langle x,v\rangle+\sinh(s)\langle v,v \rangle \\
& =-\sinh(s) =-\sinh(\tau(x,p)).
\end{align*}
Hence $-\langle p-x,v \rangle=\sinh(\tau(x,p))$ which is precisely the left hand side. So if $v$ is any other (unit) vector than the projection of $x-p$ onto $T_x\dS$ then $|\langle v,p \rangle|$ would be even larger. Thus, the inequality holds and Lemma \ref{help} implies $f(t) \geq 0$, which in turn implies $\cosh(\tau(a,p)+\tau(p,z)) \geq \cosh(\tau(a,z')$. Since $\cosh$ is monotonously increasing, we finally obtain $\tau(a,p)+\tau(p,z)\geq \tau(a,z')$. \\

\underline{Anti-de Sitter space ($K=-1$)}: This case is very similar to the case of positive curvature. In fact, one can mostly copy the assumptions and computations from the de Sitter case. This time we define $f: [0,m] \rar \R$, 
\begin{equation}
f(t)=\cos(\tau(a,p)+\tau(p,z))-\cos(\tau(a,z').
\end{equation}
We omit the calculations for this case. Keep in mind that as $\cos$ is decreasing, we have to reverse some inequalities. \\

Clearly, the last two cases also work for any other $K \neq 0$, as one only has to respect additional scaling constants. \\

\underline{Letting $\bb$ vary}:
As announced before, we do not only want to cover the case $\bb=\bz$ but rather let $\bb$ vary. First suppose $\bb \in [\bp, \bz]$. We consider a function $f_b: [0,m] \rar \R$, 
\begin{equation}
f_b(t) := g(\btau(\ba,\bp)+\btau(\bp,\bb)) - g(\btau(\ba',\bb')), 
\end{equation}
where $g$ is, depending on $K$, one of the three monotone functions we introduced for easier computations.
In the above calculations concerning the second derivative, $\bz$ (and $\bz'$) did not really play an important role, so we could have easily exchanged this point for $\bb$ (or $\bb'$). In particular, we thus have $f_b(t) \geq 0$ if we can show that $f_b$ is nonnegative at the boundary of $[0,m]$. 
By definition we have $f_b(0)=0$. 
For $f_b(m)$, i.e., $\ba=\bq$, note that 
\begin{equation}
\btau(\bq,\bp)+\btau(\bp,\bb)+\btau(\bb,\bz)=\btau(\bq,\bz)\geq \btau(\bq',\bz') \geq \btau(\bq',\bb')+\btau(\bb',\bz').
\end{equation}
Since we have $\btau(\bb,\bz)=\btau(\bb', \bz')$ by construction, $\btau(\bq,\bp)+\btau(\bp,\bb) \geq \btau(\bq',\bb')$ follows and so $f_b(m) \geq 0$. \\

Now assume $\bb \in [\by,\bz]$. In contrast to $\bb \in [\bp, \bz]$, we have to restrict the position of $\ba$ even further since it may otherwise happen that the connecting geodesic $[\ba,\bb]$ stays inside the triangles. 
The values of $\ba$ we have to check actually depend on the position of $\bb$ in the following way: denote by $\br=:\gamma_v(n)$ the point in the intersection of $[\bx,\by]$ and the extension of $[\bp,\bb]$, where $\gamma_v$ is the unique geodesic from $x$ to $y$. 
In particular, $n \leq m$ with equality if and only if $\bb=\bz$. Then we consider $f_b:[0,n] \to \R$ as above. As for $f_b(n)$, note that we have 
\begin{equation}
\btau(\br,\bp)+\btau(\bp,\bb)-\btau(\br',\bb')=\btau(\br,\bb)-\btau(\br',\bb') \geq 0.
\end{equation}
In particular, $\btau(\br,\bp)+\btau(\bp,\bb) \geq \btau(\br',\bb')$. By applying $g$ and then bringing everything to one side we infer the claim $f_b(n) \geq 0$. 
Also, $f_b(0)=g(\btau(\bx,\bp)+\btau(\bp,\bb))-g(\btau(\bx',\bb'))$ is completely analogous to the case of $\ba$ and $\bz$ we explicitly calculated (with reversed time orientation).
Thus, we have finally shown all inequalities and so $T_3$ actually satisfies the same curvature bound from above. \\

As a last remark, we observe that every argument can be repeated if the timelike relation of $p$ and $y$ is reversed, i.e., if $y \ll p$. This is entirely analogous and basically differs only up to time orientation.
\end{proof}
For applications to a gluing theorem, we of course do not want to restrict the placement of $p$ to only the longest side in a triangle. As the proof of Lemma \ref{gluinglemma} basically only uses Lemma \ref{alexlem}, and there the placement of $p$ did not matter, we immediately obtain the following corollary.
\begin{cor}[Remaining constellations of the gluing lemma]
Let $X$ and $U$ be as in Lemma \ref{gluinglemma} and let $\Delta(x,y,z)$ be a timelike triangle in $U$. Let $p$ be a point on $[x,y]$ (or $[y,z]$) and consider the two resulting subtriangles that share the (timelike) segment $[p,z]$ (or $[x,p]$). If the subtriangles satisfy a timelike curvature bound from above, then so does the original triangle.
\end{cor}
\begin{proof}
In this case we mainly use Lemma \ref{alexlem2} and then observe that everything works out similarly as in the situation of the gluing lemma we just proved.
\end{proof}
Finally, we need a version of the gluing lemma that applies even when the shared segment in a Lorentzian pre-length space is not timelike. In general, this would not make any sense in an ordinary Lorentzian pre-length space as we have no concept of spacelike distance in such spaces. However, the main application of these lemmas is when the two spaces in the amalgamation are manifolds. In this case it does make sense to talk about a shared segment which is not timelike.
\begin{lem}[Gluing lemma for manifolds]
\label{mfgluinglemma}
Let $X$ and $U$ be as in Lemma \ref{gluinglemma}.
Let $X_1$ and $X_2$ be two strongly causal spacetimes and $U_i \subseteq X_i, (i=1,2)$ with the same assumptions as in Lemma \ref{gluinglemma} as well. 
Let $T_3:=\Delta(x,y,z)$ be a timelike triangle in $U \subseteq X$ and let $p \in [x,z]$. 
Assume that there exist geodesic triangles $T_1:=\Delta(x^1,p^1,y^1)$ and $T_2:=\Delta(p^2,y^2,z^2)$ in $U_1$ and $U_2$, respectively, such that
$\tau(x^1,p^1)=\tau(x,p), \tau(x^1,y^1)=\tau(x,y), \tau(p^2,z^2)=\tau(p,z), \tau(y^2,z^2)=\tau(y,z)$ and $|p^1y^1|_{\pm}=|p^2y^2|_{\pm}$.
If $T_1$ and $T_2$ satisfy curvature bounds above by $K$ in the sense of \cite{AB08}
\footnote{
In \cite{AB08}, a Lorentzian manifold is said to have (sectional) curvature bounded above by $K \in \R$ if spacelike sectional curvatures are $\leq K$ and timelike sectional curvatures are $\geq K$. We say a geodesic triangle satisfies such a curvature bound if $|ab|_{\pm} \leq |\ba \bb|_{\pm}$ for all points $a,b$ in the triangle and corresponding comparison points $\ba, \bb$ in the comparison triangle in $M_K$.
}, then $T_3$ has timelike curvature bounded above by $K$ in the sense of Definition \ref{tl curv bounds}.
\end{lem}
\begin{proof}
The proof of Lemma \ref{gluinglemma} can be adapted entirely into this setting. One easily observes that the fact that the shared segment is not timelike does not impact the comparison calculations.
\end{proof}
\begin{rem}[Gluing with reverse curvature bound]
In the metric case, the gluing theorem only works when considering spaces with an upper curvature bound. Roughly speaking, this is because in Alexandrov's lemma one angle always increases independently of whether one starts with a concave or convex quadrilateral. This in turn prevents the gluing lemma from working in the other direction.
In our case we encounter very similar behaviour: in Lemma \ref{alexlem}, the inequality $\angle xyz \geq \angle x'y'z'$ is independent of the general shape of the quadrilateral. Note that in the Lorentzian setting we want this inequality to point in the other direction compared to the two other angle estimates because of the sign of the nonnormalized angle.
\end{rem}
\begin{rem}[Hyperbolic angles in Lorentzian pre-length spaces]
The first author is currently working together with C.\ Sämann on adapting the concept of hyperbolic angles to the synthetic setting. The approach is similar to the Alexandrov angle in metric geometry, cf.\ [Definition I.1.12]\cite{BH99}. This may simplify some of the calculations in this section while also allowing nicer formulations of some statements.
\end{rem}

\section{The gluing theorem}
This section covers the proof of an analogue of the gluing theorem of Reshetnyak for CAT($k$) spaces. 
As mentioned before, a gluing theorem for Lorentzian pre-length spaces at present does not seem to be possible, at least not in a reasonably general formulation. Indeed, one would require that (locally) any two points in the identified set are timelike related, which basically forces each $A_i$ to be a timelike geodesic (or a ``discrete'' union of timelike geodesics).
The problem is essentially the missing concept of spacelike distance. Our proof idea very much follows the metric case, where one subdivides a triangle into smaller triangles and then applies Lemma \ref{gluinglemma}. 
In the Lorentzian setting, however, it may happen that such a division into $\emph{timelike}$ triangles is not always possible: consider two copies of $\R^3_1$ glued along a vertical plane. Then one can construct an arbitrarily small timelike triangle whose intersection with the plane consists of two points which are spacelike related. \\

Nevertheless, we can formulate the gluing theorem for Lorentzian manifolds viewed as Lorentzian pre-length spaces. In this setting we have access to spacelike distances and the triangle subdivision problem can be overcome. 
\subsection{Constructing comparison neighbourhoods}
One main obstacle in the proof of the gluing theorem is the construction of suitable comparison neighbourhoods. In particular, the comparison neighbourhoods of points in the glued set additionally have to be geodesically convex rather than just ensure the existence of maximizers between timelike related points. This is because when applying the gluing lemma in the spirit of Reshetnyak it might happen that the triangle subdivision occurs along a spacelike or null geodesic. Moreover, we also need them to be causally convex to ensure the existence of geodesics between the two spaces (obtained as a limit of converging curves). \\

Fortunately, in recent work by E.\ Minguzzi such neighbourhoods are constructed, cf.\ \cite{Min15}.
The basic idea is the following: 
let $M$ be a spacetime, $p \in M$ and $U \subseteq M$ a normal neighbourhood of $p$. Define the function $D_p^2:U \to \R$ by
\begin{equation}
D_p^2(q):=g_p(\exp_p^{-1}(q),\exp_p^{-1}(q)). 
\end{equation}
That is, the function $D_p^2$ is the (squared) signed distance from the point $p$ (but applying the sign after the squaring).
In the terminology of \cite{AB08}, $D^2_p(q)$ is called the energy $E_q(p)=E_p(q)$.
Let $\gamma$ be a timelike geodesic through $p$. Then from \cite[Lemma 4]{Min15} we infer the existence of $q_1, q_2$ in the image of 
$\gamma, q_1 \ll p \ll q_2$, and a normal neighbourhood $O$ of $p$ such that, 
setting $c_1:=D_{q_1}^2(p), c_2:= D_{q_2}^2(p)$ (both of which are negative), for sufficiently close $c_1'>c_1$ and $c_2'>c_2$ the set 
\begin{equation}
\label{lens1}
(D_{q_1}^2)^{-1}(-\infty,c_1') \cap (D_{q_2}^2)^{-1}(-\infty,c_2') \cap O
\end{equation}
is geodesically convex and globally hyperbolic. Moreover, in this proof the set $O$ is chosen in such a way that $\widebar{O} \subseteq I(q_1,q_2)$ and the (closure of the) desired connected component of $(D_{q_1}^2)^{-1}(-\infty,c_1') \cap (D_{q_2}^2)^{-1}(-\infty,c_2')$ is contained in $O$. Hence the set in (\ref{lens1}) is causally convex as well.
This set might be visually described as a ``lens'', i.e., the (bounded) region obtained by intersecting two hyperboloids in Minkowski space.
In our setup, this region can more easily be described by using the time separation function $\tau$ instead of $D^2_p$.
We will go into a bit more detail on how we adapt this construction in Remark \ref{lensrem1}.
\begin{rem}[Assumptions on the glued set]
\label{assumptions on A}
When viewing manifolds as Lorentzian length spaces, one has to be especially careful when dealing with submanifolds. This is because on a Lorentzian submanifold $A \subseteq M$ there are a priori two -- potentially strongly differing -- structures when viewing them as Lorentzian pre-length spaces. One the one hand, there is the induced substructure on $A$ as a Lorentzian submanifold which is very common in Lorentzian geometry. This leads to concepts such as relative causality relations, usually denoted by $\ll_A$ and $\leq_A$. In particular, a submanifold equipped with this induced structure is always a Lorentzian length space.
On the other hand, the restricted structure as a Lorentzian pre-length space just considers the restriction of the causality relations and the time separation function to the submanifold. In this case, $A$ is usually just a Lorentzian pre-length space and one will lose the original description of $\tau, \ll$ and $\leq$. For example, if $p,q \in A, p \leq q$, then clearly $p \mathrel{{\leq}|_A} q$. But if there is no causal curve from $p$ to $q$ that stays inside $A$, we do not get $p \leq_A q$. This is why the restricted Lorentzian structure finds little application from a relativistic or differential geometric viewpoint.
A subspace of a metric space is a common example where one uses the restricted structure. But also in this case, if one starts with a length space then an arbitrary subspace will only be a metric space and not a length space in general. \\

This observation is one of the reasons we decided to not require the identified sets to be submanifolds (the other of course is a pursuit of generality). We now summarize the properties we require of the identified sets.
Let $A_1$ and $A_2$ be two closed subsets of strongly causal spacetimes $M_1$ and $M_2$, respectively. Let $f: A_1 \to A_2$ be a locally bi-Lipschitz homeomorphism (we will formulate the gluing theorem without an artificial space $A$). Then we require $A_1, A_2$ and $f$ to be compatible in the following way:
\begin{itemize}
\item[(i)] $A_1$ and $A_2$ are non-timelike locally isolating.
\item[(ii)] $f$ is $\tau$-preserving and $\leq$-preserving.
\item[(iii)] $A_1$ is ``convex'' in the following sense: for all $p \in A_1$ there exists a normal neighbourhood $V_1 \subseteq M_1$ of $p$ such that the following holds: whenever $x,y \in U_1:= V_1 \cap A_1$ then the unique geodesic connecting them in $V_1$ is contained in $U_1$. The same holds for $A_2$.
\item[(iv)] $f$ locally preserves the signed distance: for all $p \in A_1$ there exists a normal neighbourhood $V_1 \subseteq M_1$ of $p$ such that the following holds: whenever $x,y \in U_1:= V_1 \cap A_1$ then $|xy|_{\pm}=|f(x)f(y)|_{\pm}$.
\end{itemize}
Observe that both (iii) and (iv) also hold for any smaller (convex) normal neighbourhood contained in $V$.
(iii) suggests that the set $A_i$ is at least contained in a totally geodesic submanifold. However, we still gain the possibility of admitting boundaries and even corners, see Figure \ref{A choices}.
\end{rem}
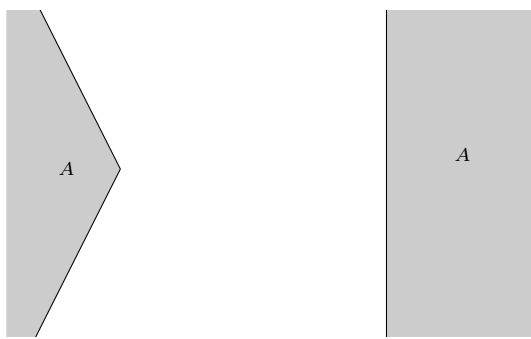
\begin{figure}
\begin{center}
\begin{tikzpicture}[line cap=round,line join=round,>=triangle 45,x=1cm,y=1cm]
\clip(-4.625244657618698,-2.217350640148503) rectangle (5.230867703908334,2.108464261205086);
\draw[fill=black!20] (-3,3) -- (-1.5,0) -- (-3,-3);

\fill[color=black!20] (2,-3) -- (2,3) -- (4,3) -- (4,-3);
\draw (2,-3) -- (2,3);
\draw (4,-3) -- (4,3);
\begin{scriptsize}
\coordinate [label=180: {$A$}] (A1) at (-2,0);
\coordinate [label=90: {$A$}] (A1) at (3,0);
\end{scriptsize}
\end{tikzpicture}
\end{center}
\caption{Two possible choices for $A$ in the Minkowski plane: a ``half-space'' with cornered boundary and a vertical strip.}
\label{A choices}
\end{figure}
This is also the reason why we do not exactly copy the construction mentioned in \cite[Lemma 4]{Min15}, because there the sets in (\ref{lens1}) have both ``control points'' on a geodesic through $p$. And if $p$ is a boundary point of $A$ there might not be a geodesic \emph{through} $p$ that stays inside $A$. 
\begin{rem}[Modified lenses]
\label{lensrem1}
Our modification of the sets in (\ref{lens1}) is very minor. 
Let $p \in A \subseteq M$ be as in Remark \ref{assumptions on A}. 
Choose normal coordinates $(\varphi, U)$ centered at $p$. In the chart neighbourhood $\varphi(U)$, let $B_S(0)$ be a Euclidean ball around $0=\varphi(p)$ such that $\varphi^{-1}(B_S(0))$ is convex (this is equivalent to saying that $B_S(0)$ is convex with respect to the push forward metric induced by $g$) and moreover such that for all smaller balls entirely contained in $B_S(0)$, their (inverse) image is convex as well.
By the non-local timelike isolation of $A$ we find $b_-,b_+ \in B_S(0) \cap \varphi(A)$ which are $\tau^{\eta}$-equidistant
\footnote{As the time separation is defined as the supremum of lengths of causal curves, it clearly is dependent on which metric is considered. We denote by $\tau^g$ the time separation measured with respect to the metric $g$.
In the same spirit, we can define the signed distance with respect to a certain metric, denoted by $| \cdot \cdot|_{\pm}^g$, which is the $g$-length of the unique $g$-geodesic connecting two points (of course the two points have to be in a normal neighbourhood with respect to $g$).} 
(and hence, as $p$ is the center of the normal coordinate chart, also $\tau^g$-equidistant) from $0$ such that $b_- \ll 0 \ll b_+$. 
Consider the timelike straight line segments $[b_-,0], [0,b_+]$ in $B_S(0)$ (which are actually contained in $\varphi(A)$ as well by the convexity of $A$ and moreover correspond to the geodesic segments in the manifold). These segments intersect at $0$ in an ordinary Minkowski hyperbolic angle of $\omega < \infty$ (which is also the hyperbolic angle measured in the manifold). 
Apply a Lorentz transformation if necessary to position these segments in the plane spanned by $\partial_0, \partial_1$ and such that they are symmetric (from a Euclidean point of view) with respect to the $\partial_1$ direction. Then the segment $[b_-,b_+]$ is parallel to the $\partial_0$-axis. 
If necessary, choose $S$ even smaller such that the ball is still convex after the Lorentz transformation.
Introduce a flat metric in the chart neighbourhood via
\begin{equation}
\label{flatmetric1}
\eta^+ := - (2dx_0)^2 + \sum_{i=1}^n (dx_i)^2.
\end{equation}
Clearly, $g_p =\eta < \eta^+$\footnote{For two Lorentzian metrics $g$ and $h$ on a manifold we say $h$ has wider lightcones than $g$, and write $g<h$, if for all points the timelike past/future with respect to $h$ contains the causal past/future with respect to $g$, cf.\ \cite{Min15}. In other sources, one may find the notation $g \prec h$ and the formulation $g(v,v) \leq 0 \Rightarrow h(v,v)<0$ for all $v \in TM$.} and by choosing $S$ smaller if necessary we can assume that $g < \eta^+$ holds on all of $B_S(0)$. 
Set $R:=\tau^{\eta^+}(b_-,b_+)$. 
Consider the following set:
\begin{equation}
\label{lenseq1}
L:=\{x \in B_S(0) \mid \tau^{\eta^+}(b_-,x) > \frac{R}{3}, \tau^{\eta^+}(x,b_+) > \frac{R}{3} \},
\end{equation}
which will be our replacement for the set in (\ref{lens1}).
Visually, this set may be described as a ``wide lens''.
Using the law of cosines, an elementary calculation yields $0 \in L$ independent of $\omega$ and $R$, see Figure \ref{lens nhoods}.
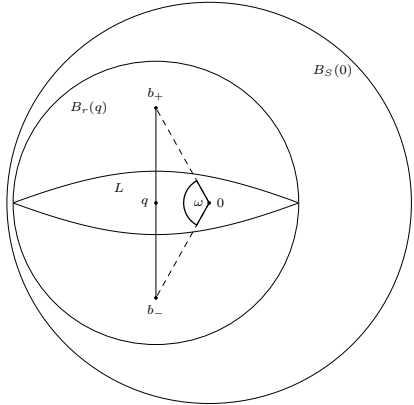
\begin{figure}
\begin{center}
\scalebox{0.7}{
\begin{tikzpicture}[line cap=round,line join=round,>=triangle 45,x=1cm,y=1cm]
\clip(-6.420459694803359,-4.268735865468322) rectangle (5.077546218513822,4.562449321445016);
\draw [shift={(0,0)},line width=0.7pt] (0,0) -- (119.05460409907714:0.479725964359997) arc (119.05460409907714:240.94539590092285:0.479725964359997) -- cycle;
\draw [line width=0.5pt] (-1,1.8)-- (-1,-1.8);
\draw [samples=50,domain=-0.446:0.446,rotate around={90:(-1,-1.8)},xshift=-1cm,yshift=-1.8cm,line width=0.5pt] plot ({1.2*(1+(\x)^2)/(1-(\x)^2)},{2.4*2*(\x)/(1-(\x)^2)});
\draw [samples=50,domain=-0.446:0.446,rotate around={90:(-1,1.8)},xshift=-1cm,yshift=1.8cm,line width=0.5pt] plot ({1.2*(-1-(\x)^2)/(1-(\x)^2)},{2.4*(-2)*(\x)/(1-(\x)^2)});
\draw [line width=0.5pt] (-1,0) circle (2.6832815729997477cm);
\draw [line width=0.5pt] (0,0) circle (3.8cm);
\draw [dashed] (0,0) -- (-1,1.8);
\draw [dashed] (0,0) -- (-1,-1.8);
\begin{scriptsize}
\coordinate [circle, fill=black, inner sep=0.7pt, label=0: {$0$}] (A1) at (0,0);
\coordinate [circle, fill=black, inner sep=0.7pt, label=90: {$b_+$}] (A1) at (-1,1.8);
\coordinate [circle, fill=black, inner sep=0.7pt, label=270: {$b_-$}] (A1) at (-1,-1.8);
\coordinate [circle, fill=black, inner sep=0.7pt, label=180: {$q$}] (A1) at (-1,0);
\coordinate [label=180: {$\omega$}] (A1) at (-0,0);
\coordinate [label=180: {$L$}] (A1) at (-1.5,0.3);
\coordinate [label=180: {$B_r(q)$}] (A1) at (-1.8,1.8);
\coordinate [label=180: {$B_S(0)$}] (A1) at (2.8,2.5);
\end{scriptsize}
\end{tikzpicture}
}
\end{center}
\caption{Constructing nice neighbourhoods in the spirit of \cite{Min15}.}
\label{lens nhoods}
\end{figure}
Say $b_-:=\gamma_1(t), b_+=\gamma_2(t)$, where $\gamma_1$ and $\gamma_2$ are the (unit speed) geodesics corresponding to the extensions of the straight line segments from above. By choosing $t$ smaller, i.e., by moving $b_-$ and $b_+$ closer to $0$ along these segments (but still keeping them equidistant to $0$), it easily follows that $R$ decreases as well. 
Moreover, $L$ has a ``width'' which is only dependent on $R$ and can be easily calculated. Denote by $q$ the midpoint of $[b_-,b_+]$, which due to the above applied Lorentz transformation is situated on the $\partial_1$-axis. Let $r$ be half the width of $L$. By choosing $t$ small enough, we can achieve that $B_r(q) \subseteq B_S(0)$. 
In particular, $L \subseteq B_r(q)$ anyways holds and $\varphi^{-1}(B_r(q))$ is geodesically convex. With arguments very similar to \cite[Theorem 10, Theorem 11 \& Lemma 4]{Min15} one can then show that $\varphi^{-1}(L)$ is geodesically convex, causally convex and globally hyperbolic.
\end{rem}

\subsection{Formulation and proof of the gluing theorem}
Now we can finally state and prove the Lorentzian analogue of the gluing theorem.
\begin{thm}[Reshetnyak's gluing theorem, Lorentzian version]
Let $(X_1,g_1)$ and $(X_2,g_2)$ be two smooth and strongly causal spacetimes with $\dim(X_1)=:n \geq m := \dim(X_2)$. Let $A_1$ and $A_2$ be two closed non-timelike locally isolating subsets of $X_1$ and $X_2$, respectively. 
Let $f:A_1 \rar A_2$ be a $\tau$-preserving and $\leq$-preserving locally bi-Lipschitz homeomorphism which locally preserves the signed distance. Suppose $A_1$ and $A_2$ are convex in the sense of Remark \ref{assumptions on A}(iii).
Suppose $X_1$ and $X_2$ have (sectional) curvature bounded above by $K \in \R$ in the sense of \cite{AB08}. Then the Lorentzian amalgamation $X:= X_1 \sqcup_A X_2$
\footnote{Technically, in this formulation the space $A$ does not exist. As it is still convenient in the proof to say when an equivalence class consists of two elements, we keep this notation. One can simply define $A:=\pi(A_1 \sqcup A_2)$.} 
is a Lorentzian pre-length space with timelike curvature bounded above by $K$. 
\end{thm}
\begin{proof}
By \cite[Proposition 3.5, Example 3.24, Theorem 3.26 \& Example 4.9]{KS18}, $X_1$ and $X_2$ are strongly causal regular (SR)-localizable Lorentzian length spaces with timelike curvature bounded above by $K$ in the sense of \cite{KS18}.
Let $[p] \in X$ and assume first $[p]=\{p^1\}$. 
We can choose comparison neighbourhoods in $X_1$ (and in $X_2$) to be (small enough) timelike diamonds, cf.\ \cite[Remark 2.2.12]{Ber20}. 
Since $A_1$ is closed, we find a neighbourhood $U_1 \subseteq X_1$ of $p^1$ which does not meet $A_1$. Since $X_1$ is strongly causal, there exists a timelike diamond $I_1(x^1,y^1)$ in $U_1$ containing $p^1$. In particular, this timelike diamond does not meet $A_1$ and hence $I_X([x],[y])=\pi(I_1(x^1,y^1))$ by Lemma \ref{amalgamated diamonds}.
Since $\ttau$ restricts to $\tau_1$ on $X_1$ by Proposition \ref{reptimesep}, it easily follows that $I_X([x],[y])$ is a comparison neighbourhood for $[p] \in X$. 
Thus, we only have to further investigate points in $A$. \\

Let $p^1 \in A_1$. 
Choose normal coordinates $(\varphi_1,V_1)$ around $p_1$ in $X_1$ 
such that $V_1$ is globally hyperbolic, convex in the sense of Remark \ref{assumptions on A}(iii) and $f$ preserves the signed distance on $V_1 \cap A_1 =:U_1$. We can further assume that $V_1$ is contained in a comparison neighbourhood and a causally closed neighbourhood of $p^1$ in $X_1$ and all small enough balls inside $V_1$ are geodesically convex as well.
Moreover, by choosing $V_1$ smaller if necessary, we can also assume that $f(U_1):=U_2$ is of the form $U_2=V_2 \cap A_2$ where $V_2$ has all the properties we imposed on $V_1$. In particular, $(\varphi_2,V_2)$ are normal coordinates around $p^2$ in $X_2$.
As $A_1$ is non-timelike locally isolating, we find $b_-^1,b_+^1 \in U_1$ which we choose $\tau^{g_1}$-equidistant from $p^1$ such that $b_-^1 \ll_1 p^1 \ll_1 b_+^1$. Since $f$ is $\tau$-preserving and hence $\ll$-preserving, we have $b_-^2,b_+^2 \in U_2$ as well as $b_-^2 \ll_2 p^2 \ll_2 b_+^2$ (and they are $\tau^{g_2}$-equidistant from $p^2$). 
By applying a Lorentz transformation, we can assume $\varphi_1(p^1)=0=\varphi_2(p^2), \varphi_1(b_-^1)=\varphi_2(b_-^2)=:b_-, \varphi_1(b_+^1)=\varphi_2(b_+^2)=:b_+$ and the straight line segments $[b_-,0]$ and $[0,b_+]$ lie in the plane spanned by $\partial_0,\partial_1$ and $b_-$ and $b_+$ are symmetric with respect to the $\partial_1$ direction (as in Remark \ref{lensrem1}).
Now consider a set as the one described in (\ref{lenseq1}) and recall the metric defined in (\ref{flatmetric1}). Set $R:= \tau^{\eta^+}(b_-,b_+)$. That is, define the set
\begin{equation}
\widetilde{L}:= \{ x \in \R^n \mid \tau^{\eta^+}(b_-,x) > \frac{R}{3}, \tau^{\eta^+}(x,b_+) > \frac{R}{3} \}
\end{equation}
and set $L_i:=\varphi_i^{-1}(\widetilde{L})$
\footnote{
If the dimension of $X_2$ is lower, one may want to view $\varphi_2$ as embedding $\R^m$ into $\R^n$, i.e., $\varphi_2 : V_2 \to \R^m \times \{0\} \subseteq \R^n$.
}. 
By the observations in Remark \ref{lensrem1}, we have $0 \in \widetilde{L}$ and hence $p^i \in L_i$.
In particular, by moving $b_-$ and $b_+$ closer to $0$ along these segments, we can achieve that 
both $L_1$ and $L_2$ have the desired properties of geodesic convexity, causal convexity and global hyperbolicity. 
While we can make no statement about the equality of $L_1$ and $L_2$ as they belong to different manifolds which we cannot relate as a whole, we indeed can say something about their intersections with $A_i$! Namely, we claim
$f(L_1 \cap A_1) = L_2 \cap A_2$. This fact is absolutely essential so that the neighbourhoods of $[p]$ coming from $X_1$ and $X_2$, respectively, are compatible.
To this end observe that as we are in normal coordinates around $p^i$ and $f$ is signed distance preserving, we have
\begin{equation}
|0 \varphi^1(a^1)|^{\eta}_{\pm} = 
|p^1a^1|^{g_1}_{\pm} = |p^2a^2|^{g_2}_{\pm} = 
|0 \varphi^2(a^2)|^{\eta}_{\pm}.
\end{equation}
This in turn yields an equality on the nonnormalized angles:
\begin{equation}
\angle b_- 0 \varphi_1(a^1) = \angle b_- 0 \varphi_2(a^2).
\end{equation}
Thus, the $\eta$-triangles $\Delta(b_-,0,\varphi_i(a^i)), i=1,2$ have two sides of equal length and an equal angle between these sides.
Then the law of cosines implies that the opposite side is equal as well, i.e., $|b_- \varphi^1(a^1)|^{\eta}_{\pm}=|b_- \varphi^2(a^2)|^{\eta}_{\pm}$.
An analogous argument gives
$|\varphi^1(a^1) b_+|^{\eta}_{\pm}=|\varphi^2(a^2) b_+|^{\eta}_{\pm}$.
Thus, $\varphi_1(a^1)$ and $\varphi_2(a^2)$ both lie on the intersection of two
lightcones and/or hyperboloids with respect to $\eta$. 
In any case this intersection is a Euclidean sphere of two dimensions less than the manifold (if the manifolds do not have the same dimension, it is two less than the lower dimensional manifold).
Moreover, the center of this sphere is located on $[b_-,b_+]$ and thus, as $[b_-,b_+]$ is parallel to $\partial_0$, all points on this sphere have the same time coordinate.
Because of the equality on the signed distance we obtain 
$\eta(b_- - a^1, b_- - a^1)=\eta(b_- - a^2, b_- - a^2)$. 
This, together with the fact that the first components of $a^1$ and $a^2$ are equal, implies 
$\eta^+(b_- - a^1, b_- - a^1)=\eta^+(b_- - a^2, b_- - a^2)$
and hence
$\tau^{\eta^+}(b_-,a^1) = \tau^{\eta^+}(b_-,a^2)$.
In particular, we have $\varphi_1(a^1) \in \widetilde{L}$ if and only if $\varphi_2(a^2) \in \widetilde{L}$ and so $a^1 \in L_1 \cap A_1$ if and only if $a^2 \in L_2 \cap A_2$.
We claim that $L:=\pi(L_1 \sqcup L_2)$ is a comparison neighbourhood for $[p] \in A$. \\

We first show that $L$ is timelike geodesic, i.e., for all $[x],[y] \in L$ with $[x] \tll [y]$ there exists a $\ttau$-realizing causal curve from $[x]$ to $[y]$ contained in $L$. 
To this end we claim the following: Let $\{x^1\}=[x] \in \pi(L_1) \subseteq L$ and $\{y^2\} = [y] \in \pi(L_2) \subseteq L$. Then $J_X([x],[y]) \cap A \subseteq L \cap A$. To see this, let $[q] \in J_X([x],[y]) \cap A$. Then $[q]=\{q^1,q^2\}$. 
Since $[x] \tleq [q]$, we find a chain of the form $x^1 \leq_1 a_1^1 \sim a_1^2 \leq_2 a_2^2 \sim a_2^1 \leq_1 \ldots \sim a_n^1 \leq_1 q^1$. Since $f$ is $\leq$-preserving, we obtain $x^1 \leq_1 q^1$ by the transitivity of $\leq_1$.
Since $g_i < \eta^+$, we get $\varphi_1(x^1) \ll^{\eta^+} \varphi_1(q^1)$. 
Clearly, $b_- \ll^{\eta^+} \varphi_1(x^1)$ as $x^1 \in L_1$. Then by the reverse triangle inequality we obtain $\tau^{\eta^+}(b_-,\varphi_1(q^1)) \geq \tau^{\eta^+}(b_-,x) > \frac{R}{3}$.
A similar argument implies $\tau^{\eta^+}(\varphi_2(q^2),b_+)\geq \tau^{\eta^+}(\varphi_2(y^2),b_+)>\frac{R}{3}$. 
By the above arguments, we obtain both inequalities for $\varphi_1(q^1)$ and $\varphi_2(q^2)$. Thus, these points are contained in $\widetilde{L}$ and so $q^i \in L_i \cap A_i$ and $[q] \in L \cap A$ follows. \\

Now let $[x],[y] \in L$ with $[x] \tll [y]$. If both belong to one space, we take the (projection of the) original geodesic connecting them. More precisely, if $x^i \in [x], y^i \in [y], i \in \{1,2\}$, then $\pi \circ \gamma_{x^i y^i}$ is a $\ttau$-realizing curve connecting $[x]$ and $[y]$ (since $\ttau$ restricts to $\tau_i$ on $X_i$). It is even contained in $\pi(L_i)$ as $L_i$ is causally convex.
So the only relevant case is (up to symmetry) $\{x^1\}=[x]$ and $\{y^2\} = [y]$ with $[x] \tll [y]$. By Proposition \ref{reptimesep} we find a sequence $([a_n])_{n \in \N}$ such that 
\begin{equation}
\lim_{n \to \infty} \tau_1(x^1,a_n^1) + \tau_2(a_n^2,y^2) = \ttau([x],[y]).
\end{equation}
By definition, we have $[a_n] \in J_X([x],[y]) \cap A$ for all $n$. By the above considerations, it then follows that $[a_n] \in L \cap A$ for all $n$. In particular, we find corresponding sequences $(a_n^1)_{n \in \N}$ in $L_1 \cap A_1$ and $(a_n^2)_{n \in \N}$ in $L_2 \cap A_2$. By construction, we have $L_1 \subseteq J_1(b_-^1,b_+^1)$ and $L_2 \subseteq J_2(b_-^2,b_+^2)$. 
Since $V_1$ and $V_2$ are globally hyperbolic and $A_1$ and $A_2$ are closed, we find that these sequences converge to some $a^1 \in A_1$ and $a^2 \in A_2$. In particular, $[a_n] \to [a] \in \overline{J_X([x],[y])} \cap A$.
As $[a_n] \in J_X([x],[y]) \cap A$, we have $x^1 \leq_1 a_n^1$ and $a_n^2 \leq_2 y^2$ for all $n$. Since $V_1$ and $V_2$ are contained in a causally closed neighbourhood, it follows that $x^1 \leq_1 a^1$ and $a^2 \leq_2 y^2$ and hence $[a] \in J_X([x],[y]) \cap A$.
In summary, we have
\begin{align*}
\ttau([x],[y]) & = 
\lim_{n \to \infty} \tau_1(x^1,a_n^1) + \tau_2(a_n^2,y^2) =
\lim_{n \to \infty} \tau_1(x^1,a_n^1) + \lim_{n \to \infty} \tau_2(a_n^2,y^2) \\
& = \tau_1(x^1,a^1) + \tau_2(a^2,y^2), 
\end{align*}
where the last equality follows since $\tau_1$ and $\tau_2$ are continuous on $L_1$ and $L_2$, respectively. Consider the two original geodesics from $x^1$ to $a^1$ and from $a^2$ to $y^2$ which are contained in $L_1$ and $L_2$, respectively, as these sets are causally convex. Then the concatenation of their projections is a $\ttau$-realizing curve from $[x]$ to $[y]$ contained in $L$. \\

The fact that $\ttau|_{L \times L}$ is finite and continuous easily follows from $L$ being timelike geodesic: if $x^i \in [x], y^i \in [y]$, then $\ttau([x],[y]) = \tau_i(x^i,y^i)$ by Proposition \ref{reptimesep}. By assumption, this is finite and $\tau_i$ is continuous. 
So again, the case we have to investigate further is (up to symmetry) $\{x^1\}=[x]$ and $\{y^2\} = [y]$. If $\ttau([x],[y])>0$, then by the above there exists $[a] \in A$ such that $\tau_1(x^1,a^1) + \tau_2(a^2,y^2) = \ttau([x],[y])$. The left hand side is finite by assumption. \\

As for the continuity of $\ttau$, note that $\ttau$ is lower semi-continuous by definition.
To see that $\ttau$ is upper semi-continuous in $([x],[y])$, let $[x_n] \to [x]$ and $[y_n] \to [y]$ be two sequences in $L$. We want to show $\ttau([x],[y]) \geq \limsup \ttau([x_n],[y_n])$. 
If $\limsup \ttau([x_n],[y_n]) = 0$, there is nothing to show. 
Otherwise (at least for a subsequence converging to the $\limsup$) we have $[x_n] \tll [y_n]$ for large enough $n$. 
In this case let $[a_n] \in A$ be such that $\tau_1(x_n^1,a_n^1)+\tau_2(a_n^2,y_n^2)=\ttau([x_n],[y_n])$, which exists since $L$ is timelike geodesic. 
Similar to the above arguments, it follows that $a_n^1 \in  L_1 \cap A_1$ and $a_n^2 \in L_2 \cap A_2$ (here we have $[a_n] \in J_X([x_n],[y_n])$, otherwise the argument is the same). By global hyperbolicity of $V_1$ and $V_2$ we infer the existence of a convergent subsequence. Without loss of generality the whole sequence converges, say $[a_n] \to [a]$. Then we compute
\begin{align*}
\ttau([x],[y]) & \geq \tau_1(x^1,a^1)+\tau_2(a^2,y^2) = 
\lim_{n \rar \infty} \tau_1(x_n^1,a_n^1) + \lim_{n \rar \infty} \tau_2(a_n^2,y_n^2) \\
& = \lim_{n \rar \infty} \tau_1(x_n^1,a_n^1) + \tau_2(a_n^2,y_n^2) = 
\lim_{n \rar \infty} \ttau([x_n],[y_n]), 
\end{align*}
where the first equality holds since $\tau_1$ and $\tau_2$ are (locally) continuous. \\

So it is only left to show the triangle comparison condition. 
To this end consider a timelike triangle $T_1:=\Delta([x],[y],[z])$ in $L$. 
Clearly, if all three points lie in a single space, then the triangle satisfies the curvature bound by assumption. So we only need to consider triangles passing through $A$, for which there are two possibilities. Either an endpoint ($[x]$ or $[z]$) of the triangle is isolated, i.e., one endpoint is in one space while the other two points are in the other space. Or the intermediate detour-point ($[y]$) is isolated, see Figure \ref{tr config}. \\

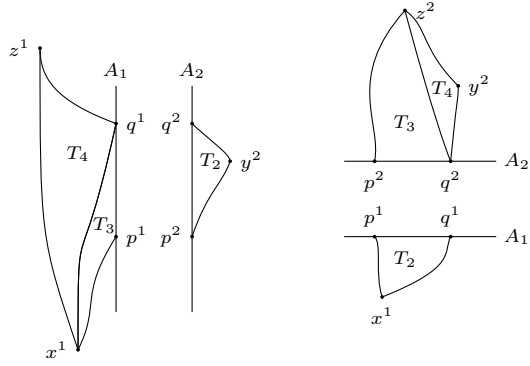
\begin{figure}
\begin{center}
\begin{tikzpicture}
\draw (0,0) -- (0,3);
\draw (1,0) -- (1,3);

\draw (-0.5,-0.5) .. controls (-0.2,0.1) and (-0.5,0.2) .. (0,1);
\draw (1,1) .. controls (1.2,1.6) and (1.4,1.7) .. (1.5,2);
\draw (1.5,2) .. controls (1.5,2.1) and (1.2,2.3) .. (1,2.5);
\draw (0,2.5) .. controls (-0.8,2.8) and (-1,3.1) .. (-1,3.5);
\draw (-0.5,-0.5) .. controls (-1,1) and (-1,1) .. (-1,3.5);

\draw (-0.5,-0.5) .. controls (-0.5,1.5) and (-0.5,0.2) .. (0,2.5);

\begin{scriptsize}
\coordinate [label=90: {$A_1$}] (A1) at (0,3);
\coordinate [label=90: {$A_2$}] (A1) at (1,3);

\coordinate [circle, fill=black, inner sep=0.5pt, label=180: {$x^1$}] (A1) at (-0.5,-0.5);
\coordinate [circle, fill=black, inner sep=0.5pt, label=0: {$p^1$}] (A1) at (0,1);
\coordinate [circle, fill=black, inner sep=0.5pt, label=180: {$p^2$}] (A1) at (1,1);
\coordinate [circle, fill=black, inner sep=0.5pt, label=0: {$q^1$}] (A1) at (0,2.5);
\coordinate [circle, fill=black, inner sep=0.5pt, label=180: {$q^2$}] (A1) at (1,2.5);

\coordinate [circle, fill=black, inner sep=0.5pt, label=0: {$y^2$}] (A1) at (1.5,2);
\coordinate [circle, fill=black, inner sep=0.5pt, label=180: {$z^1$}] (A1) at (-1,3.5);

\coordinate [label=270: {$T_4$}] (A1) at (-0.5,2.3);
\coordinate [label=170: {$T_2$}] (A1) at (1.5,1.8);
\coordinate [label=170: {$T_3$}] (A1) at (0.1,0.95);
\end{scriptsize}
\draw (3,1) -- (5,1);
\draw (3,2) -- (5,2);

\draw (3.5,0.2) .. controls (4.5,0.6) and (4.3,0.8) .. (4.4,1);
\draw (4.4,2) .. controls (4.5,3) and (4.5,2.8) .. (4.5,3);
\draw (3.5,0.2) .. controls (3.4,0.5) and (3.5,0.9) .. (3.4,1);
\draw (3.4,2) .. controls (3.5,2.5) and (3,3.1) .. (3.8,4);
\draw (4.5,3) .. controls (4,3.5) and (4.1,3.8) .. (3.8,4);
\draw (4.4,2) .. controls (4.2,2.5) and (3.85,3.8) .. (3.8,4);

\draw (-0.5,-0.5) .. controls (-0.5,1.5) and (-0.5,0.2) .. (0,2.5);

\begin{scriptsize}
\coordinate [label=0: {$A_1$}] (A1) at (5,1);
\coordinate [label=0: {$A_2$}] (A1) at (5,2);

\coordinate [circle, fill=black, inner sep=0.5pt, label=270: {$x^1$}] (A1) at (3.5,0.2);
\coordinate [circle, fill=black, inner sep=0.5pt, label=90: {$q^1$}] (A1) at (4.4,1);
\coordinate [circle, fill=black, inner sep=0.5pt, label=270: {$q^2$}] (A1) at (4.4,2);
\coordinate [circle, fill=black, inner sep=0.5pt, label=90: {$p^1$}] (A1) at (3.4,1);
\coordinate [circle, fill=black, inner sep=0.5pt, label=270: {$p^2$}] (A1) at (3.4,2);

\coordinate [circle, fill=black, inner sep=0.5pt, label=0: {$y^2$}] (A1) at (4.5,3);
\coordinate [circle, fill=black, inner sep=0.5pt, label=0: {$z^2$}] (A1) at (3.8,4);

\coordinate [label=90: {$T_2$}] (A1) at (3.8,0.5);
\coordinate [label=170: {$T_4$}] (A1) at (4.55,2.75);
\coordinate [label=90: {$T_3$}] (A1) at (3.8,2.3);
\end{scriptsize}
\end{tikzpicture}
\end{center}
\caption{The two cases of different triangle configurations.}
\label{tr config}
\end{figure}
The case where $[y]$ is isolated, say $[x],[z] \in X_1, [y] \in X_2$, is easily finished: indeed, choose arbitrary points $[p] \in [[x],[y]] \cap A, [q] \in [[y],[z]] \cap A$. Then $[p],[q] \in L \cap A$. 
Since $[x] \tll [p] \tll [y] \tll [q] \tll [z]$ and $L$ is timelike geodesic, we find a $\ttau$-realizing (causal) curve from $[p]$ to $[q]$ entirely contained in $L$.
We can divide $T_1$ into three smaller timelike triangles. Consider $T_2:=\Delta([p],[y],[q]), T_3:=\Delta([x],[p],[q])$ and $T_4:= \Delta([x],[q],[z])$. 
By the convexity of $A$ in the sense of Remark \ref{assumptions on A}(iii), we have $[[p],[q]] \subseteq L \cap A$. 
Thus, $T_2 \subseteq X_2$ and $T_3 \subseteq X_1$ and so both timelike triangles satisfy the curvature bound. Then by Lemma \ref{gluinglemma} also the bigger triangle $T_5:=\Delta([x],[y],[q])$ formed by $T_2$ and $T_3$ satisfies the same curvature bound. Applying the gluing lemma once more to the triangles $T_4$ and $T_5$ we obtain the desired curvature bound for the whole triangle $T_1$ and this case is finished. \\

Finally, we consider the case where an endpoint is isolated, say $[x] \in X_1, [y],[z] \in X_2$ (the case of $[z]$ being isolated is clearly symmetric). 
Choose points $[p] \in [[x],[z]] \cap A, [q] \in [[x],[y]] \cap A$. Then $[p],[q] \in L \cap A$. 
Since $L_1$ and $L_2$ are convex in the sense of Remark \ref{assumptions on A}(iii), there is a unique geodesic from $p^1$ to $q^1$ in $L_1 \cap A_1$ and a unique geodesic from $p^2$ to $q^2$ in $L_2 \cap A_2$, respectively. 
Moreover, these geodesics correspond to each other under $f$ and have the same signed length since $f$ preserves the signed distance. 
We denote this image in $X$ by $[[p],[q]]$ although it may not be a causal curve if $[p]$ and $[q]$ are not causally related.
By construction, we can identify the ``subtriangles'' $T_2:=\Delta([x],[p],[q])$ with $\Delta(x^1,p^1,q^1)$ and $T_3:=\Delta([p],[q],[z])$ with $\Delta(p^2,q^2,z^2)$.
In this way we can view $T_2$ and $T_3$ as triangles in the manifolds $X_1$ and $X_2$, respectively, independently of the causal character of $[[p],[q]]$. 
Thus, we can apply the manifold version of the gluing lemma, cf.\ Lemma \ref{mfgluinglemma}, to $T_2$ and $T_3$. In this way, we obtain that $T_5:= \Delta([x],[q],[z])$ satisfies the desired curvature bound. In particular, this is a valid timelike triangle in $X$.
At last, we apply Lemma \ref{gluinglemma} to the triangles $T_4 := \Delta([q],[y],[z])$ and $T_5$ and we finally obtain that the original triangle $T_1$ satisfies the curvature bound as well. Thus, the proof is completed.
\end{proof}
\begin{rem}[On regularity]
Clearly, one can formulate the above theorem for Lorentzian metrics of regularity $C^2$. As the techniques in \cite{Min15} work even in $C^{1,1}$, it is expected that this holds for the gluing theorem as well.
\end{rem}
\begin{chapt*}{Acknowledgments}
We want to thank Michael Kunzinger and Clemens Sämann for helpful discussions and comments. \\

This work was supported by research grant P33594 of the
Austrian Science Fund FWF.
\end{chapt*}

\bibliographystyle{alpha}
\bibliography{Biblio}

\begin{thebibliography}{CKHR17}

\bibitem[AB08]{AB08}
S.~B. Alexander and R.~L. Bishop.
\newblock Lorentz and semi-{R}iemannian spaces with {A}lexandrov curvature
  bounds.
\newblock {\em Communications in {A}nalysis and {G}eometry}, 16(2):251--282,
  2008.

\bibitem[ACS20]{ACS20}
L.~{Aké Hau}, A.J. {Cabrera Pacheco}, and D.A. Solis.
\newblock On the causal hierarchy of {L}orentzian length spaces.
\newblock {\em Classical Quantum Gravity}, 37(21):215013, 2020.

\bibitem[AGKS19]{AGKS19}
S.~B. Alexander, M.~Graf, M.~Kunzinger, and C.~S{ä}mann.
\newblock {Generalized cones as {L}orentzian length spaces: Causality,
  curvature, and singularity theorems}.
\newblock {\em Communicatios in {A}nalysis and {G}eometry}, 2019.
\newblock to appear, Preprint: \url{https://arxiv.org/pdf/1909.09575.pdf}.

\bibitem[AKP19]{AKP19}
S.~B. Alexander, V.~Kapovitch, and A.~Petrunin.
\newblock {\em An Invitation to {A}lexandrov Geometry: {CAT(0)} spaces}.
\newblock SpringerBriefs in Mathematics. Springer, Berlin, 2019.

\bibitem[BBI01]{BBI}
D.~Burago, Y.~Burago, and S.~Ivanov.
\newblock {\em A course in metric geometry}, volume~33 of {\em Graduate Studies
  in Mathematics}.
\newblock American Mathematical Society, Providence, {RI}, 2001.

\bibitem[Ber20]{Ber20}
T.~Beran.
\newblock Lorentzian length spaces.
\newblock Master's thesis, University of Vienna, 2020.
\newblock \url{https://phaidra.univie.ac.at/open/o:1363059}.

\bibitem[BFK98a]{BFK98a}
D.~Burago, S.~Ferleger, and A.~Kononenko.
\newblock A geometric approach to semi-dispersing billiards.
\newblock {\em Ergodic Theory and Dynamical Systems}, 18(2):303--319, 1998.

\bibitem[BFK98b]{BFK98b}
D.~Burago, S.~Ferleger, and A.~Kononenko.
\newblock Uniform estimates on the number of collisions in semi-dispersing
  billiards.
\newblock {\em Annals of Mathematics, Second Series}, 147(3):695--708, 1998.

\bibitem[BGH21]{BGH21}
A.~Burtscher and L.~García-Heveling.
\newblock Time functions on {L}orentzian length spaces.
\newblock \url{https://arxiv.org/pdf/2108.02693.pdf}, 2021.

\bibitem[BH99]{BH99}
M.~R. Bridson and A.~Haeflinger.
\newblock {\em Metric spaces of non-positive curvature}, volume 319 of {\em
  Comprehensive Studies in Mathematics}.
\newblock Springer, Berlin, 1999.

\bibitem[CKHR17]{CHKR17}
W.~J. Cunningham, D.~Krioukov, J.~Halverson, and D.~Rideout.
\newblock Exact geodesic distances in {F}{L}{R}{W} spacetimes.
\newblock {\em Physical Rreview D}, 96:103538, 2017.

\bibitem[CM20]{CM20}
F.~Cavalletti and A.~Mondino.
\newblock {Optimal transport in {L}orentzian synthetic spaces, synthetic
  timelike {R}icci curvature lower bounds and applications}.
\newblock \url{https://arxiv.org/pdf/2004.08934.pdf}, 2020.

\bibitem[GKS19]{GKS19}
J.~D.~E. Grant, M.~Kunzinger, and C.~S{ä}mann.
\newblock {Inextendibility of spacetimes and {L}orentzian length spaces}.
\newblock {\em Annals of Global Analysis and Geometry}, 55(1):133--147, 2019.

\bibitem[Kir18]{Kir18}
M.~Kirchberger.
\newblock Lorentzian comparison geometry.
\newblock Master's thesis, University of Vienna, 2018.
\newblock \url{https://phaidra.univie.ac.at/open/o:1351312}.

\bibitem[KS18]{KS18}
M.~Kunzinger and C.~S{ä}mann.
\newblock {Lorentzian length spaces}.
\newblock {\em Annals of Global Analysis and Geometry}, 54(3):399--447, 2018.

\bibitem[KS21]{KS21}
M.~Kunzinger and R.~Steinbauer.
\newblock Null distance and convergence of {L}orentzian length spaces.
\newblock \url{https://arxiv.org/pdf/2106.05393.pdf}, 2021.

\bibitem[Min15]{Min15}
E.~Minguzzi.
\newblock Convex neighbourhoods for {L}ipschitz connections and sprays.
\newblock {\em Monatshefte {M}athematik}, 177:569--625, 2015.

\bibitem[MS21]{MS21}
R.J. McCann and C.~S{\"a}mann.
\newblock A {L}orentzian analog for {H}ausdorff dimension and measure.
\newblock \url{https://arxiv.org/pdf/2110.04386.pdf}, 2021.

\bibitem[O'N83]{on83}
B.~O'Neill.
\newblock {\em Semi-Riemannian geometry with applications to relativity},
  volume 103 of {\em Pure and Applied Mathematics}.
\newblock Academic Press, Inc. [Harcourt Brace Jovanovich, Publishers], New
  York, 1983.

\bibitem[Rot20]{Rot20}
F.~Rott.
\newblock Reshetnyak's gluing theorem and its applications to billiards.
\newblock Master's thesis, University of Vienna, 2020.
\newblock \url{https://phaidra.univie.ac.at/open/o:1389853}.

\bibitem[SV16]{SV16}
C.~Sormani and C.~Vega.
\newblock Null distance on a spacetime.
\newblock {\em Classical {Q}uantum {G}ravity}, 33(8):085001, 2016.

\end{thebibliography}
\end{document}